\documentclass[11pt]{amsart}
\pdfoutput=1

\usepackage{amssymb}
\usepackage{stmaryrd}
\usepackage{microtype}
\usepackage{mathtools}
\usepackage{enumitem}
\usepackage{tikz-cd}

\usepackage{booktabs}
\usepackage{pgfplotstable}

\usepackage[margin=1in, marginparwidth=.75in]{geometry}

\usepackage[numbers,longnamesfirst]{natbib}

\usepackage{hyperref}
\hypersetup{
  pdftitle={HDG methods in finite element exterior calculus},
  pdfauthor={Calvin Reedy, Ari Stern, Yunhui Xue},
  pdfsubject={MSC 2020: 65N30},
  bookmarksopen=false
}

\usepackage[capitalize,nameinlink]{cleveref}

\DeclareMathSymbol{\star}{\mathord}{letters}{"3F}

\makeatletter
\g@addto@macro\bfseries{\boldmath}
\makeatother

\newcommand{\jump}[1]{\llbracket #1 \rrbracket}

\newcommand{\avg}[1]{\{\mkern-6.5mu\{ #1 \}\mkern-6.5mu\}}
\newcommand{\bigavg}[1]{\bigl\{\mkern-8mu\bigl\{ #1 \bigr\}\mkern-8mu\bigr\}}

\newcommand{\lVERT}{\lvert\mkern-2mu\lvert\mkern-2mu\lvert}
\newcommand{\rVERT}{\rvert\mkern-2mu\rvert\mkern-2mu\rvert}
\newcommand{\biglVERT}{\bigl\lvert\mkern-3mu\bigl\lvert\mkern-3mu\bigl\lvert}
\newcommand{\bigrVERT}{\bigr\rvert\mkern-3mu\bigr\rvert\mkern-3mu\bigr\rvert}

\newcommand{\ringhat}[1]{\mathchoice
  {\smash{\mathring{\widehat{#1}}}}%
  {\vphantom{{\mathring{\widehat{#1}}}}\smash{\mathring{\widehat{#1}}}}%
  {\vphantom{{\mathring{\widehat{#1}}}}\smash{\mathring{\widehat{#1}}}}%
  {\vphantom{{\mathring{\widehat{#1}}}}\smash{\mathring{\widehat{#1}}}}%
}

\theoremstyle{plain}
\newtheorem{theorem}             {Theorem}[section]
\newtheorem{lemma}      [theorem]{Lemma}
\newtheorem{corollary}  [theorem]{Corollary}
\newtheorem{proposition}[theorem]{Proposition}

\theoremstyle{definition}
\newtheorem{definition} [theorem]{Definition}
\newtheorem{assumption} [theorem]{Assumption}

\theoremstyle{remark}
\newtheorem{remark}     [theorem]{Remark}

\AddToHook{env/lemma/begin}{\crefalias{theorem}{lemma}}
\AddToHook{env/corollary/begin}{\crefalias{theorem}{corollary}}
\AddToHook{env/proposition/begin}{\crefalias{theorem}{proposition}}
\AddToHook{env/definition/begin}{\crefalias{theorem}{definition}}
\AddToHook{env/assumption/begin}{\crefalias{theorem}{assumption}}
\AddToHook{env/remark/begin}{\crefalias{theorem}{remark}}
\crefname{assumption}{Assumption}{Assumptions}

\title{HDG methods in finite element exterior calculus}

\author{Calvin Reedy}
\author{Ari Stern}
\author{Yunhui Xue}

\address{Department of Mathematics\\
  Washington University in St.~Louis}

\email{c.m.reedy@wustl.edu}
\email{stern@wustl.edu}
\email{marston@wustl.edu}

\begin{document}

\begin{abstract}
  We develop and analyze HDG methods for two central problems in
  finite element exterior calculus, the Hodge--Dirac problem and the
  Hodge--Laplace problem, in arbitrary dimension~$n$. Our analysis
  allows for equal-order polynomial spaces with $ \mathcal{O} (1) $
  penalty parameters, by contrast with previous work on the
  Hodge--Laplace problem requiring an underlying conforming complex
  and $ \mathcal{O} (h) $ or $ \mathcal{O} ( h ^{-1} ) $
  penalties. For the Hodge--Dirac problem with degree-$r$ polynomials,
  our error estimates give optimal order-$(r+1)$ convergence in all
  form degrees under suitable regularity hypotheses. For the
  Hodge--Laplace problem, we prove optimal order-$(r+1)$ convergence
  for the $k$-form solution and order-$(r+\frac{1}{2})$ convergence
  for the $ ( k \pm 1 )$-form auxiliary variables. We obtain improved
  order-$ (r + 1) $ auxiliary-variable estimates when $ k = 1 $ and/or
  $ k = n - 1 $, which in particular sharpens some recent HDG error
  estimates for the two-dimensional vector Poisson equation. Our
  analysis also encompasses cases of lower regularity, including
  solutions with reentrant-corner singularities on non-convex domains,
  and requires only mild mesh regularity conditions. The results are
  illustrated by numerical experiments in dimensions two and three.
\end{abstract}

\maketitle

\section{Introduction}

Finite element exterior calculus (FEEC) consolidates many families of
finite element methods for scalar and vector fields in $\mathbb{R}^2$
and $\mathbb{R}^3$, viewing them as special cases of methods for
differential forms in $\mathbb{R}^n$
\citep{ArFaWi2006,ArFaWi2010,Arnold2018}. Methods that had previously
been analyzed case-by-case, including primal Lagrange finite element
methods and the mixed methods of \citet{RaTh1977}, \citet{BrDoMa1985},
and \citet{Nedelec1980,Nedelec1986}, are thereby subject to a single
unifying analysis. However, most of the analysis conducted within the
FEEC framework has been of \emph{conforming} methods, including those
just mentioned.

It has been an open problem in FEEC to perform a similar unifying
analysis for hybridizable discontinuous Galerkin (HDG)
methods---specifically, an analysis that includes \emph{equal-order}
HDG methods, where the finite element spaces for all solution
components consist of discontinuous degree-$r$ polynomials. There has
been some progress on HDG methods in FEEC in recent years:
\begin{itemize}
\item \citet{AwFaGuSt2023} generalize the unified hybridization
  framework of \citet{CoGoLa2009} for the Hodge--Laplace problem in
  FEEC, including conforming, nonconforming, and HDG methods
  \citep[Section 8]{AwFaGuSt2023}. However, analysis is only performed
  for hybridized conforming methods, and there is no analysis of the
  HDG methods. \citet{StZa2026} present a more general hybridization
  framework for Hamiltonian PDEs in FEEC, including the Hodge--Dirac
  and Hodge--Laplace problems, and discuss structure-preserving
  properties of HDG methods, but again there is no analysis.

\item \citet{HoLiXu2022} analyze a family of extended Galerkin (XG)
  methods for the Hodge--Laplace problem, some of which coincide with
  HDG methods \citep[Section 6]{HoLiXu2022}. However, their analysis
  requires a conforming finite element complex and $ \mathcal{O} (h) $
  or $ \mathcal{O} ( h ^{-1} ) $ penalties to ensure that the XG
  solution is nearly a conforming FEEC solution for small $h$. Indeed,
  their ``analysis for XG methods is not applicable to more general
  finite element spaces, e.g., DG spaces of piecewise polynomials of
  equal degrees used in the literature \ldots most existing
  equal-order hybridized DG methods use $ \mathcal{O} (1) $
  stabilization parameters.'' \citep[Remark 4.7]{HoLiXu2022}

\item A very recent paper by \citet{CoNuSa2026} analyzes equal-order
  HDG methods with $ \mathcal{O} (1) $ penalties for the
  two-dimensional vector Laplacian, corresponding to the $1$-form
  Hodge Laplacian in $\mathbb{R}^2$. When using degree-$r$ polynomial
  finite elements, they obtain optimal order-$ ( r + 1 ) $ convergence
  for the vector solution but suboptimal order-$ ( r + \frac{1}{2} ) $
  error estimates for the ``auxiliary variables'' corresponding to the
  curl and divergence. These estimates are therefore not as sharp as
  they could be, in addition to being limited to dimension two. The
  analysis also requires a convex domain (in order to satisfy an
  elliptic regularity assumption \citep[Equation 5]{CoNuSa2026}) and a
  quasi-uniform family of meshes.
\end{itemize}

In this paper, we develop and analyze equal-order HDG methods for the
Hodge--Dirac and Hodge--Laplace problems in FEEC, in arbitrary
dimension $n$. When degree-$r$ polynomial finite elements and
$ \mathcal{O} (1) $ penalties are used, under suitable regularity
hypotheses, our error estimates for the Hodge--Dirac problem give
optimal order-$(r+1)$ convergence in all solution components. For the
$k$-form Hodge--Laplace problem, we prove optimal order-$( r + 1 )$
convergence in the $k$-form solution component and
order-$ ( r + \frac{1}{2} )$ convergence in the $ ( k \pm 1 )$-form
auxiliary variables; however, this improves to order-$(r+1)$
convergence in one or both of the auxiliary variables when $ k = 1 $
and/or $ k = n -1 $. In particular, for the $1$-form Hodge--Laplace
problem in $\mathbb{R}^2$, we get optimal order-$(r+1)$ convergence in
all variables, improving on the results of \citet{CoNuSa2026}. We also
obtain error estimates for lower-regularity cases, including
non-convex domains with reentrant corners. Meshes are assumed to be
shape-regular but not necessarily quasi-uniform. Our analytical
approach adapts and extends the seminal analysis of local DG methods
by \citet{CaCoPeSc2000} for the scalar Poisson problem.

The paper is organized as follows:
\begin{itemize}

\item \Cref{s:preliminaries} introduces the exterior calculus
  preliminaries we will use throughout the paper. We recall the
  Hodge--Dirac and Hodge--Laplace problems, along with the appropriate
  notions of boundary trace, jump, and average needed to formulate and
  analyze the HDG methods.

\item \Cref{s:dirac_methods} develops the HDG methods for the
  Hodge--Dirac problem, proves existence and uniqueness for both the
  local solvers and the global numerical solution, and develops an
  equivalent ``jump-average'' DG formulation used for the subsequent
  error analysis.

\item \Cref{s:dirac_error} proves \emph{a priori} error estimates for
  the Hodge--Dirac HDG methods, depending on domain regularity, mesh
  regularity, and Sobolev regularity of the exact solution. Our
  approach broadly follows the roadmap of the analysis in
  \citet{CaCoPeSc2000}.

\item \Cref{s:laplace} develops and analyzes the HDG methods for the
  Hodge--Laplace problem. The close relationship between the
  Hodge--Dirac and Hodge--Laplace problems allows us to reuse many of
  the previous results, particularly the approximation-theoretic
  estimates in \cref{s:KA_KB}.

\item \Cref{s:laplace_uhat} briefly discusses an alternative
  hybridization of the Hodge--Laplace HDG methods that gives the same
  numerical solutions but has potential advantages for
  implementation. In particular, static condensation yields a
  symmetric positive-definite Schur complement (modulo harmonic forms)
  rather than a saddle-point problem.

\item \cref{s:numerical_experiments} presents numerical experiments
  for the Hodge--Dirac and Hodge--Laplace problems in dimensions two
  and three, illustrating the convergence behavior for smooth
  solutions on convex domains as well as minimum-regularity solutions
  on an L-shaped domain.

\end{itemize}

\subsection*{Acknowledgments}

Ari Stern acknowledges the support of the National Science Foundation
(DMS-2208551) and the Simons Foundation (SFI-MPS-TSM-00014348).

\section{Preliminaries}
\label{s:preliminaries}

\subsection{Exterior calculus}

Let $\Omega \subset \mathbb{R}^n $ be a bounded Lipschitz domain, and
let $ \Lambda ^k (\Omega) $ denote the space of smooth differential
$k$-forms on $\Omega$. We recall the \emph{exterior derivative}
$ \mathrm{d} \colon \Lambda ^k (\Omega) \rightarrow \Lambda ^{ k + 1 }
(\Omega) $ and \emph{codifferential}
$ \delta \coloneqq ( - 1 ) ^k \star ^{-1} \mathrm{d} \star \colon
\Lambda ^k (\Omega) \rightarrow \Lambda ^{ k -1 } (\Omega) $, where
$ \star \colon \Lambda ^k (\Omega) \rightarrow \Lambda ^{ n - k }
(\Omega) $ is the Euclidean \emph{Hodge star} isomorphism. We will
often consider all form degrees simultaneously, writing
\begin{equation*}
  \Lambda (\Omega) \coloneqq \bigoplus _{ k = 0 } ^n \Lambda ^k (\Omega) ,
\end{equation*}
so that $ \mathrm{d} $, $ \delta $, and $ \star $ are operators on
$ \Lambda (\Omega) $. The \emph{Hodge--Dirac operator} on
$ \Lambda (\Omega) $ is $ \mathrm{D} \coloneqq \mathrm{d} + \delta $,
i.e.,
\begin{align*}
  ( \mathrm{D} w ) ^0 &= \delta w ^1 \\
  ( \mathrm{D} w ) ^k &= \mathrm{d} w ^{ k -1 } + \delta w ^{ k + 1 } \quad \text{for} \quad   k = 1, \ldots, n -1 ,\\
  ( \mathrm{D} w ) ^n &= \mathrm{d} w ^{ n -1 } .
\end{align*}
Since $ \mathrm{d} \mathrm{d} = 0 $ and $ \delta \delta = 0 $, the
square of the Hodge--Dirac operator is
$ \mathrm{D} ^2 = \mathrm{d} \delta + \delta \mathrm{d} \eqqcolon L $,
which is the \emph{Hodge--Laplace operator}. Note that $ \mathrm{D} $
is ungraded, mixing forms of different degrees, while $L$ takes
$k$-forms to $k$-forms.

In FEEC, we also encounter several Hilbert spaces of differential
forms, where $ \mathrm{d} $ and $\delta$ are taken in the sense of
distributions. Let $ L ^2 \Lambda (\Omega) $ be the Hilbert space of
differential forms on $\Omega$ with coefficients in $ L ^2 (\Omega) $,
denoting its inner product by $ ( \cdot , \cdot ) _\Omega $ and
associated norm by $ \lVert \cdot \rVert _\Omega $. Similarly, let
$ H ^s \Lambda (\Omega) $ be the space of differential forms with
coefficients in the Sobolev space $ H ^s (\Omega) $, denoting its norm
by $ \lVert \cdot \rVert _{ s, \Omega } $ and seminorm by
$ \lvert \cdot \rvert _{ s, \Omega } $. We also have the Hilbert spaces
\begin{align*}
  H \Lambda (\Omega) &\coloneqq \bigl\{ w \in L ^2 \Lambda (\Omega) : \mathrm{d} w \in L ^2 \Lambda (\Omega) \bigr\} ,\\
  \intertext{equipped with the graph inner product
  $ ( z, w ) _{ H \Lambda (\Omega) } \coloneqq ( z, w ) _\Omega + (
  \mathrm{d} z , \mathrm{d} w ) _\Omega $ and norm
  $ \lVert \cdot \rVert _{ H \Lambda (\Omega) } $, and}
  H ^\ast \Lambda (\Omega) &\coloneqq \bigl\{ w \in L ^2 \Lambda (\Omega) : \delta w \in L ^2 \Lambda (\Omega) \bigr\},
\end{align*}
equipped with the graph inner product
$ ( z, w ) _{ H ^\ast \Lambda (\Omega) } \coloneqq ( z, w ) _\Omega +
( \delta z , \delta w ) _\Omega $ and norm
$ \lVert \cdot \rVert _{ H ^\ast \Lambda (\Omega) } $. The
intersection $ H \Lambda (\Omega) \cap H ^\ast \Lambda (\Omega) $ is
equipped with the inner product
$ ( z, w ) _{ H \Lambda (\Omega) \cap H ^\ast \Lambda (\Omega) }
\coloneqq ( z, w ) _\Omega + ( \mathrm{d} z , \mathrm{d} w ) _\Omega +
( \delta z , \delta w ) _\Omega $ and the corresponding norm
$ \lVert \cdot \rVert _{ H \Lambda (\Omega) \cap H ^\ast \Lambda
  (\Omega) } $.

\begin{remark}
  \label{r:proxies}
  In $ \mathbb{R}^3 $, we have the following correspondence between
  exterior calculus spaces and operators and the familiar spaces and
  operators of vector calculus:
  \begin{equation*}
    \begin{tikzcd}
      \Lambda ^0 (\Omega) \ar[r, shift left, "\mathrm{d}"] \ar[d,
      phantom, sloped, "\cong"] & \Lambda ^1 (\Omega) \ar[l, shift left,
      "\delta"] \ar[r, shift left, "\mathrm{d}"] \ar[d, phantom, sloped,
      "\cong"] & \Lambda ^2 (\Omega) \ar[l, shift left, "\delta"] \ar[r,
      shift left, "\mathrm{d}"] \ar[d, phantom, sloped, "\cong"] &
      \Lambda ^3 (\Omega) \ar[l, shift left, "\delta"] \ar[d, phantom,
      sloped, "\cong"]
      \\
      C ^\infty (\Omega) \ar[r, shift left, "\operatorname{grad}"] & C
      ^\infty ( \Omega , \mathbb{R}^3 ) \ar[l, shift left,
      "-\operatorname{div}"] \ar[r, shift left, "\operatorname{curl}"] &
      C ^\infty ( \Omega , \mathbb{R}^3 ) \ar[l, shift left,
      "\operatorname{curl}"] \ar[r, shift left, "\operatorname{div}"] &
      C ^\infty ( \Omega ) \ar[l, shift left, "-\operatorname{grad}"] \mathrlap{\quad .}
    \end{tikzcd}
  \end{equation*}
  Note that $L$ corresponds to
  $ - \operatorname{div} \operatorname{grad} $ on scalar fields and
  $ \operatorname{curl} \operatorname{curl} - \operatorname{grad}
  \operatorname{div} $ on vector fields, which (modulo sign
  convention) are the usual scalar and vector Laplace operators. The
  Hilbert spaces $ H \Lambda ^k (\Omega) $ and
  $ H ^\ast \Lambda ^k (\Omega) $ may be identified with
  $ H ^1 (\Omega) $, $ H ( \operatorname{curl}; \Omega ) $,
  $ H ( \operatorname{div}; \Omega ) $, and $ L ^2 (\Omega) $.
\end{remark}

\subsection{Tangential and normal traces}

Assuming for the moment that the boundary $ \partial \Omega $ is
smooth, the trace map
$ \operatorname{tr} \colon \Lambda (\Omega) \rightarrow \Lambda (
\partial \Omega ) $ is pullback by the inclusion
$ \partial \Omega \hookrightarrow \Omega $. Given
$ w \in \Lambda (\Omega) $, the \emph{tangential trace}
$ w ^{\mathrm{tan}} \in \Lambda ( \partial \Omega ) $ and \emph{normal
  trace} $ w ^{\mathrm{nor}} \in \Lambda ( \partial \Omega ) $ are
defined by
\begin{equation*}
  w ^{\mathrm{tan}} \coloneqq \operatorname{tr} w, \qquad w ^{\mathrm{nor}} \coloneqq \widehat{ \star } ^{-1} \operatorname{tr} \star w ,
\end{equation*}
where $ \widehat{ \star } $ is the Hodge star on $ \partial \Omega $
with respect to the orientation induced by $ \Omega $ and metric
induced by $ \mathbb{R}^n $. Note that if
$ w \in \Lambda ^k (\Omega) $, then
$ w ^{\mathrm{tan}} \in \Lambda ^k ( \partial \Omega ) $ and
$ w ^{\mathrm{nor}} \in \Lambda ^{ k -1 } ( \partial \Omega ) $; in
particular, $n$-forms have vanishing tangential trace, while $0$-forms
have vanishing normal trace. The definitions above yield a
particularly elegant integration-by-parts formula for differential
forms: from Stokes's theorem and the Leibniz rule for $ \mathrm{d} $,
we have
\begin{equation}
  \label{e:ibp}
  \langle z ^{\mathrm{tan}} , w ^{\mathrm{nor}} \rangle _{ \partial \Omega } = ( \mathrm{d} z , w ) _\Omega - ( z , \delta w ) _\Omega ,
\end{equation}
where $ \langle \cdot , \cdot \rangle _{ \partial \Omega } $ denotes
the $ L ^2 $ inner product on $ \partial \Omega $;
cf.~\citep[Proposition 2.2]{AwFaGuSt2023}.

When $ \partial \Omega $ is merely Lipschitz, \citet{Weck2004} showed
that it is possible to define weak tangential and normal traces such
that \eqref{e:ibp} holds for all $ z \in H \Lambda (\Omega) $ and
$ w \in H ^\ast \Lambda (\Omega) $, where
$ \langle \cdot , \cdot \rangle _{ \partial \Omega } $ is a duality
pairing extending the $ L ^2 $ inner product. (See also
\citet{KuAu2012} for an excellent discussion of
\citeauthor{Weck2004}'s results with applications.) We may then define
the spaces
\begin{align*}
  \mathring{ H } \Lambda (\Omega) &\coloneqq \bigl\{ w \in H \Lambda (\Omega) : w ^{\mathrm{tan}} = 0 \bigr\} ,\\
  \mathring{ H } ^\ast \Lambda (\Omega) &\coloneqq \bigl\{ w \in H ^\ast \Lambda (\Omega) : w ^{\mathrm{nor}} = 0 \bigr\} ,
\end{align*}
which coincide with the closures of compactly supported smooth forms
in $ H \Lambda (\Omega) $ and $ H ^\ast \Lambda (\Omega) $,
respectively \citep[Theorems 5 and 7]{Weck2004}. Note that
$ \mathrm{d} \rvert _{ H \Lambda (\Omega) } $ and
$ \delta \rvert _{ \mathring{ H } ^\ast \Lambda (\Omega) } $ are
adjoints, as are
$ \mathrm{d} \rvert _{ \mathring{ H } \Lambda (\Omega) } $ and
$ \delta \rvert _{ H ^\ast \Lambda (\Omega) } $, since in either case
the boundary terms on the left-hand side of \eqref{e:ibp}
vanish. Consequently,
$ \mathrm{D} \rvert _{ H \Lambda (\Omega) \cap \mathring{ H } ^\ast
  \Lambda (\Omega) } $ and
$ \mathrm{D} \rvert _{ \mathring{ H }\Lambda (\Omega) \cap H ^\ast
  \Lambda (\Omega) } $ are self-adjoint.

\begin{remark}
  For $ \Omega \subset \mathbb{R}^3 $, these traces of differential
  forms correspond to ordinary traces of scalar fields and
  tangential/normal traces of vector fields (see \citet[Table
  1]{AwFaGuSt2023}), and \eqref{e:ibp} recovers the familiar
  integration-by-parts formulas of vector calculus. The weak trace
  theory of \citet{Weck2004} reduces to that of
  \citet{BuCi2001a,BuCi2001b,BuCoSh2002}.
\end{remark}

\subsection{The Hodge--Dirac and Hodge--Laplace problems}
\label{s:dirac_laplace}

We consider the Hodge--Dirac and Hodge--Laplace problems with
vanishing normal boundary conditions, both of which can be described
in terms of the self-adjoint restriction
$ \mathrm{D} \rvert _{ H \Lambda (\Omega) \cap \mathring{ H } ^\ast
  \Lambda (\Omega) } $. This operator has a nullspace
\begin{equation*}
  \mathfrak{H} (\Omega) \coloneqq \bigl\{ w \in H \Lambda (\Omega) \cap \mathring{ H } ^\ast \Lambda (\Omega) : \mathrm{D} w = 0 \bigr\} ,
\end{equation*}
called the space of \emph{harmonic} forms on $\Omega$. We also denote
the spaces of \emph{exact} and \emph{coexact} forms,
\begin{equation*}
  \mathfrak{B} (\Omega) \coloneqq \mathrm{d} H \Lambda (\Omega) , \qquad \mathring{ \mathfrak{B} } ^\ast (\Omega) \coloneqq \delta \mathring{ H } ^\ast \Lambda (\Omega) .
\end{equation*}

\begin{definition}
  Given $ f \in L ^2 \Lambda (\Omega) $, the \emph{Hodge--Dirac
    problem} is to find $ z \in \mathfrak{H} (\Omega) ^\perp $ and
  $ p \in \mathfrak{H} (\Omega) $ such that
  \begin{alignat*}{2}
    \mathrm{D} z + p &= f \quad &\text{in } &\Omega ,\\
    z ^{\mathrm{nor}} &= 0 \quad &\text{on } &\partial \Omega .
  \end{alignat*}
  The solution gives the \emph{Hodge decomposition}
  $ f = \mathrm{d} z + \delta z + p $. That is, it decomposes $f$ into
  the orthogonal sum of an exact form $ \mathrm{d} z $, a coexact form
  $ \delta z $, and a harmonic form $p$.
\end{definition}

The well-posedness of the Hodge--Dirac problem on Lipschitz domains
follows from a result of \citet{Picard1984} showing that
$ H \Lambda (\Omega) \cap \mathring{ H } ^\ast \Lambda (\Omega) $ is
compactly embedded in $ L ^2 \Lambda (\Omega) $. In particular, this
implies that
$ \mathrm{D} \rvert _{ H \Lambda (\Omega) \cap \mathring{ H } ^\ast
  \Lambda (\Omega) } $ has closed range, so by Banach's closed range
theorem we have
\begin{align*}
  L ^2 \Lambda (\Omega)
  &= \mathrm{D} \bigl( H \Lambda (\Omega) \cap \mathring{ H } ^\ast \Lambda (\Omega) \bigr) \oplus \mathfrak{H} (\Omega) \\
  &= \mathfrak{B} (\Omega) \oplus \mathring{ \mathfrak{B} } ^\ast (\Omega) \oplus \mathfrak{H} (\Omega) .
\end{align*}
Finding this decomposition is precisely the Hodge--Dirac
problem. Furthermore, there exists a constant $c$, depending only on
$\Omega$, such that for all
$ w \in H \Lambda (\Omega) \cap \mathring{ H } ^\ast \Lambda (\Omega)
\cap \mathfrak{H} (\Omega) ^\perp $,
\begin{equation*}
  \lVert w \rVert _\Omega \leq c \lVert \mathrm{D} w \rVert _\Omega ,
\end{equation*}
which \citet[Theorem 2.2]{ArFaWi2006} call \emph{Poincar\'e's
  inequality}, and which implies stability of the Hodge--Dirac
problem. For further discussion of these and other operator-theoretic
properties of the Hodge--Dirac operator on a Lipschitz domain, see
\citet{AxMc2004}.

We may also obtain Sobolev regularity and stability results depending
on the geometry of the domain $\Omega$. Following the terminology of
\citep[Section 7.7]{ArFaWi2006}, say that $\Omega$ is
\emph{$t$-regular} for some $ 0 < t \leq 1 $ if
$ H \Lambda (\Omega) \cap \mathring{ H } ^\ast \Lambda (\Omega) $ is
continuously embedded in $ H ^t \Lambda (\Omega) $, i.e., there exists
a constant $C$ such that for all
$ w \in H \Lambda (\Omega) \cap \mathring{ H } ^\ast \Lambda (\Omega)
$,
\begin{equation*}
  \lVert w \rVert _{ t, \Omega } \leq C \lVert w \rVert _{ H \Lambda (\Omega) \cap H ^\ast \Lambda (\Omega) } .
\end{equation*}
In particular, all Lipschitz domains are $ \frac{1}{2} $-regular
\citep[Theorem 11.2]{MiMiTa2001}, Lipschitz polyhedra in
$\mathbb{R}^3$ are $t$-regular for some $ t > \frac{1}{2} $
\citep{AmBeDaGi1998}, and bounded smooth domains
\citep{Gaffney1951,Friedrichs1955} as well as convex Lipschitz domains
\citep{Mitrea2001} are $1$-regular. This immediately implies that the
solution to the Hodge--Dirac problem satisfies
$ z \in H ^t \Lambda (\Omega) $ for any given
$ f \in L ^2 \Lambda (\Omega) $, along with a corresponding stability
estimate.

\begin{definition}
  Given $ f \in L ^2 \Lambda ^k (\Omega) $, the \emph{Hodge--Laplace
    problem} is to find $ u \in \mathfrak{H} ^k (\Omega) ^\perp $ and
  $ p \in \mathfrak{H} ^k (\Omega) $ such that
  \begin{alignat*}{2}
    L u + p &= f \quad &\text{in } &\Omega ,\\
    u ^{\mathrm{nor}} &= 0 \quad &\text{on } &\partial \Omega ,\\
    ( \mathrm{d} u ) ^{\mathrm{nor}} &= 0 \quad &\text{on } &\partial \Omega .
  \end{alignat*}
  The solution gives the Hodge decomposition
  $ f = \mathrm{d} \sigma + \delta \rho + p $, where
  $ \sigma = \delta u $ and $ \rho = \mathrm{d} u $. Equivalently, the
  Hodge--Laplace problem seeks $\sigma$, $u$, $\rho$, and $p$
  satisfying
  \begin{subequations}
    \label{e:laplace_first-order}
    \begin{alignat}{2}
      \delta u &= \sigma \quad &\text{in } &\Omega ,\\
      \mathrm{d} \sigma + \delta \rho + p &= f \quad &\text{in } &\Omega ,\\
      \mathrm{d} u &= \rho \quad &\text{in } &\Omega ,\\
      u ^{\mathrm{nor}} &= 0 \quad &\text{on } &\partial \Omega ,\\
      \rho ^{\mathrm{nor}} &= 0 \quad &\text{on } &\partial \Omega .
    \end{alignat}
  \end{subequations}
\end{definition}

Since $ L = \mathrm{D} ^2 $, solving the Hodge--Laplace problem
amounts to solving two Hodge--Dirac problems in succession: first
$ \mathrm{D} ( \sigma \oplus \rho ) + p = f $, then
$ \mathrm{D} u = \sigma \oplus \rho $. Well-posedness of the
Hodge--Laplace problem thus follows immediately from that for the
Hodge--Dirac problem; see the discussion in \citet[Section
4]{LeSt2016} and references therein. As with the Hodge--Dirac problem,
$t$-regularity of $\Omega$ implies $ H ^t \Lambda (\Omega) $
regularity and stability results for the Hodge--Laplace problem.

\begin{remark}
  For vanishing tangential boundary conditions, one would instead take
  the other self-adjoint restriction
  $ \mathrm{D} \rvert _{ \mathring{ H } \Lambda (\Omega) \cap H ^\ast
    \Lambda (\Omega) } $, noting that analogous compactness and
  $t$-regularity results also hold for
  $ \mathring{ H } \Lambda (\Omega) \cap H ^\ast \Lambda (\Omega)
  $. In this setting, the appropriate space of harmonic forms is
  \begin{equation*}
    \mathring{ \mathfrak{H} } (\Omega) \coloneqq \bigl\{ w \in \mathring{ H } \Lambda (\Omega) \cap H ^\ast \Lambda (\Omega) : \mathrm{D} w = 0 \bigr\} ,
  \end{equation*}
  the boundary conditions for the Hodge--Dirac problem are
  $ z ^{\mathrm{tan}} = 0 $, and the boundary conditions for the
  Hodge--Laplace problem are $ \sigma ^{\mathrm{tan}} = 0 $ and
  $ u ^{\mathrm{tan}} = 0 $.
\end{remark}

\subsection{Broken spaces, jumps, and averages}

Throughout this paper, we will encounter ``broken'' spaces of
differential forms with respect to a partition of $\Omega$ into
non-overlapping Lipschitz subdomains $ K \in \mathcal{T} _h $ (e.g., a
simplicial triangulation), and we will consider traces, jumps, and
averages on the skeleton
$ \partial \mathcal{T} _h \coloneqq \bigsqcup _{ K \in \mathcal{T} _h
} \partial K $. Denote
\begin{equation*}
  ( \cdot , \cdot ) _{ \mathcal{T} _h } \coloneqq \sum _{ K \in \mathcal{T} _h } ( \cdot , \cdot ) _K , \qquad \langle \cdot , \cdot \rangle _{ \partial \mathcal{T} _h } \coloneqq \sum _{ K \in \mathcal{T} _h } \langle \cdot , \cdot \rangle _{ \partial K } .
\end{equation*}
If $ z \in \prod _{ K \in \mathcal{T} _h } H \Lambda (K) $ and
$ w \in \prod _{ K \in \mathcal{T} _h } H ^\ast \Lambda (K) $,
integrating by parts on each $K$ and summing gives
\begin{equation}
  \label{e:ibp_broken}
  \langle z ^{\mathrm{tan}} , w ^{\mathrm{nor}} \rangle _{ \partial \mathcal{T} _h } = ( \mathrm{d} z , w ) _{ \mathcal{T} _h } - ( z, \delta w ) _{ \mathcal{T} _h }.
\end{equation}
In the situations where we will apply \eqref{e:ibp_broken}, we will
generally have
$ z, w \in \prod _{ K \in \mathcal{T} _h } H ^t \Lambda (K) $ with
$ t > \frac{1}{2} $, so the trace theorem implies
$ z ^{\mathrm{tan}} , w ^{\mathrm{nor}} \in L ^2 \Lambda ( \partial
\mathcal{T} _h ) \coloneqq \prod _{ K \in \mathcal{T} _h } L ^2
\Lambda ( \partial K ) $, and
$ \langle \cdot , \cdot \rangle _{ \partial \mathcal{T} _h } $ is
genuinely the $ L ^2 $ inner product rather than a weak duality
pairing.

Since $ \partial \mathcal{T} _h $ is a \emph{disjoint} union of the
subdomain boundaries, each interface
$ e = \partial K ^+ \cap \partial K ^- $ between adjacent
$ K ^\pm \in \mathcal{T} _h $ appears twice: once as
$ e ^+ \subset \partial K ^+ $, once as
$ e ^- \subset \partial K ^- $, with opposite orientations. Hence,
traces in $ L ^2 \Lambda ( \partial \mathcal{T} _h ) $ are generally
\emph{double-valued}, taking independent values on $ e ^\pm $. The
conforming spaces $ H \Lambda (\Omega) $ and
$ H ^\ast \Lambda (\Omega) $ are characterized by having
\emph{single-valued} tangential and normal traces, respectively---but
each type of trace has a distinct notion of ``single-valued'':
\begin{itemize}
\item If $ w \in H \Lambda (\Omega) $, then
  $ w ^{\mathrm{tan}} \rvert _{ e ^+ } = w ^{\mathrm{tan}} \rvert _{ e
    ^- } $, since the tangential trace is orientation-independent.

\item If $ w \in H ^\ast \Lambda (\Omega) $, then
  $ w ^{\mathrm{nor}} \rvert _{ e ^+ } = - w ^{\mathrm{nor}} \rvert _{
    e ^- } $, since $ \widehat{ \star } $ flips sign depending on
  orientation.
\end{itemize}
See the discussion in \citet[Section 3.2]{StZa2026}. This observation
about orientation dependence motivates the following definitions;
cf.~\citep[Definition 3.2]{StZa2026}.

\begin{definition}
  \label{d:jump-avg}
  Given
  $ \widehat{ w } ^{\mathrm{tan}} \in L ^2 \Lambda ( \partial
  \mathcal{T} _h ) $, define the \emph{tangential jump}
  $ \jump{ \widehat{ w } ^{\mathrm{tan}} } \in L ^2 \Lambda ( \partial
  \mathcal{T} _h ) $ by
  \begin{alignat*}{2}
    \jump{ \widehat{ w } ^{\mathrm{tan}} } _{ e ^\pm } &\coloneqq \frac{1}{2} \bigl(  \widehat{ w } ^{\mathrm{tan}} \rvert _{ e ^\pm } - \widehat{ w } ^{\mathrm{tan}} \rvert _{ e ^\mp } \bigr) , \quad & e &\not\subset \partial \Omega, \\
    \jump{ \widehat{ w } ^{\mathrm{tan}} } _{e \phantom{^\pm}} &\coloneqq 0 , \quad & e &\subset \partial \Omega ,\\
    \intertext{and the \emph{tangential average} $ \avg{ \widehat{ w } ^{\mathrm{tan}} } \in L ^2 \Lambda ( \partial \mathcal{T} _h ) $ by}
    \avg{ \widehat{ w } ^{\mathrm{tan}} } _{ e ^\pm } &\coloneqq \frac{1}{2} \bigl(  \widehat{ w } ^{\mathrm{tan}} \rvert _{ e ^+ } + \widehat{ w } ^{\mathrm{tan}} \rvert _{ e ^- } \bigr) , \quad & e &\not\subset \partial \Omega, \\
    \avg{ \widehat{ w } ^{\mathrm{tan}} } _{e \phantom{^\pm}} &\coloneqq \widehat{ w } ^{\mathrm{tan}} \rvert _e , \quad & e &\subset \partial \Omega .\\
    \intertext{Similarly, given $ \widehat{ w } ^{\mathrm{nor}} \in L ^2 \Lambda ( \partial \mathcal{T} _h ) $, define the \emph{normal jump} $ \jump{ \widehat{ w } ^{\mathrm{nor}} } \in L ^2 \Lambda ( \partial \mathcal{T} _h ) $ by}
    \jump{ \widehat{ w } ^{\mathrm{nor}} } _{ e ^\pm } &\coloneqq \frac{1}{2} \bigl(  \widehat{ w } ^{\mathrm{nor}} \rvert _{ e ^+ } + \widehat{ w } ^{\mathrm{nor}} \rvert _{ e ^- } \bigr) , \quad & e &\not\subset \partial \Omega, \\
    \jump{ \widehat{ w } ^{\mathrm{nor}} } _{e \phantom{^\pm}} &\coloneqq \widehat{ w } ^{\mathrm{nor}} \rvert _e , \quad & e &\subset \partial \Omega ,\\
    \intertext{and the \emph{normal average} $ \avg{ \widehat{ w } ^{\mathrm{nor}} } \in L ^2 \Lambda ( \partial \mathcal{T} _h ) $ by}
    \avg{ \widehat{ w } ^{\mathrm{nor}} } _{ e ^\pm } &\coloneqq \frac{1}{2} \bigl(  \widehat{ w } ^{\mathrm{nor}} \rvert _{ e ^\pm } - \widehat{ w } ^{\mathrm{nor}} \rvert _{ e ^\mp } \bigr) , \quad & e &\not\subset \partial \Omega, \\
    \avg{ \widehat{ w } ^{\mathrm{nor}} } _{e \phantom{^\pm}} &\coloneqq 0, \quad & e &\subset \partial \Omega .
  \end{alignat*}
  A tangential or normal trace is \emph{single-valued} if its jump
  vanishes on $ \partial \mathcal{T} _h \setminus \partial \Omega $.
\end{definition}

\begin{remark}
  While the factor of $ \frac{1}{2} $ in the jump is non-standard,
  including it simplifies many calculations and formulas. As usual,
  there is some flexibility in defining jump and average on
  $ \partial \Omega $, and we make the choice above since it is
  convenient for problems with normal boundary conditions:
  $ w \in H \Lambda (\Omega) \cap \mathring{ H } ^\ast \Lambda
  (\Omega) $ says precisely that $ \jump{ w ^{\mathrm{tan}} } = 0 $
  and $ \jump{ w ^{\mathrm{nor}} } = 0 $; cf.~\citep[Proposition
  3.1]{AwFaGuSt2023}.  A different choice, for problems with
  tangential boundary conditions, appears in \citep[Definition
  3.2]{StZa2026}.
\end{remark}

\Cref{d:jump-avg} has the property that
$ \widehat{w} ^{\mathrm{tan}} = \avg{ \widehat{w} ^{\mathrm{tan}} } +
\jump{ \widehat{ w } ^{\mathrm{tan}} } $ and
$ \widehat{w} ^{\mathrm{nor}} = \avg{ \widehat{w} ^{\mathrm{nor}} } +
\jump{ w ^{\mathrm{nor}} } $, and furthermore, these decompositions
are $ L ^2 \Lambda ( \partial \mathcal{T} _h ) $-orthogonal. Indeed,
$ \avg{ \cdot } $ and $ \jump{ \cdot } $ are
$ L ^2 \Lambda ( \partial \mathcal{T} _h ) $-orthogonal projections
onto the two different types of single-valued traces (with boundary
conditions on $ \partial \Omega $ ensuring that the sum is
direct). This orthogonality is expressed in the following identities.

\begin{proposition}[{\citet[Proposition~3.4]{StZa2026}}]
  \label{p:jump-avg}
  For all
  $ \widehat{ z } ^{\mathrm{tan}} , \widehat{ z } ^{\mathrm{nor}} ,
  \widehat{ w } ^{\mathrm{tan}} , \widehat{ w } ^{\mathrm{nor}} \in L
  ^2 \Lambda ( \partial \mathcal{T} _h ) $,
  \begin{align*}
    \langle \widehat{ z } ^{\mathrm{tan}} , \widehat{ w } ^{\mathrm{nor}} \rangle _{ \partial \mathcal{T} _h } &= \bigl\langle \avg{ \widehat{ z } ^{\mathrm{tan}} }, \jump{ \widehat{ w } ^{\mathrm{nor}} } \bigr\rangle _{ \partial \mathcal{T} _h } + \bigl\langle \jump{ \widehat{ z } ^{\mathrm{tan}} }, \avg{ \widehat{ w } ^{\mathrm{nor}} } \bigr\rangle _{ \partial \mathcal{T} _h } , \\
    \langle \widehat{ z } ^{\mathrm{tan}} , \widehat{ w } ^{\mathrm{tan}} \rangle _{ \partial \mathcal{T} _h } &= \bigl\langle \avg{ \widehat{ z } ^{\mathrm{tan}} }, \avg{ \widehat{ w } ^{\mathrm{tan}} } \bigr\rangle _{ \partial \mathcal{T} _h } + \bigl\langle \jump{ \widehat{ z } ^{\mathrm{tan}} }, \jump{ \widehat{ w } ^{\mathrm{tan}} } \bigr\rangle _{ \partial \mathcal{T} _h } , \\
    \langle \widehat{ z } ^{\mathrm{nor}} , \widehat{ w } ^{\mathrm{nor}} \rangle _{ \partial \mathcal{T} _h } &= \bigl\langle \avg{ \widehat{ z } ^{\mathrm{nor}} }, \avg{ \widehat{ w } ^{\mathrm{nor}} } \bigr\rangle _{ \partial \mathcal{T} _h } + \bigl\langle \jump{ \widehat{ z } ^{\mathrm{nor}} }, \jump{ \widehat{ w } ^{\mathrm{nor}} } \bigr\rangle _{ \partial \mathcal{T} _h } .
  \end{align*}
\end{proposition}

The following simple corollary, which is just the Pythagorean theorem,
will be useful for bounding jumps and averages by ordinary traces.

\begin{corollary}
  \label{c:pythag}
  For all
  $ \widehat{ w } ^{\mathrm{tan}} , \widehat{ w } ^{\mathrm{nor}} \in
  L ^2 \Lambda ( \partial \mathcal{T} _h ) $,
  \begin{align*}
    \lVert \widehat{ w } ^{\mathrm{tan}} \rVert ^2 _{ \partial \mathcal{T} _h } &= \bigl\lVert \avg{ \widehat{ w } ^{\mathrm{tan}} } \bigr\rVert ^2 _{ \partial \mathcal{T} _h } + \bigl\lVert \jump{ \widehat{ w } ^{\mathrm{tan}} } \bigr\rVert _{ \partial \mathcal{T} _h } ^2 , \\
    \lVert \widehat{ w } ^{\mathrm{nor}} \rVert ^2 _{ \partial \mathcal{T} _h } &= \bigl\lVert \avg{ \widehat{ w } ^{\mathrm{nor}} } \bigr\rVert ^2 _{ \partial \mathcal{T} _h } + \bigl\lVert \jump{ \widehat{ w } ^{\mathrm{nor}} } \bigr\rVert _{ \partial \mathcal{T} _h } ^2 .
  \end{align*}

\end{corollary}

\section{HDG methods for the Hodge--Dirac problem}
\label{s:dirac_methods}

\subsection{Weak formulation}

To motivate the weak formulation for the HDG methods, we first make
the following observations about solutions to the Hodge--Dirac
problem:

\begin{itemize}
\item Using integration by parts, the equation
  $ \mathrm{D} z + p = f $ can be rewritten as
  \begin{equation*}
    ( z, \mathrm{D} w ) _{ \mathcal{T} _h }  + ( p, w ) _{ \mathcal{T} _h } + \langle z ^{\mathrm{tan}} , w ^{\mathrm{nor}} \rangle _{ \partial \mathcal{T} _h } - \langle z ^{\mathrm{nor}} , w ^{\mathrm{tan}} \rangle _{ \partial \mathcal{T} _h } = ( f, w ) _{ \mathcal{T} _h } ,
  \end{equation*}
  for all broken test functions
  $ w \in \prod _{ K \in \mathcal{T} _h } H \Lambda (K) \cap H ^\ast
  \Lambda (K) $.

\item The orthogonality condition
  $ z \in \mathfrak{H} (\Omega) ^\perp $ can be written as
  \begin{equation*}
    ( z, q ) _{ \mathcal{T} _h } = 0 ,
  \end{equation*}
  for all harmonic test functions $ q \in \mathfrak{H} (\Omega) $.

\item Since
  $ z \in H \Lambda (\Omega) \cap \mathring{ H } ^\ast \Lambda
  (\Omega) $, we have $ \jump{ z ^{\mathrm{tan}} } = 0 $ and
  $ \jump{ z ^{\mathrm{nor}} } = 0 $. By \cref{p:jump-avg}, the latter
  condition says that $ z ^{\mathrm{nor}} $ is
  $ L ^2 \Lambda ( \partial \mathcal{T} _h )$-orthogonal to
  single-valued tangential traces, i.e.,
  \begin{equation*}
    \langle z ^{\mathrm{nor}} , \widehat{ w } ^{\mathrm{tan}} \rangle _{ \partial \mathcal{T} _h } = 0 ,
  \end{equation*}
  for all test functions
  $ \widehat{ w } ^{\mathrm{tan}} \in L ^2 \Lambda ( \partial
  \mathcal{T} _h ) $ with
  $ \jump{ \widehat{ w } ^{\mathrm{tan}} } = 0 $.
\end{itemize}
We now define the equal-order spaces of polynomial degree $r$
\begin{equation*}
  W _h \coloneqq \prod _{ K \in \mathcal{T} _h } \mathcal{P} _r \Lambda (K) , \qquad \widehat{ W } _h ^{\mathrm{tan}} \coloneqq \prod _{ e \subset \partial \mathcal{T} _h } \mathcal{P} _r \Lambda (e) , \qquad \widehat{ V } _h ^{\mathrm{tan}} \coloneqq  \bigl\{ \widehat{ w } _h ^{\mathrm{tan}} \in \widehat{ W } _h ^{\mathrm{tan}} : \jump{ \widehat{ w } _h ^{\mathrm{tan}} } = 0 \bigr\} ,
\end{equation*}
along with the harmonic spaces
\begin{equation*}
  \overline{ \mathfrak{H} } _h \coloneqq W _h \cap \prod _{ K \in \mathcal{T} _h } \mathring{ \mathfrak{H} } (K) , \qquad \mathfrak{H} _h \coloneqq W _h \cap \mathfrak{H} (\Omega) .
\end{equation*}
Following the terminology of \citep{AwFaGuSt2023}, we call
$ \overline{ \mathfrak{H} } _h $ the space of \emph{local harmonic
  forms}.

\begin{remark}
  \label{r:contractible}
  If the $ K \in \mathcal{T} _h $ are contractible (e.g., simplices),
  then $ \mathring{ \mathfrak{H} } (K) $ consists of constant
  $n$-forms, and
  $ \overline{ \mathfrak{H} } _h = \prod _{ K \in \mathcal{T} _h }
  \mathcal{P} _0 \Lambda ^n ( K ) $ is simply a $ \mathrm{DG}_0 $
  space. Similarly, if $\Omega$ is contractible, then
  $ \mathfrak{H} (\Omega) $ consists of constant $0$-forms, and
  $ \mathfrak{H} _h = \mathfrak{H} (\Omega) \cong \mathbb{R} $.

  However, if $\Omega$ has nontrivial topology, then
  $ \mathfrak{H} (\Omega) $ may contain non-polynomial harmonic
  $k$-forms for $ k > 0 $, in which case
  $ \mathfrak{H} _h \subsetneq \mathfrak{H} (\Omega) $ is a proper
  subspace. Note that there are some circumstances where
  $ \mathfrak{H} _h = \mathfrak{H} (\Omega) $ holds for $\Omega$ with
  nontrivial topology: for example, if $\Omega$ is a rectangle with
  periodic boundary conditions (i.e., a flat torus), then
  $ \mathfrak{H} _h ^1 = \mathfrak{H} ^1 (\Omega) $ consists of
  constant $1$-forms.

  Finally, we note that our $ \mathfrak{H} _h $ is generally distinct
  from the space of ``discrete harmonic forms'' arising in conforming
  FEEC \citep{ArFaWi2006,ArFaWi2010,Arnold2018}, where the latter is
  not necessarily a subspace of $ \mathfrak{H} (\Omega) $ but is
  guaranteed to have the same dimension, regardless of the topology of
  $\Omega$; see also \citet{Christiansen2007}.
\end{remark}

\begin{definition}
  The HDG method is defined by the following weak problem: Find
  \begin{alignat*}{4}
    \text{(local variables)} \qquad z _h &\in W _h , \qquad & \overline{p} _h &\in \overline{ \mathfrak{H} } _h , \\
    \text{(global variables)} \qquad \overline{z} _h &\in \overline{ \mathfrak{H} } _h , \qquad & p _h &\in \mathfrak{H} _h , \qquad & \widehat{ z } _h ^{\mathrm{tan}} &\in \widehat{ V } _h ^{\mathrm{tan}} ,
  \end{alignat*}
  satisfying
  \begin{subequations}
    \label{e:dirac_ns}
    \begin{alignat}{2}
      ( z _h , \mathrm{D} w _h ) _{ \mathcal{T} _h } + ( p _h - \overline{p} _h , w _h ) _{ \mathcal{T} _h } + \langle \widehat{ z } _h ^{\mathrm{tan}} , w _h ^{\mathrm{nor}} \rangle _{ \partial \mathcal{T} _h } - \langle \widehat{ z } _h ^{\mathrm{nor}} , w _h ^{\mathrm{tan}} \rangle _{ \partial \mathcal{T} _h } &= ( f, w _h ) _{ \mathcal{T} _h } , \quad &\forall w _h &\in W _h , \label{e:dirac_ns_w} \\
      ( \overline{z} _h - z _h , \overline{q} _h ) _{ \mathcal{T} _h } &= 0 , \quad &\forall \overline{q} _h &\in \overline{ \mathfrak{H} } _h, \label{e:dirac_ns_qbar} \\
      ( \overline{p} _h , \overline{ w } _h ) _{ \mathcal{T} _h } &= 0 , \quad &\forall \overline{ w } _h &\in \overline{ \mathfrak{H} } _h, \label{e:dirac_ns_wbar} \\
      ( z _h , q _h ) _{ \mathcal{T} _h } &= 0 , \quad &\forall q _h &\in \mathfrak{H} _h, \label{e:dirac_ns_q} \\
      \langle \widehat{ z } _h ^{\mathrm{nor}} , \widehat{ w } _h ^{\mathrm{tan}} \rangle _{ \partial \mathcal{T} _h } &= 0 , \quad &\forall \widehat{ w } _h ^{\mathrm{tan}} &\in \widehat{ V } _h ^{\mathrm{tan}} , \label{e:dirac_ns_wtan}
    \end{alignat}
  \end{subequations}
  where
  \begin{equation*}
    \widehat{ z } _h ^{\mathrm{nor}} \coloneqq z _h ^{\mathrm{nor}} - \alpha ( \widehat{ z } _h ^{\mathrm{tan}} - z _h ^{\mathrm{tan}} ) .
  \end{equation*}
  The penalty $\alpha$ is assumed to be piecewise-constant, i.e.,
  $ \alpha _e ^k \in \mathbb{R} $ for each
  $ e \subset \partial \mathcal{T} _h $ and form degree $k$.
\end{definition}

\begin{remark}
  In \cref{s:dirac_existence_uniqueness}, we will see that a further
  condition on the penalty, that $\alpha$ alternates sign in $k$,
  guarantees existence and uniqueness of the numerical solution. This
  same condition also allows for \emph{static condensation}, whereby
  the local variables are efficiently eliminated, yielding a smaller
  Schur complement system in the global variables alone.
\end{remark}

Before exploring this method further, we prove a few basic facts about
its numerical solutions.

\begin{lemma}
  \label{l:dirac_properties}
  A solution to \eqref{e:dirac_ns} satisfies
  \begin{equation*}
    \overline{p} _h = 0 , \qquad \overline{z} _h = P _{ \overline{ \mathfrak{H} } _h } z _h , \qquad p _h = P _{ \mathfrak{H} _h } f , \qquad \jump{ \widehat{ z } _h ^{\mathrm{nor}} } = 0 ,
  \end{equation*}
  where $ P _{ \overline{ \mathfrak{H} } _h } $ and
  $ P _{ \mathfrak{H} _h } $ denote $ L ^2 $-orthogonal projection
  onto the local and global harmonic spaces.
\end{lemma}

\begin{proof}
  The first equality follows from \eqref{e:dirac_ns_wbar} with
  $ \overline{w} _h = \overline{p} _h $, and the second is immediate
  from \eqref{e:dirac_ns_qbar}. For the third, observe that every
  $ q _h \in \mathfrak{H} _h \subset \mathfrak{H} (\Omega) $ satisfies
  $ \mathrm{D} q _h = 0 $, $ \jump{ q _h ^{\mathrm{nor}} } = 0 $, and
  $ \jump{ q _h ^{\mathrm{tan}} } = 0 $; since
  $ \overline{p} _h = 0 $, it follows that \eqref{e:dirac_ns_w} with
  $ w _h = q _h $ simplifies to
  $ ( p _h , q _h ) _{ \mathcal{T} _h } = ( f , q _h ) _{ \mathcal{T}
    _h } $. Finally, the fourth follows from \eqref{e:dirac_ns_wtan}
  with
  $ \widehat{ w } _h ^{\mathrm{tan}} = \jump{ \widehat{ z } _h
    ^{\mathrm{nor}} } $; this is in
  $ \widehat{ V } _h ^{\mathrm{tan}} $ since $\alpha$ is
  piecewise-constant and, by \cref{d:jump-avg}, normal jumps are
  tangentially single-valued.
\end{proof}

\begin{remark}
  \label{r:local_harmonics}
  This lemma shows that the local harmonics do not affect the other
  solution components: we could get the same $ z _h $, $ p _h $, and
  $ \widehat{ z } _h ^{\mathrm{tan}} $ by dropping the
  $ \overline{p} _h $ term from \eqref{e:dirac_ns_w} and eliminating
  \eqref{e:dirac_ns_qbar}--\eqref{e:dirac_ns_wbar} altogether. However,
  the local harmonics are needed to obtain well-defined local solvers
  and perform static condensation, as we will see in
  \cref{s:dirac_existence_uniqueness}.
\end{remark}

An equivalent (and more obviously symmetric) weak form is obtained by
substituting the expression for $ \widehat{ z } _h ^{\mathrm{nor}} $
into \eqref{e:dirac_ns} and integrating by parts, yielding
\begin{subequations}
  \label{e:dirac}
  \begin{alignat}{2}
    ( z _h , \delta  w _h ) _{ \mathcal{T} _h } + ( \delta z _h , w _h ) _{ \mathcal{T} _h } + ( p _h - \overline{p} _h , w _h ) _{ \mathcal{T} _h } \qquad & \label{e:dirac_w} \\
    {}+ \langle \widehat{ z } _h ^{\mathrm{tan}} , w _h ^{\mathrm{nor}} \rangle _{ \partial \mathcal{T} _h } + \bigl\langle \alpha ( \widehat{ z } _h ^{\mathrm{tan}} - z _h ^{\mathrm{tan}} ) , w _h ^{\mathrm{tan}} \bigr\rangle _{ \partial \mathcal{T} _h } &= ( f, w _h ) _{ \mathcal{T} _h } , \quad &\forall w _h &\in W _h , \notag \\
    ( \overline{z} _h - z _h , \overline{q} _h ) _{ \mathcal{T} _h } &= 0 , \quad &\forall \overline{q} _h &\in \overline{ \mathfrak{H} } _h, \label{e:dirac_qbar} \\
    ( \overline{p} _h , \overline{ w } _h ) _{ \mathcal{T} _h } &= 0 , \quad &\forall \overline{ w } _h &\in \overline{ \mathfrak{H} } _h, \label{e:dirac_wbar} \\
    ( z _h , q _h ) _{ \mathcal{T} _h } &= 0 , \quad &\forall q _h &\in \mathfrak{H} _h, \label{e:dirac_q} \\
    \bigl\langle z _h ^{\mathrm{nor}} - \alpha ( \widehat{ z } _h ^{\mathrm{tan}} - z _h ^{\mathrm{tan}} ) , \widehat{ w } _h ^{\mathrm{tan}} \bigr\rangle _{ \partial \mathcal{T} _h } &= 0 , \quad &\forall \widehat{ w } _h ^{\mathrm{tan}} &\in \widehat{ V } _h ^{\mathrm{tan}} . \label{e:dirac_wtan}
  \end{alignat}
\end{subequations}
We will use this formulation interchangeably with \eqref{e:dirac_ns}.

\begin{remark}
  This is a special case of the LDG-H methods proposed in
  \citet[Section 4.2]{StZa2026} for a more general class of semilinear
  Hamiltonian PDEs involving the Hodge--Dirac operator; compare
  \eqref{e:dirac_w} and \eqref{e:dirac_wtan} with \citep[Equation
  36]{StZa2026}. However, that paper examines structure-preserving
  properties for \emph{families} of solutions and leaves open the
  treatment of harmonic forms and boundary conditions needed to
  specify a \emph{particular} solution; see \citep[Remark
  3.1]{StZa2026}. Our incorporation of local/global harmonics and
  boundary conditions above is new, as is the analysis to follow.
\end{remark}

\subsection{Existence and uniqueness}
\label{s:dirac_existence_uniqueness}

We next prove both local and global existence and uniqueness of the
weak formulation for the HDG method. By ``local'' existence and
uniqueness we mean that, given any value of the global variables, it
is possible to solve uniquely for the local variables satisfying
\eqref{e:dirac_ns_w}--\eqref{e:dirac_ns_qbar} or equivalently
\eqref{e:dirac_w}--\eqref{e:dirac_qbar}. These \emph{local solvers},
in the language of \citep{CoGoLa2009}, allow us to eliminate the local
variables and obtain a smaller Schur complement system containing only
the global variables, a procedure known as \emph{static
  condensation}. This elimination can be done quite cheaply: since the
local function spaces are broken, it is possible to apply the local
solvers element-by-element on each $ K \in \mathcal{T} _h $, in
parallel if desired. Static condensation thus substantially reduces
the number of degrees of freedom in the resulting global system and is
one of the primary selling points of HDG methods; by contrast,
non-hybridizable DG methods may have many globally coupled degrees of
freedom that cannot be eliminated in this manner \citep{CoGoLa2009}.

The existence and uniqueness results, along with others to follow,
depend on the following assumption that the penalty alternates sign
between even and odd $k$. We will assume that this condition holds for
the remainder of the paper, referring to it as the
\emph{alternating-sign assumption}.

\begin{assumption}
  \label{a:alternating-sign}
  Assume $ ( - 1 ) ^k \alpha ^k $ is either strictly positive for all
  $k$ or strictly negative for all $k$.
\end{assumption}

It will also be useful to split forms into their even-$k$ and odd-$k$
components, which we denote using the superscripts $+$ and $-$. For
example, $ w _h \in W _h $ may be written as
$ w _h = w _h ^+ + w _h ^- $, where
\begin{equation*}
  w _h ^+ \coloneqq \bigoplus _{ k \text{ even} } w _h ^k , \qquad w _h ^- \coloneqq \bigoplus _{ k \text{ odd} } w _h ^k .
\end{equation*}
We likewise use $ \alpha ^+ $ and $ \alpha ^- $ to denote the even-$k$
and odd-$k$ components of the penalty parameters.

\begin{theorem}[local existence and uniqueness]
  \label{t:dirac_local_eu}
  Given any $ \overline{ z } _h \in \overline{ \mathfrak{H} } _h $,
  $ p _h \in \mathfrak{H} _h $,
  $ \widehat{ z } _h ^{\mathrm{tan}} \in \widehat{ V } _h
  ^{\mathrm{tan}} $, and $ f \in L ^2 \Lambda (\Omega) $, there exist
  unique $ z _h \in W _h $ and
  $ \overline{p} _h \in \overline{ \mathfrak{H} } _h $ satisfying
  \eqref{e:dirac_w} and \eqref{e:dirac_qbar}.
\end{theorem}

\begin{proof}
  Since this is a square system, it suffices to show that
  \begin{subequations}
    \begin{alignat}{2}
      ( z _h , \delta  w _h ) _{ \mathcal{T} _h } + ( \delta z _h , w _h ) _{ \mathcal{T} _h } - ( \overline{p} _h , w _h ) _{ \mathcal{T} _h } - \langle \alpha z _h ^{\mathrm{tan}} , w _h ^{\mathrm{tan}} \rangle _{ \partial \mathcal{T} _h } &= 0 , \quad &\forall w _h &\in W _h , \label{e:dirac_local_w} \\
      -( z _h , \overline{q} _h ) _{ \mathcal{T} _h } &= 0 , \quad &\forall \overline{q} _h &\in \overline{ \mathfrak{H} } _h, \label{e:dirac_local_qbar}
    \end{alignat}
  \end{subequations}
  has only the trivial solution $ z _h = 0 $ and
  $ \overline{p} _h = 0 $.

  First, take $ w _h = \overline{p} _h $ in
  \eqref{e:dirac_local_w}. Integrating by parts and rearranging gives
  \begin{equation*}
    ( \overline{p} _h , \overline{p} _h ) _{ \mathcal{T} _h } = ( z _h , \mathrm{D} \overline{p} _h ) _{ \mathcal{T} _h } - \langle z _h ^{\mathrm{nor}} + \alpha z _h ^{\mathrm{tan}} , \overline{p} _h ^{\mathrm{tan}} \rangle _{ \partial \mathcal{T} _h } = 0 ,
  \end{equation*}
  since $ \mathrm{D} \overline{p} _h = 0 $ and
  $ \overline{p} _h ^{\mathrm{tan}} = 0 $ by definition of
  $ \overline{ \mathfrak{H} } _h $. Hence, $ \overline{p} _h = 0 $.

  Next, write $ z _h = z _h ^+ + z _h ^- $, and take
  $ w _h = z _h ^+ - z _h ^- $ in \eqref{e:dirac_local_w}, i.e.,
  $ w _h ^k = ( - 1 ) ^k z _h ^k $. The
  $ \pm ( z _h ^\mp , \delta z _h ^\pm ) _{ \mathcal{T} _h } $ and
  $ \pm ( \delta z _h ^\mp , z _h ^\pm ) _{ \mathcal{T} _h } $ terms
  cancel, leaving only
  \begin{equation*}
    \langle \alpha ^+ z _h ^{+, \mathrm{tan}} , z _h ^{+, \mathrm{tan}} \rangle _{ \partial \mathcal{T} _h } - \langle \alpha ^- z _h ^{-, \mathrm{tan}} , z _h ^{-, \mathrm{tan}} \rangle _{ \partial \mathcal{T} _h } = 0 .
  \end{equation*}
  The alternating-sign assumption implies that the left-hand side is
  either positive- or negative-definite in $ z _h ^{\mathrm{tan}} $,
  so its vanishing implies $ z _h ^{\mathrm{tan}} = 0 $. Integrating
  \eqref{e:dirac_local_w} by parts then gives
  \begin{equation*}
    ( \mathrm{D} z _h , w _h ) _{ \mathcal{T} _h } = 0 , \quad \forall w _h \in W _h ,
  \end{equation*}
  and taking $ w _h = \mathrm{D} z _h $ gives $ \mathrm{D} z _h = 0 $,
  i.e., $ z _h \in \overline{ \mathfrak{H} } _h $. Finally, taking
  $ \overline{q} _h = z _h $ in \eqref{e:dirac_local_qbar} implies
  $ z _h = 0 $.
\end{proof}

\begin{theorem}[global existence and uniqueness]
  \label{t:dirac_global_eu}
  Given any $ f \in L ^2 \Lambda (\Omega) $, there exists a unique
  solution to \eqref{e:dirac}.
\end{theorem}

\begin{proof}
  Since \eqref{e:dirac} is a square system, it suffices to show that
  $ f = 0 $ has only the trivial solution. \Cref{l:dirac_properties}
  immediately gives $ \overline{p} _h = 0 $ and $ p _h = 0 $, while
  $ \overline{z} _h = P _{ \overline{ \mathfrak{H} } _h } z _h $
  clearly vanishes if $ z _h $ does. It therefore remains to show that
  $ z _h = 0 $ and $ \widehat{ z } _h ^{\mathrm{tan}} = 0 $.

  First, observe that adding \eqref{e:dirac_w} and \eqref{e:dirac_wtan}
  gives
  \begin{equation*}
    ( z _h , \delta  w _h ) _{ \mathcal{T} _h } + ( \delta z _h , w _h ) _{ \mathcal{T} _h } + \langle \widehat{ z } _h ^{\mathrm{tan}} , w _h ^{\mathrm{nor}} \rangle _{ \partial \mathcal{T} _h } + \langle z _h ^{\mathrm{nor}}, \widehat{ w } _h ^{\mathrm{tan}} \rangle _{ \partial \mathcal{T} _h }
    - \bigl\langle \alpha ( \widehat{ z } _h ^{\mathrm{tan}} - z _h ^{\mathrm{tan}} ) , \widehat{ w } _h ^{\mathrm{tan}} - w _h ^{\mathrm{tan}} \bigr\rangle _{ \partial \mathcal{T} _h } = 0.
  \end{equation*}
  Similarly to the proof of \cref{t:dirac_local_eu}, we
  take $ w _h = z _h ^+ - z _h ^- $ and
  $ \widehat{ w } _h ^{\mathrm{tan}} = \widehat{ z } _h ^{+,
    \mathrm{tan}} - \widehat{ z } _h ^{-, \mathrm{tan}} $. The
  $ \pm ( z _h ^\mp , \delta z _h ^\pm ) _{ \mathcal{T} _h } $ and
  $ \pm ( \delta z _h ^\mp , z _h ^\pm ) _{ \mathcal{T} _h } $ terms
  cancel, as do the
  $ \pm \langle \widehat{ z } _h ^{\mp, \mathrm{tan}} , z _h ^{ \pm,
    \mathrm{nor}} \rangle _{ \partial \mathcal{T} _h } $ and
  $ \pm \langle z _h ^{\mp, \mathrm{nor}} , \widehat{ z } _h ^{ \pm,
    \mathrm{tan}} \rangle _{ \partial \mathcal{T} _h } $ terms,
  leaving only
  \begin{equation*}
    \bigl\langle \alpha ^+ ( \widehat{ z } _h ^{+, \mathrm{tan}} - z _h ^{+, \mathrm{tan}} ) , \widehat{ z } _h ^{+, \mathrm{tan}} - z _h ^{+, \mathrm{tan}} \bigr\rangle _{ \partial \mathcal{T} _h } - \bigl\langle \alpha ^- ( \widehat{ z } _h ^{-, \mathrm{tan}} - z _h ^{-, \mathrm{tan}} ) , \widehat{ z } _h ^{-, \mathrm{tan}} - z _h ^{-, \mathrm{tan}} \bigr\rangle _{ \partial \mathcal{T} _h } = 0 .
  \end{equation*}
  The alternating-sign assumption implies
  $ \widehat{ z } _h ^{\mathrm{tan}} = z _h ^{\mathrm{tan}} $, and
  consequently
  $ \widehat{ z } _h ^{\mathrm{nor}} = z _h ^{\mathrm{nor}} $ as
  well. Since $ \jump{ \widehat{ z } _h ^{\mathrm{tan}} } = 0 $ by
  definition of $ \widehat{ V } _h ^{\mathrm{tan}} $, and
  $ \jump{ \widehat{ z } _h ^{\mathrm{nor}} } = 0 $ by
  \cref{l:dirac_properties}, we now have
  $ z _h \in H \Lambda (\Omega) \cap \mathring{ H } ^\ast \Lambda
  (\Omega) $. Equation \eqref{e:dirac_w} therefore simplifies to
  \begin{equation*}
    ( z _h , \delta  w _h ) _{ \mathcal{T} _h } + ( \delta z _h , w _h ) _{ \mathcal{T} _h } + \langle z _h ^{\mathrm{tan}} , w _h ^{\mathrm{nor}} \rangle _{ \partial \mathcal{T} _h } = 0 , \quad \forall w _h \in W _h ,
  \end{equation*}
  and integrating by parts gives
  \begin{equation*}
    ( \mathrm{D} z _h , w _h ) _{ \mathcal{T} _h } = 0 , \quad \forall w _h \in W _h .
  \end{equation*}
  Taking $ w _h = \mathrm{D} z _h $ implies $ \mathrm{D} z _h = 0 $,
  i.e., $ z _h \in \mathfrak{H} _h $. Finally, taking $ q _h = z _h $
  in \eqref{e:dirac_q} gives $ z _h = 0 $, and thus
  $ \widehat{ z } _h ^{\mathrm{tan}} = z _h ^{\mathrm{tan}} = 0 $ as
  well.
\end{proof}

\subsection{Jump-average DG formulation}

We may also express the HDG method in a more classical DG form, using
jumps and averages instead of the additional hybrid ``hat''
variables. While the hybrid formulation is preferable for
implementation (due to static condensation, ease of local assembly,
etc.), the equivalent formulation developed in this section will be
convenient for performing the subsequent error analysis. Since the
local harmonics are only included in the hybrid formulation for
purposes of defining the local solvers and performing static
condensation (see \cref{r:local_harmonics}), they are omitted from
this ``unhybridized'' formulation.

In general, the penalty $\alpha$ may take separate values on the two
sides $ e ^\pm $ of an internal facet
$ e \subset \partial \mathcal{T} _h \setminus \partial \Omega $. We
take its jump and average in the tangential-trace sense of
\cref{d:jump-avg},
\begin{alignat*}{4}
  \jump{ \alpha } _{ e ^\pm } &\coloneqq \frac{1}{2} \bigl( \alpha _{ e ^\pm } - \alpha _{ e ^\mp } \bigr) , \qquad & \avg{ \alpha } _{ e ^\pm } &\coloneqq \frac{1}{2} \bigl( \alpha _{ e ^+ } + \alpha _{ e ^- } \bigr) , \qquad & e &\not\subset \partial \Omega, \\
  \jump{ \alpha } _{e \phantom{^\pm}} &\coloneqq 0 , \qquad & \avg{ \alpha } _{e \phantom{^\pm}} &\coloneqq \alpha _e , \qquad & e &\subset \partial \Omega ,
\end{alignat*}
and we say $\alpha$ is single-valued if $ \jump{ \alpha } = 0 $. It
will also be useful to denote the reciprocal
$ \beta \coloneqq 1 / \alpha $, whose jump and average we take in the
exact same way. We will use the identities
\begin{alignat*}{2}
  \jump{ \alpha \widehat{ w } ^{\mathrm{tan}} } &= \avg{ \alpha } \jump{ \widehat{ w } ^{\mathrm{tan}} } + \jump{ \alpha } \avg{ \widehat{ w } ^{\mathrm{tan}} }, \qquad & \avg{ \alpha \widehat{ w } ^{\mathrm{tan}} } &= \avg{ \alpha } \avg{ \widehat{ w } ^{\mathrm{tan}} } + \jump{ \alpha } \jump { \widehat{ w } ^{\mathrm{tan}} } , \\
  \jump{ \beta \widehat{ w } ^{\mathrm{nor}} } &= \avg{ \beta } \jump{ \widehat{ w } ^{\mathrm{nor}} } + \jump{ \beta } \avg{ \widehat{ w } ^{\mathrm{nor}} } , \qquad & \avg{ \beta \widehat{ w } ^{\mathrm{nor}} } &= \avg{ \beta } \avg{ \widehat{ w } ^{\mathrm{nor}} } + \jump{ \beta } \jump { \widehat{ w } ^{\mathrm{nor}} } ,
\end{alignat*}
which are straightforward to verify from the definitions, for all
$ \widehat{ w } ^{\mathrm{tan}} , \widehat{ w } ^{\mathrm{nor}} \in L
^2 \Lambda ( \partial \mathcal{T} _h ) $.

\begin{lemma}
  \label{l:dg_traces}
  The approximate traces for the HDG method satisfy
  \begin{align*}
    \widehat{ z } _h ^{\mathrm{tan}} &= \avg{ z _h ^{\mathrm{tan}} } + \frac{ 1 }{ \avg{\alpha} } \jump{ z _h ^{\mathrm{nor}} } + \frac{ \jump{\alpha} }{ \avg{ \alpha } } \jump{ z _h ^{\mathrm{tan}} } , \\
    \widehat{ z } _h ^{\mathrm{nor}} &= \avg{ z _h ^{\mathrm{nor}} } + \frac{ 1 }{ \avg{\beta} } \jump{ z _h ^{\mathrm{tan}} } + \frac{ \jump{\beta} }{ \avg{ \beta } } \jump{ z _h ^{\mathrm{nor}} } .
  \end{align*}
\end{lemma}

\begin{proof}
  Recalling that normal jump has the same definition as tangential average, we
  have
  \begin{align*}
    \jump{ \widehat{ z } _h ^{\mathrm{nor}} }
    &= \jump{ z _h ^{\mathrm{nor}} } - \bigavg{ \alpha ( \widehat{ z } _h ^{\mathrm{tan}} - z _h ^{\mathrm{tan}} ) } \\
    &= \jump{ z _h ^{\mathrm{nor}} } - \avg{\alpha}\avg{\widehat{ z } _h ^{\mathrm{tan}} - z _h ^{\mathrm{tan}} } - \jump{\alpha}\jump{ \widehat{ z } _h ^{\mathrm{tan}} - z _h ^{\mathrm{tan}} } \\
    &= \jump{ z _h ^{\mathrm{nor}} } - \avg{\alpha} \widehat{ z } _h ^{\mathrm{tan}} + \avg{\alpha} \avg{ z _h ^{\mathrm{tan}} } + \jump{\alpha}\jump{z _h ^{\mathrm{tan}}} .
  \end{align*}
  Since $ \jump{ \widehat{ z } _h ^{\mathrm{nor}} } = 0 $, dividing
  by $ \avg{\alpha} $ and rearranging gives the first
  identity. Likewise, observing that
  $ \widehat{ z } _h ^{\mathrm{tan}} = z _h ^{\mathrm{tan}} - \beta
  ( \widehat{ z } _h ^{\mathrm{nor}} - z _h ^{\mathrm{nor}} ) $
  satisfies $ \jump{ \widehat{ z } _h ^{\mathrm{tan}} } = 0 $, a
  similar calculation gives the second identity.
\end{proof}

This allows us to eliminate the hybrid variables to obtain the
following equivalent formulation.

\begin{proposition}
  \label{p:dg}
  The HDG method is equivalent to finding $ z _h \in W _h $ and
  $ p _h \in \mathfrak{H} _h $ satisfying
  \begin{subequations}
    \label{e:dirac_dg}
    \begin{alignat}{2}
      ( z _h , \delta  w _h ) _{ \mathcal{T} _h } + ( \delta z _h , w _h ) _{ \mathcal{T} _h } + ( p _h , w _h ) _{ \mathcal{T} _h } \qquad \qquad \label{e:dirac_dg_w} \\
      {}+ \bigl\langle \avg{ z _h ^{\mathrm{tan}} }, \jump{ w _h ^{\mathrm{nor}} } \bigr\rangle _{ \partial \mathcal{T} _h } + \bigl\langle \jump{ z _h ^{\mathrm{nor}} } , \avg{ w _h ^{\mathrm{tan}} } \bigr\rangle _{ \partial \mathcal{T} _h } \notag \\
      {}+ \biggl\langle \frac{ 1 }{ \avg{\alpha} } \jump{ z _h ^{\mathrm{nor}} } + \frac{ \jump{\alpha} }{ \avg{\alpha} } \jump{ z _h ^{\mathrm{tan}} } , \jump{ w _h ^{\mathrm{nor}} } \biggr\rangle _{ \partial \mathcal{T} _h } \notag \\
      {}- \biggl\langle \frac{ 1 }{ \avg{\beta} } \jump{ z _h ^{\mathrm{tan}} } + \frac{ \jump{\beta} }{ \avg{\beta} } \jump{ z _h ^{\mathrm{nor}} } , \jump{ w _h ^{\mathrm{tan}} } \biggr\rangle _{ \partial \mathcal{T} _h }
       &= ( f, w _h ) _{ \mathcal{T} _h } , \quad &\forall w _h &\in W _h , \notag \\
      ( z _h , q _h ) _{ \mathcal{T} _h } &= 0 , \quad &\forall q _h &\in \mathfrak{H} _h. \label{e:dirac_dg_q}
    \end{alignat}
  \end{subequations}
\end{proposition}

\begin{proof}
  Starting with \eqref{e:dirac}, we eliminate the local harmonic
  variables and equations \eqref{e:dirac_qbar}--\eqref{e:dirac_wbar}
  using \cref{l:dirac_properties}, and we eliminate the approximate
  trace variables and equation \eqref{e:dirac_wtan} using
  \cref{l:dg_traces}. Combining \cref{p:jump-avg,l:dg_traces}, we have
  \begin{equation*}
    \langle \widehat{ z } _h ^{\mathrm{tan}} , w _h ^{\mathrm{nor}} \rangle _{ \partial \mathcal{T} _h } = \biggl\langle \avg{ z _h ^{\mathrm{tan}} } + \frac{ 1 }{ \avg{\alpha} } \jump{ z _h ^{\mathrm{nor}} } + \frac{ \jump{\alpha} }{ \avg{ \alpha } } \jump{ z _h ^{\mathrm{tan}} } , \jump{ w _h ^{\mathrm{nor}} } \biggr\rangle _{ \partial \mathcal{T} _h } ,
  \end{equation*}
  and
  \begin{align*}
    \bigl\langle \alpha ( \widehat{ z } _h ^{\mathrm{tan}} - z _h ^{\mathrm{tan}} ) , w _h ^{\mathrm{tan}} \bigr\rangle _{ \partial \mathcal{T} _h }
    &= - \langle \widehat{ z } _h ^{\mathrm{nor}} - z _h ^{\mathrm{nor}}, w _h ^{\mathrm{tan}} \rangle _{ \partial \mathcal{T} _h }\\
    &= \bigl\langle \jump{z _h ^{\mathrm{nor}} } , \avg{ w _h ^{\mathrm{tan}} } \bigr\rangle _{ \partial \mathcal{T} _h } - \biggl\langle \frac{ 1 }{ \avg{\beta} } \jump{ z _h ^{\mathrm{tan}} } + \frac{ \jump{\beta} }{ \avg{\beta} } \jump{ z _h ^{\mathrm{nor}} } , \jump{ w _h ^{\mathrm{tan}} } \biggr\rangle _{ \partial \mathcal{T} _h } .
  \end{align*}
  Substituting these into \eqref{e:dirac_w} with
  $ \overline{p} _h = 0 $ yields \eqref{e:dirac_dg_w}, while
  \eqref{e:dirac_q} is identical to \eqref{e:dirac_dg_q}.
\end{proof}

\begin{remark}
  Note that
  $ \jump{ \beta } / \avg{ \beta } = - \jump{ \alpha } / \avg{ \alpha
  } $, since
  \begin{equation*}
    \avg{ \alpha } \jump{ \beta } + \jump{ \alpha } \avg{ \beta } =
    \jump{ \alpha \beta } = \jump{ 1 } = 0 .
  \end{equation*}
  Hence, the formulation \eqref{e:dirac_dg} is symmetric.
\end{remark}

\begin{corollary}
  \label{c:dg}
  If $\alpha$ is single-valued, then the approximate traces for the
  HDG method satisfy
  \begin{align*}
    \widehat{ z } _h ^{\mathrm{tan}} &= \avg{ z _h ^{\mathrm{tan}} } + \beta \jump{ z _h ^{\mathrm{nor}} } , \\
    \widehat{ z } _h ^{\mathrm{nor}} &= \avg{ z _h ^{\mathrm{nor}} } + \alpha \jump{ z _h ^{\mathrm{tan}} } ,
  \end{align*}
  and the method is equivalent to finding $ z _h \in W _h $ and
  $ p _h \in \mathfrak{H} _h $ satisfying
  \begin{alignat*}{2}
    ( z _h , \delta  w _h ) _{ \mathcal{T} _h } + ( \delta z _h , w _h ) _{ \mathcal{T} _h } + ( p _h , w _h ) _{ \mathcal{T} _h } \qquad \qquad \\
    {}+ \bigl\langle \avg{ z _h ^{\mathrm{tan}} }, \jump{ w _h ^{\mathrm{nor}} } \bigr\rangle _{ \partial \mathcal{T} _h } + \bigl\langle \jump{ z _h ^{\mathrm{nor}} } , \avg{ w _h ^{\mathrm{tan}} } \bigr\rangle _{ \partial \mathcal{T} _h } \quad \\
    {}+ \bigl\langle \beta \jump{ z _h ^{\mathrm{nor}} } , \jump{ w _h ^{\mathrm{nor}} } \bigr\rangle _{ \partial \mathcal{T} _h } - \bigl\langle \alpha \jump{ z _h ^{\mathrm{tan}} } , \jump{ w _h ^{\mathrm{tan}} } \bigr\rangle _{ \partial \mathcal{T} _h }
    &= ( f, w _h ) _{ \mathcal{T} _h } , \quad &\forall w _h &\in W _h , \\
    ( z _h , q _h ) _{ \mathcal{T} _h } &= 0 , \quad &\forall q _h &\in \mathfrak{H} _h.
  \end{alignat*}
\end{corollary}

\section{Error analysis for the Hodge--Dirac problem}
\label{s:dirac_error}

The outline of this section, which follows the general roadmap of
\citet{CaCoPeSc2000}, is as follows:
\begin{itemize}
\item First, we express the symmetric, indefinite variational problem
  \eqref{e:dirac_dg} in terms of a non-symmetric, positive-semidefinite
  bilinear form $\mathcal{A}$, which has an associated seminorm
  $ \lvert \cdot \rvert _{\mathcal{A}} $.

\item Next, we develop abstract error estimates in the $\mathcal{A}$
  seminorm and $ L ^2 $ norm. Just as in \citep{CaCoPeSc2000}, these
  estimates are expressed in terms of a pair of functionals called
  $ K _{\mathcal{A}} $ and $ K _{\mathcal{B}} $, which encode the
  approximation-theoretic properties of the finite element space
  $ W _h $.

\item After this, we determine $ K _{\mathcal{A}} $ and
  $ K _{\mathcal{B}} $ using approximation theory.

\item Lastly, we combine the abstract and approximation-theoretic
  results to get our final estimates.
\end{itemize}

\subsection{Bilinear forms}

We begin by observing that \eqref{e:dirac_dg_w} can be written
as
\begin{equation}
  \label{e:a_dg}
  a ( z _h , w _h ) = ( f - p _h , w _h ) _{ \mathcal{T} _h } , \quad \forall w _h \in W _h ,
\end{equation}
where $ a ( \cdot , \cdot ) $ is the symmetric bilinear form
\begin{multline*}
  a ( z, w ) \coloneqq ( z , \delta  w ) _{ \mathcal{T} _h } + ( \delta z , w ) _{ \mathcal{T} _h } + \bigl\langle \avg{ z ^{\mathrm{tan}} }, \jump{ w ^{\mathrm{nor}} } \bigr\rangle _{ \partial \mathcal{T} _h } + \bigl\langle \jump{ z ^{\mathrm{nor}} } , \avg{ w ^{\mathrm{tan}} } \bigr\rangle _{ \partial \mathcal{T} _h } \\
  + \biggl\langle \frac{ 1 }{ \avg{\alpha} } \jump{ z ^{\mathrm{nor}} } + \frac{ \jump{\alpha} }{ \avg{\alpha} } \jump{ z ^{\mathrm{tan}} } , \jump{ w ^{\mathrm{nor}} } \biggr\rangle _{ \partial \mathcal{T} _h } - \biggl\langle \frac{ 1 }{ \avg{\beta} } \jump{ z ^{\mathrm{tan}} } + \frac{ \jump{\beta} }{ \avg{\beta} } \jump{ z ^{\mathrm{nor}} } , \jump{ w ^{\mathrm{tan}} } \biggr\rangle _{ \partial \mathcal{T} _h } .
\end{multline*}
For purposes of the subsequent analysis, it will be helpful to
introduce a related bilinear form that is \emph{not} symmetric but
\emph{is} positive-semidefinite, so that we can make use of its
associated seminorm.

From the alternating-sign assumption, $ \alpha ^+ $ and $ \alpha ^- $
have strictly opposite signs. As in the existence and uniqueness
proofs, taking $ w = z ^- - z ^+ $ causes the
$ a ( z ^\pm , z ^\mp ) $ terms to cancel, leaving only
\begin{align*}
  a ( z, z ^- - z ^+ ) &= a ( z ^- , z ^- ) - a ( z ^+ , z ^+ ) \\
  &= \biggl\langle \frac{ 1 }{ \avg{\alpha ^+ } } \jump{ z ^{-, \mathrm{nor}} } , \jump{ z ^{-, \mathrm{nor}} } \biggr\rangle _{ \partial \mathcal{T} _h } - \biggl\langle \frac{ 1 }{ \avg{\beta ^- } } \jump{ z ^{-, \mathrm{tan}} } , \jump{ z ^{-, \mathrm{tan}} } \biggr\rangle _{ \partial \mathcal{T} _h } \\
  &\quad - \biggl\langle \frac{ 1 }{ \avg{\alpha ^- } } \jump{ z ^{+, \mathrm{nor}} } , \jump{ z ^{+, \mathrm{nor}} } \biggr\rangle _{ \partial \mathcal{T} _h } + \biggl\langle \frac{ 1 }{ \avg{\beta ^+ } } \jump{ z ^{+, \mathrm{tan}} } , \jump{ z ^{+, \mathrm{tan}} } \biggr\rangle _{ \partial \mathcal{T} _h } .
\end{align*}
If $ \alpha ^+ > 0 $ and $ \alpha ^- < 0 $, then this is
positive-semidefinite. Otherwise, it is negative-semidefinite, and we
can obtain a positive-semidefinite form by flipping the sign.

\begin{definition}
  \label{d:A}
  Define the bilinear form
  \begin{equation*}
    \mathcal{A} ( z, w ) \coloneqq
    \begin{cases}
      a ( z, w ^- - w ^+ ) , & \text{if $ \alpha ^+ > 0 $ and $ \alpha ^- < 0 $,} \\
      a ( z, w ^+ - w ^- ) , & \text{if $ \alpha ^+ < 0 $ and $ \alpha ^- > 0 $,}
    \end{cases}
  \end{equation*}
  and define the seminorm $ \lvert \cdot \rvert _{\mathcal{A}} $ by
  \begin{align*}
    \lvert w \rvert _{\mathcal{A}} ^2
    &\coloneqq \mathcal{A} ( w, w ) \\
    &= \biggl\langle \frac{ 1 }{ \bigl\lvert \avg{ \alpha } \bigr\rvert } \jump{ w ^{\mathrm{nor}} }, \jump{ w ^{\mathrm{nor}} } \biggr\rangle _{ \partial \mathcal{T} _h } + \biggl\langle \frac{ 1 }{ \bigl\lvert \avg{\beta} \bigr\rvert } \jump{ w ^{\mathrm{tan}} } , \jump{ w ^{\mathrm{tan}} } \biggr\rangle _{ \partial \mathcal{T} _h } .
  \end{align*}
\end{definition}

Both bilinear forms may be extended from $ W _h $ to the
infinite-dimensional broken space
\begin{equation*}
  W \coloneqq \prod _{ K \in \mathcal{T} _h } H \Lambda (K) \cap H ^\ast \Lambda (K) \cap H ^t \Lambda (K),
\end{equation*}
where $ t > \frac{1}{2} $ so that traces are in
$ L ^2 \Lambda ( \partial \mathcal{T} _h ) $. In particular, assuming
$\Omega$ is $t$-regular (see \cref{s:dirac_laplace}), the exact
solution to the Hodge--Dirac problem $ \mathrm{D} z + p = f $
satisfies
\begin{equation}
  \label{e:a_exact}
  a ( z, w ) = ( f - p , w ) _{ \mathcal{T} _h } , \quad \forall w \in W ,
\end{equation}
which follows from $ \jump{ z ^{\mathrm{tan}} } = 0 $,
$ \jump{ z ^{\mathrm{nor}} } = 0 $, and integration by parts using
\cref{p:jump-avg}.

\subsection{Abstract error estimates}
\label{s:dirac_error_abstract}

As in \citet{CaCoPeSc2000}, the bulk of our error analysis will
involve proving two approximation-theoretic inequalities that relate
the bilinear form $\mathcal{A} ( \cdot , \cdot ) $ to the
$ L ^2 $-orthogonal projection onto $ W _h $.  Following
\citep{CaCoPeSc2000}, we first assume that these inequalities hold in
an abstract form, use them to prove $ \mathcal{A} $-seminorm and
$ L ^2 $-norm error estimates, and fill in the approximation-theoretic
details in \cref{s:KA_KB}.

Let $ \xi _w \coloneqq w - \Pi w $, where $\Pi$ denotes the
$ L ^2 $-orthogonal projection onto $ W _h $. Suppose, as we will
subsequently show, that we have inequalities of the form
\begin{subequations}
  \label{e:KA_KB}
  \begin{align}
    \bigl\lvert \mathcal{A} ( \xi _z, \xi _w ) \bigr\rvert &\leq K _{ \mathcal{A} } ( z, w ) , \label{e:KA} \\
    \bigl\lvert \mathcal{A} ( \xi _z, w _h ) \bigr\rvert &\leq K _{\mathcal{B}} (z) \lvert w _h \rvert _{\mathcal{A}} , \label{e:KB_zw} \\
    \intertext{for all $ z , w \in W $ and $ w _h \in W _h $. We note that the proof of \eqref{e:KB_zw} will not depend on the order of the arguments on the left-hand side, so we also have}
    \bigl\lvert \mathcal{A} ( w _h , \xi _z ) \bigr\rvert &\leq \lvert w _h \rvert _{\mathcal{A}} K _{\mathcal{B}} (z) . \label{e:KB_wz} \tag{\ref*{e:KB_zw}$ ^\prime $}
\end{align}
\end{subequations}
For the time being, we will take the functionals $ K _{\mathcal{A}} $
and $ K _{\mathcal{B}} $ satisfying \eqref{e:KA_KB} as given.

Before beginning the error analysis, we require the following
assumption on the domain $\Omega$, which holds in particular whenever
$\Omega$ is contractible (see \cref{r:contractible}).

\begin{assumption}
  \label{a:harmonic}
  Assume that $ \mathfrak{H} (\Omega) \subset W _h $, so that
  $ \mathfrak{H} _h = \mathfrak{H} (\Omega) $.
\end{assumption}

For our purposes, this assumption has two important
consequences. First, it implies that
$ p _h = p = P _{ \mathfrak{H} (\Omega) } f $, so from \eqref{e:a_dg}
and \eqref{e:a_exact} we get \emph{Galerkin orthogonality},
\begin{alignat*}{2}
  a ( z - z _h , w _h ) &= 0 , \quad &\forall w _h &\in W _h ,\\
  \intertext{which may also be written as}
  \mathcal{A} ( z - z _h , w _h ) &= 0 , \quad &\forall w _h &\in W _h .
\end{alignat*}
Second, it implies $ z - z _h \in \mathfrak{H} (\Omega) ^\perp $,
which we use in the duality argument for the $ L ^2 $-norm error
estimate.  In the analysis to follow, we denote the error by
$ e _z \coloneqq z - z _h $.

We first prove an $\mathcal{A}$-seminorm error estimate; compare
\citep[Lemma 2.3]{CaCoPeSc2000}.

\begin{lemma}
  \label{l:dirac_error_A}
  Under the preceding assumptions, we have the $\mathcal{A}$-seminorm
  error estimate
  \begin{equation*}
    \lvert e _z \rvert _{\mathcal{A}} ^2 \leq K _{\mathcal{A}} ( z, z ) + K _{\mathcal{B}} (z) ^2 .
  \end{equation*}
\end{lemma}

\begin{proof}
  Since $ \Pi z _h = z _h $, we can decompose the error as
  \begin{equation}
    \label{e:ez_decomp}
    e _z = ( z - \Pi z ) + \Pi ( z - z _h ) = \xi _z + \Pi e _z .
  \end{equation}
  Using this decomposition and Galerkin orthogonality, we have
  \begin{align*}
    \lvert e _z \rvert _{\mathcal{A}} ^2
    &= \mathcal{A} ( e _z , e _z ) \\
    &= \mathcal{A} ( e _z , \xi _z ) \\
    &= \mathcal{A} ( \xi _z, \xi _z ) + \mathcal{A} ( \Pi e _z , \xi _z ) \\
    &\leq K _{\mathcal{A}} ( z, z ) + \lvert \Pi e _z \rvert _{\mathcal{A}} K _{\mathcal{B}} (z) ,
  \end{align*}
  where the last line applies \eqref{e:KA} and \eqref{e:KB_wz}. To
  estimate $ \lvert \Pi e _z \rvert _{\mathcal{A}} $, we take
  \begin{align*}
    \lvert \Pi e _z \rvert _{\mathcal{A}} ^2
    &= \mathcal{A} ( \Pi e _z , \Pi e _z ) \\
    &= - \mathcal{A} ( \xi _z , \Pi e _z ) \\
    &\leq K _{\mathcal{B}} (z) \lvert \Pi e _z \rvert _{\mathcal{A}} ,
  \end{align*}
  again using Galerkin orthogonality, along with
  \eqref{e:KB_zw}. Thus,
  \begin{equation}
    \label{e:Pi_ez}
    \lvert \Pi e _z \rvert _{\mathcal{A}} \leq K _{\mathcal{B}} (z) ,
  \end{equation}
  and substituting this into the first estimate gives the claimed
  inequality.
\end{proof}

To obtain an $ L ^2 $-norm error estimate, we introduce the
\emph{adjoint problem}. If $ \alpha ^+ > 0 $ and $ \alpha ^- < 0 $,
then the adjoint problem is: Given
$ \lambda \in \mathfrak{H} (\Omega) ^\perp $, find
$\varphi \in \mathfrak{H} (\Omega) ^\perp $ satisfying
\begin{alignat*}{2}
  \mathrm{D} \varphi &= \lambda ^+ - \lambda ^-  \quad &\text{in } &\Omega ,\\
  \varphi  ^{\mathrm{nor}} &= 0 \quad &\text{on } &\partial \Omega .
\end{alignat*}
This is, of course, just a particular instance of the Hodge--Dirac
problem, and we observe that the Hodge--Dirac equation may be
equivalently written
$ \mathrm{D} ( \varphi ^- - \varphi ^+ ) = \lambda $. Hence, just as
in \eqref{e:a_exact}, the solution satisfies
\begin{equation*}
  a ( \varphi ^- - \varphi ^+ , w ) = ( \lambda , w ) _{ \mathcal{T} _h } , \quad \forall w \in W ,
\end{equation*}
which can be written equivalently as
\begin{equation}
  \label{e:dirac_adjoint}
  \mathcal{A} ( w, \varphi ) = ( w, \lambda ) _{ \mathcal{T} _h } , \quad \forall w \in W .
\end{equation}
On the other hand, if $ \alpha ^+ < 0 $ and $ \alpha ^- > 0 $, then we
flip the $ \pm $ superscripts in the strong form of the adjoint
problem, so that the solution again satisfies \eqref{e:dirac_adjoint}.

We now use duality to prove an $ L ^2 $-norm error estimate; compare
\citep[Lemma 2.4]{CaCoPeSc2000}.

\begin{lemma}
  \label{l:dirac_error_l2}
  Under the preceding assumptions, we have the $ L ^2 $-norm error
  estimate
  \begin{equation*}
    \lVert e _z \rVert _\Omega \leq \sup _{ \lambda \in \mathfrak{H} (\Omega) ^\perp \setminus \{ 0 \} } \frac{ K _{\mathcal{A}} ( z, \varphi ) }{ \lVert \lambda \rVert _\Omega } + K _{\mathcal{B}} (z) \sup _{ \lambda \in \mathfrak{H} (\Omega) ^\perp \setminus \{ 0 \} } \frac{ K _{\mathcal{B}} (\varphi) }{ \lVert \lambda \rVert _\Omega } ,
  \end{equation*}
  where $\varphi$ denotes the solution to the adjoint problem for
  $\lambda$.
\end{lemma}

\begin{proof}
  For any $ \lambda \in \mathfrak{H} (\Omega) ^\perp $, taking
  $ w = e _z $ in \eqref{e:dirac_adjoint} gives
  \begin{align*}
    ( e _z , \lambda ) _{ \mathcal{T} _h }
    &= \mathcal{A} ( e _z, \varphi ) \\
    &= \mathcal{A} ( e _z , \xi _\varphi ) \\
    &= \mathcal{A} ( \xi _z , \xi _\varphi ) + \mathcal{A} ( \Pi e _z , \xi _\varphi ),\\
    &\leq K _{\mathcal{A}} ( z , \varphi ) + \lvert \Pi e _z \rvert _{\mathcal{A}}  K _{\mathcal{B}} (\varphi) \\
    &\leq K _{\mathcal{A}} ( z , \varphi ) + K _{\mathcal{B}} (z) K _{\mathcal{B}} (\varphi) ,
  \end{align*}
  where the second line holds by Galerkin orthogonality, the third
  holds by \eqref{e:ez_decomp}, the fourth holds by \eqref{e:KA} and
  \eqref{e:KB_wz}, and the last holds by \eqref{e:Pi_ez}. Finally,
  since $ e _z \in \mathfrak{H} (\Omega) ^\perp $, we have
  \begin{equation*}
    \lVert e _z \rVert _\Omega = \sup _{ \lambda \in \mathfrak{H} (\Omega) ^\perp \setminus \{ 0 \} } \frac{ ( e _z, \lambda ) _\Omega }{ \lVert \lambda \rVert _\Omega } ,
  \end{equation*}
  noting that $ ( \cdot , \cdot ) _\Omega $ and
  $ ( \cdot , \cdot ) _{ \mathcal{T} _h } $ coincide.
\end{proof}

\subsection{Approximation theory: determining $ K _{ \mathcal{A} } $ and $ K _{ \mathcal{B} } $}
\label{s:KA_KB}

We now turn to the ``meat'' of our error analysis, which consists of
proving the inequalities stated in \eqref{e:KA_KB} and expressing the
functionals $ K _{\mathcal{A}} $ and $ K _{\mathcal{B}} $ that satisfy
them. To do so, we begin with the same set of assumptions on
$ \mathcal{T} _h $ as in \citep[Section
2.3]{CaCoPeSc2000}. Shape-regular simplicial meshes are included, but
the assumptions are broad enough to allow meshes with various cell
types (including curved elements) and hanging nodes.

\begin{assumption}
  \label{a:mesh}
  Assume that each $ K \in \mathcal{T} _h $ satisfies the following
  conditions:
  \begin{enumerate}[label=(\roman*)]
  \item It is affine-equivalent to one of finitely many reference
    cells.

  \item Its diameter is comparable to that of its neighboring
    cells. \label{a:hK_hK'}

  \item It contains a ball of comparable diameter (\emph{shape
      regularity}).
  \end{enumerate}
\end{assumption}

Denote the diameter of $ K \in \mathcal{T} _h $ by $ h _K $, and let
$ h \coloneqq \max _{ K \in \mathcal{T} _h } h _K $. In the estimates
to follow, we use the typical notation $ \lesssim $ to denote
inequality up to a multiplicative constant that is independent of the
arguments. For example, the \emph{Bramble--Hilbert lemma} allows us to
approximate
\begin{equation*}
  \lvert \xi _z \rvert _{ \mu , K } \lesssim h _K ^{ s - \mu } \lvert z \rvert _{ s, K } , \quad \forall K \in \mathcal{T} _h ,
\end{equation*}
for all $ z \in H ^s \Lambda (\Omega) $ with
$ \mu \leq s \leq r + 1 $, while combining this with the scaled
\emph{trace inequality} gives
\begin{equation*}
  \lVert \xi _z ^{\mathrm{tan}} \rVert _{ \partial K } \lesssim h _K ^{ s - 1/2 } \lvert z \rvert _{ s, K } , \qquad \lVert \xi _z ^{\mathrm{nor}} \rVert _{ \partial K } \lesssim h _K ^{ s - 1/2 } \lvert z \rvert _{ s, K } , \qquad \forall K \in \mathcal{T} _h .
\end{equation*}
These are standard results of finite element approximation theory, and
we will frequently employ them in the subsequent
analysis.

Finally, we assume that the penalty parameters are
$ \mathcal{O} (1) $. While the analysis might be adapted to
$ \mathcal{O} (h) $ or $ \mathcal{O} ( h ^{-1} ) $ or other scaling,
as it is in \citep{CaCoPeSc2000}, experiments show that this leads to
suboptimal convergence for the equal-order spaces. We also note that
$ \mathcal{O} (1) $ penalties have the effect of penalizing tangential
and normal jumps comparably, whereas $ \mathcal{O} (h) $ and
$ \mathcal{O} ( h ^{-1} ) $ penalties favor tangential continuity at
the expense of normal continuity, or vice versa, as
$ h \rightarrow 0 $.

\begin{assumption}
  \label{a:O(1)}
  Assume that
  $ \lvert \alpha \rvert , \lvert \beta \rvert \lesssim 1 $.
\end{assumption}

\begin{remark}
  \label{r:O(1)}
  It follows that $ 1/ \avg{ \alpha } $ and $ 1/ \avg{ \beta } $ are
  also $ \mathcal{O} (1) $: since $ \alpha ^k $ and $ \beta ^k $ have
  the same sign on all of $ \partial \mathcal{T} _h $ for each $k$,
  there is no possibility of cancellation in the
  denominator. Therefore, we have the equivalence
  \begin{equation}
    \label{e:A_equiv_L2}
    \lvert w \rvert _{\mathcal{A}} ^2 \asymp \bigl\lVert \jump{ w ^{\mathrm{nor}} } \bigr\rVert _{ \partial \mathcal{T} _h } ^2 + \bigl\lVert \jump{ w ^{\mathrm{tan}} } \bigr\rVert _{ \partial \mathcal{T} _h } ^2 ,
  \end{equation}
  where the notation $ \asymp $ means that both $ \lesssim $ and
  $ \gtrsim $ hold. We also obtain
  $ \bigl\lvert \jump{ \alpha } / \avg{ \alpha } \bigr\rvert ,
  \bigl\lvert \jump{ \beta } / \avg{ \beta } \bigr\rvert \leq 1 $.
\end{remark}

\begin{lemma}
  \label{l:KA}
  Suppose that
  \begin{equation*}
    z \in \prod _{ K \in \mathcal{T} _h } H \Lambda (K) \cap H ^\ast
    \Lambda (K) \cap H ^s \Lambda (K) , \qquad
    w \in \prod _{ K \in \mathcal{T} _h } H \Lambda (K) \cap H ^\ast
    \Lambda (K) \cap H ^t \Lambda (K) ,
  \end{equation*}
  where $ \frac{1}{2} < s, t \leq r + 1 $. Then there exists a
  constant $C$ such that \eqref{e:KA} holds with
  \begin{equation*}
    K _{\mathcal{A}} ( z, w ) = C h ^{ s + t -1 } \lvert z \rvert _{ s, \mathcal{T} _h } \lvert w \rvert _{ t, \mathcal{T} _h } .
  \end{equation*}
\end{lemma}

\begin{proof}
  First, observe that the claimed expression for
  $ K _{\mathcal{A}} ( z, w ) $ is unchanged if we replace $w$ by
  $ w ^- - w ^+ $ or $ w ^+ - w ^- $. Therefore, we may replace
  $ \mathcal{A} ( \cdot , \cdot ) $ by $ a ( \cdot , \cdot ) $ in
  \eqref{e:KA}, and it suffices to show
  $ \bigl\lvert a ( \xi _z, \xi _w ) \bigr\rvert \leq K _{\mathcal{A}}
  ( z, w ) $. Using integration by parts and \cref{p:jump-avg}, we have
  \begin{multline*}
    a ( \xi _z , \xi _w ) = ( \mathrm{D} \xi _z  , \xi _w  ) _{ \mathcal{T} _h } - \bigl\langle \jump{ \xi _z ^{\mathrm{tan}} }, \avg{ \xi _w ^{\mathrm{nor}} } \bigr\rangle _{ \partial \mathcal{T} _h } + \bigl\langle \jump{ \xi _z ^{\mathrm{nor}} } , \avg{ \xi _w ^{\mathrm{tan}} } \bigr\rangle _{ \partial \mathcal{T} _h } \\
    + \biggl\langle \frac{ 1 }{ \avg{\alpha} } \jump{ \xi _z ^{\mathrm{nor}} } + \frac{ \jump{\alpha} }{ \avg{\alpha} } \jump{ \xi _z ^{\mathrm{tan}} } , \jump{ \xi _w ^{\mathrm{nor}} } \biggr\rangle _{ \partial \mathcal{T} _h } - \biggl\langle \frac{ 1 }{ \avg{\beta} } \jump{ \xi _z ^{\mathrm{tan}} } + \frac{ \jump{\beta} }{ \avg{\beta} } \jump{ \xi _z ^{\mathrm{nor}} } , \jump{ \xi _w ^{\mathrm{tan}} } \biggr\rangle _{ \partial \mathcal{T} _h } .
  \end{multline*}
  We now apply the triangle inequality and bound these terms one at a
  time.

  The trace terms are all controlled similarly. For example,
  \begin{align*}
    \Bigl\lvert \bigl\langle \jump{ \xi _z ^{\mathrm{tan}} }, \avg{ \xi _w ^{\mathrm{nor}} } \bigr\rangle _{ \partial \mathcal{T} _h } \Bigr\rvert
    &\leq \bigl\lVert \jump{ \xi _z ^{\mathrm{tan}} }\bigr\rVert _{ \partial \mathcal{T} _h } \bigl\lVert \avg{ \xi _w ^{\mathrm{nor}} } \bigr\rVert _{ \partial \mathcal{T} _h }\\
    &\leq \lVert \xi _z ^{\mathrm{tan}} \rVert _{ \partial \mathcal{T} _h } \lVert \xi _w ^{\mathrm{nor}} \rVert _{ \partial \mathcal{T} _h } \\
    &\lesssim h ^{ s - 1/2 } \lvert z \rvert _{ s, \mathcal{T} _h } \, h ^{ t - 1/2 } \lvert w \rvert _{ t, \mathcal{T} _h } ,
  \end{align*}
  where the first inequality is by Cauchy--Schwarz, the second is by
  \cref{c:pythag}, and the last is by the scaled trace inequality
  together with Bramble--Hilbert. The remaining trace terms can be
  controlled in the same way, with the $ \mathcal{O} (1) $ penalty
  coefficients absorbed into the multiplicative constant.

  It thus remains to control
  $ ( \mathrm{D} \xi _z , \xi _w ) _{ \mathcal{T} _h } $. Letting
  $ \mu \coloneqq \min \{ 1, s \} $, we have
  \begin{align*}
    \bigl\lvert ( \mathrm{D} \xi _z , \xi _w ) _{ \mathcal{T} _h } \bigr\rvert
    &\leq \lVert \mathrm{D} \xi _z \rVert _{ \mu -1 , \mathcal{T} _h } \lVert \xi _w \rVert _{ 1 - \mu , \mathcal{T} _h } \\
    &\lesssim \lVert \xi _z \rVert _{ \mu , \mathcal{T} _h } \lVert \xi _w \rVert _{ 1 - \mu , \mathcal{T} _h } \\
    &\lesssim h ^{ s - \mu } \lvert z \rvert _{ s, \mathcal{T} _h } h ^{ t + \mu - 1 } \lvert w \rvert _{ t , \mathcal{T} _h } .
  \end{align*}
  Here, the first line holds by definition of the nonpositive Sobolev
  dual norm $ \lVert \cdot \rVert _{ \mu -1 , \mathcal{T} _h } $; in
  the case $ \mu = 1 $, this is just the $ L ^2 $ Cauchy--Schwarz
  inequality. The second line uses the fact that $ \mathrm{D} $ is a
  first-order partial differential operator and thus is bounded from
  $ H ^\mu \Lambda ( K ) $ to $ H ^{ \mu -1 } \Lambda (K) $ for each
  $ K \in \mathcal{T} _h $. The last line applies the Bramble--Hilbert
  lemma, noting that $ \mu \leq s $ and $ 1 - \mu < \frac{1}{2} < t $.
\end{proof}

\begin{lemma}
  \label{l:KB}
  Suppose
  $z \in \prod _{ K \in \mathcal{T} _h } H \Lambda (K) \cap H ^\ast
  \Lambda (K) \cap H ^s \Lambda (K) $, where
  $ \frac{1}{2} < s \leq r + 1 $, and $ w _h \in W _h $. Then there
  exists a constant $C$ such that \eqref{e:KB_zw} and \eqref{e:KB_wz}
  hold with
  \begin{equation*}
    K _{\mathcal{B}} (z) = C h ^{ s - 1/2 } \lvert z \rvert _{ s , \mathcal{T} _h } .
  \end{equation*}
\end{lemma}

\begin{proof}
  As in the proof of \cref{l:KA}, we begin by observing that
  $ \lvert w _h \rvert _{\mathcal{A}} $ and the claimed form of
  $ K _{\mathcal{B}} (z) $ are both unchanged by flipping the signs of
  the even- or odd-degree components. Hence, we may replace
  $ \mathcal{A} ( \cdot , \cdot ) $ by $ a ( \cdot , \cdot ) $, and it
  suffices to prove
  \begin{equation*}
    \bigl\lvert a ( \xi _z, w _h ) \bigr\rvert \leq K _{\mathcal{B}} (z) \lvert w _h \rvert _{\mathcal{A}} .
  \end{equation*}
  Since $ a ( \cdot , \cdot ) $ is symmetric, this suffices to prove
  both \eqref{e:KB_zw} and \eqref{e:KB_wz}. Similarly to the proof of
  \cref{l:KA}, we integrate by parts and apply \cref{p:jump-avg} to
  write
  \begin{multline*}
    a ( \xi _z , w _h ) = ( \xi _z  , \mathrm{D} w _h  ) _{ \mathcal{T} _h } + \bigl\langle \avg{ \xi _z ^{\mathrm{tan}} }, \jump{ w _h ^{\mathrm{nor}} } \bigr\rangle _{ \partial \mathcal{T} _h } - \bigl\langle \avg{ \xi _z ^{\mathrm{nor}} } , \jump{ w _h ^{\mathrm{tan}} } \bigr\rangle _{ \partial \mathcal{T} _h } \\
    + \biggl\langle \frac{ 1 }{ \avg{\alpha} } \jump{ \xi _z ^{\mathrm{nor}} } + \frac{ \jump{\alpha} }{ \avg{\alpha} } \jump{ \xi _z ^{\mathrm{tan}} } , \jump{ w _h ^{\mathrm{nor}} } \biggr\rangle _{ \partial \mathcal{T} _h } - \biggl\langle \frac{ 1 }{ \avg{\beta} } \jump{ \xi _z ^{\mathrm{tan}} } + \frac{ \jump{\beta} }{ \avg{\beta} } \jump{ \xi _z ^{\mathrm{nor}} } , \jump{ w _h ^{\mathrm{tan}} } \biggr\rangle _{ \partial \mathcal{T} _h } .
  \end{multline*}
  The first term vanishes since $ \xi _z $ is orthogonal to
  $ \mathrm{D} w _h \in W _h $. Collecting the remaining terms gives
  \begin{align*}
    a ( \xi _z, w _h ) &= \biggl\langle \avg{\xi _z ^{\mathrm{tan}}} + \frac{ 1 }{ \avg{\alpha} } \jump{ \xi _z ^{\mathrm{nor}} } + \frac{ \jump{\alpha} }{ \avg{\alpha} } \jump{ \xi _z ^{\mathrm{tan}} } , \jump{ w _h ^{\mathrm{nor}} } \biggr\rangle _{ \partial \mathcal{T} _h } \\
    &\quad - \biggl\langle \avg{ \xi _z ^{\mathrm{nor}} } + \frac{ 1 }{ \avg{\beta} } \jump{ \xi _z ^{\mathrm{tan}} } + \frac{ \jump{\beta} }{ \avg{\beta} } \jump{ \xi _z ^{\mathrm{nor}} } , \jump{ w _h ^{\mathrm{tan}} } \biggr\rangle _{ \partial \mathcal{T} _h } .
  \end{align*}
  Applying Cauchy--Schwarz and the triangle inequality, and absorbing
  the penalty coefficients into the multiplicative constant, gives
  \begin{align*}
    \bigl\lvert a ( \xi _z, w _h ) \bigr\rvert
    &\lesssim \Bigl( \bigl\lVert \avg{\xi _z ^{\mathrm{tan}}} \bigr\rVert _{ \partial \mathcal{T} _h } + \bigl\lVert \jump{\xi _z ^{\mathrm{nor}}} \bigr\rVert _{ \partial \mathcal{T} _h } + \bigl\lVert \jump{\xi _z ^{\mathrm{tan}}} \bigr\rVert _{ \partial \mathcal{T} _h } \Bigr) \bigl\lVert \jump{ w _h ^{\mathrm{nor}} } \bigr\rVert _{ \partial \mathcal{T} _h } \\
    &\quad + \Bigl( \bigl\lVert \avg{\xi _z ^{\mathrm{nor}}} \bigr\rVert _{ \partial \mathcal{T} _h } + \bigl\lVert \jump{\xi _z ^{\mathrm{tan}}} \bigr\rVert _{ \partial \mathcal{T} _h } + \bigl\lVert \jump{\xi _z ^{\mathrm{nor}}} \bigr\rVert _{ \partial \mathcal{T} _h } \Bigr) \bigl\lVert \jump{ w _h ^{\mathrm{tan}} } \bigr\rVert _{ \partial \mathcal{T} _h } \\
    &\lesssim \Bigl( \lVert \xi _z ^{\mathrm{tan}} \rVert ^2 _{ \partial \mathcal{T} _h } + \lVert \xi _z ^{\mathrm{nor}} \rVert ^2 _{ \partial \mathcal{T} _h } \Bigr) ^{ 1/2 } \Bigl( \bigl\lVert \jump{ w _h ^{\mathrm{nor}} } \bigr\rVert _{ \partial \mathcal{T} _h } ^2 + \bigl\lVert \jump{ w _h ^{\mathrm{tan}} } \bigr\rVert _{ \partial \mathcal{T} _h } ^2 \Bigr) ^{ 1/2 } \\
    &\lesssim h ^{ s - 1/2 } \lvert z \rvert _{ s, \mathcal{T} _h } \lvert w _h \rvert _{\mathcal{A}} ,
  \end{align*}
  where the second estimate is by discrete Cauchy--Schwarz and
  \cref{c:pythag}, and the last is by the scaled trace inequality
  together with Bramble--Hilbert and \eqref{e:A_equiv_L2}.
\end{proof}

\subsection{Final error estimates}

We now have the necessary ingredients to state the error estimates
implied by our analysis, which we do by combining the abstract and
approximation-theoretic results.

\begin{theorem}
  \label{t:dirac_error}
  Let $ ( z, p ) $ solve the Hodge--Dirac problem on $\Omega$, and let
  $ ( z _h , p _h ) $ be the degree-$r$ HDG solution in the form
  \eqref{e:dirac_dg}. Suppose $\Omega$ is $t$-regular and
  $ z \in H ^s \Lambda (\Omega) $, where
  $ \frac{1}{2} < t \leq s \leq r + 1 $. Under the preceding
  assumptions, the error satisfies the estimates
  \begin{equation*}
    \lvert z - z _h \rvert _{\mathcal{A}} \lesssim h ^{ s - 1/2 } \lvert z \rvert _{s, \Omega} , \qquad \lVert z - z _h \rVert _\Omega \lesssim h ^{ s + t -1 } \lvert z \rvert _{s, \Omega} .
  \end{equation*}
\end{theorem}

\begin{proof}
  The $\mathcal{A}$-seminorm error estimate follows from the abstract
  estimate in \cref{l:dirac_error_A}, where \cref{l:KA,l:KB} give
  \begin{equation*}
    K _{\mathcal{A}} ( z, z ) = C h ^{ 2 s -1 } \lvert z \rvert _{s, \Omega} ^2 , \qquad K _{\mathcal{B}} (z) = C h ^{ s - 1/2 } \lvert z \rvert _{s, \Omega} .
  \end{equation*}
  For the $ L ^2 $-norm estimate, a solution to the adjoint problem
  satisfies $ \varphi \in H ^t \Lambda (\Omega) $ by $t$-regularity,
  so \cref{l:KA,l:KB} give
  \begin{equation*}
    K _{\mathcal{A}} ( z, \varphi ) = C h ^{ s + t -1 } \lvert z \rvert _{s, \Omega} \lvert \varphi \rvert _{t, \Omega} , \qquad K _{\mathcal{B}} ( \varphi ) = C h ^{ t - 1/2 } \lvert \varphi \rvert _{t, \Omega} .
  \end{equation*}
  Moreover, by $t$-regularity and Poincar\'e's inequality, we have
  $ \lvert \varphi \rvert _{t, \Omega} \lesssim \lVert \mathrm{D} \varphi \rVert
  _\Omega = \lVert \lambda \rVert _\Omega $ for all
  $ \lambda \in \mathfrak{H} (\Omega) ^\perp $. Hence, the
  $ L ^2 $-norm error estimate follows from the abstract estimate in
  \cref{l:dirac_error_l2}.
\end{proof}

We conclude this section by analyzing the error of the approximate
traces $ \widehat{ z } _h ^{\mathrm{tan}} $ and
$ \widehat{ z } _h ^{\mathrm{nor}} $, which we estimate in the scaled
norm
$ \lVERT \cdot \rVERT _{ \partial \mathcal{T} _h } ^2 \coloneqq \sum
_{ K \in \mathcal{T} _h } h _K \lVert \cdot \rVert _{ \partial K } ^2
$. This norm has the same scaling properties as
$ \lVert \cdot \rVert _{ \mathcal{T} _h } $, and applying the scaled
trace inequality term-by-term gives
\begin{equation*}
  \lVERT w ^{\mathrm{tan}} \rVERT _{ \partial \mathcal{T} _h } \lesssim \lVert w \rVert _{ \mathcal{T} _h } + h ^t \lvert w \rvert _{ t, \mathcal{T} _h } , \qquad \lVERT w ^{\mathrm{nor}} \rVERT _{ \partial \mathcal{T} _h } \lesssim \lVert w \rVert _{ \mathcal{T} _h } + h ^t \lvert w \rvert _{ t, \mathcal{T} _h } , \qquad \forall w \in W .
\end{equation*}
Before proceeding to the trace estimates, we first prove a lemma
showing that \cref{c:pythag} allows us to control jumps and averages
by ordinary traces in
$ \lVERT \cdot \rVERT _{ \partial \mathcal{T} _h } $, just as we have
done in $ \lVert \cdot \rVert _{ \partial \mathcal{T} _h } $.

\begin{lemma}
  \label{l:pythag_triple_norm}
  If $ \mathcal{T} _h $ satisfies
  \hyperref[a:hK_hK']{\cref*{a:mesh}\ref*{a:hK_hK'}}, then for all
  $ \widehat{ w } ^{\mathrm{tan}} , \widehat{ w } ^{\mathrm{nor}} \in
  L ^2 \Lambda ( \partial \mathcal{T} _h ) $,
  \begin{align*}
    \lVERT \widehat{ w } ^{\mathrm{tan}} \rVERT ^2 _{ \partial \mathcal{T} _h } &\asymp \biglVERT \avg{ \widehat{ w } ^{\mathrm{tan}} } \bigrVERT ^2 _{ \partial \mathcal{T} _h } + \biglVERT \jump{ \widehat{ w } ^{\mathrm{tan}} } \bigrVERT _{ \partial \mathcal{T} _h } ^2 , \\
    \lVERT \widehat{ w } ^{\mathrm{nor}} \rVERT ^2 _{ \partial \mathcal{T} _h } &\asymp \biglVERT \avg{ \widehat{ w } ^{\mathrm{nor}} } \bigrVERT ^2 _{ \partial \mathcal{T} _h } + \biglVERT \jump{ \widehat{ w } ^{\mathrm{nor}} } \bigrVERT _{ \partial \mathcal{T} _h } ^2 .
  \end{align*}
\end{lemma}

\begin{proof}
  Both statements are proved in the same way, so we only consider the
  tangential-trace case. First, applying \cref{c:pythag} to the
  restriction $ \widehat{ w } ^{\mathrm{tan}} \rvert _{ e ^\pm } $ with
  $ e = \partial K ^+ \cap \partial K ^- $, we obtain
  \begin{align*}
    \lVert \widehat{ w } ^{\mathrm{tan}} \rVert _{ e ^+ } ^2 + \lVert \widehat{ w } ^{\mathrm{tan}} \rVert _{ e ^- } ^2
    &= \bigl\lVert \avg{ \widehat{ w } ^{\mathrm{tan}} } \bigr\rVert_{ e ^+ }  ^2 + \bigl\lVert \jump{ \widehat{ w } ^{\mathrm{tan}} } \bigr\rVert_{ e ^+ } ^2 \\
    &\quad + \bigl\lVert \avg{ \widehat{ w } ^{\mathrm{tan}} } \bigr\rVert_{ e ^- } ^2 + \bigl\lVert \jump{ \widehat{ w } ^{\mathrm{tan}} } \bigr\rVert_{ e ^- }  ^2 .
  \end{align*}
  Since $ h _{K ^+} \asymp h _{ K ^- } $ by
  \hyperref[a:hK_hK']{\cref*{a:mesh}\ref*{a:hK_hK'}}, this implies
  \begin{align*}
    h _{ K ^+ } \lVert \widehat{ w } ^{\mathrm{tan}} \rVert _{ e ^+ } ^2 + h _{ K ^- } \lVert \widehat{ w } ^{\mathrm{tan}} \rVert _{ e ^- } ^2
    &\asymp h _{ K ^+ } \Bigl( \bigl\lVert \avg{ \widehat{ w } ^{\mathrm{tan}} } \bigr\rVert_{ e ^+ }  ^2 + \bigl\lVert \jump{ \widehat{ w } ^{\mathrm{tan}} } \bigr\rVert_{ e ^+ } ^2 \Bigr) \\
    &\quad + h _{ K ^- } \Bigl( \bigl\lVert \avg{ \widehat{ w } ^{\mathrm{tan}} } \bigr\rVert_{ e ^- }  ^2 + \bigl\lVert \jump{ \widehat{ w } ^{\mathrm{tan}} } \bigr\rVert_{ e ^- } ^2 \Bigr).
  \end{align*}
  On the other hand, for $ e = \partial K \cap \partial \Omega $, we
  immediately have
  \begin{equation*}
    h _K \lVert \widehat{ w } ^{\mathrm{tan}} \rVert _e ^2 = h _K \Bigl( \bigl\lVert \avg{ \widehat{ w } ^{\mathrm{tan}} } \bigr\rVert_e ^2 + \bigl\lVert \jump{ \widehat{ w } ^{\mathrm{tan}} } \bigr\rVert_e ^2 \Bigr) .
  \end{equation*}
  Therefore, summing over all $ e \subset \partial \mathcal{T} _h $
  gives
  \begin{align*}
    \lVERT \widehat{ w } ^{\mathrm{tan}} \rVERT _{ \partial \mathcal{T} _h } ^2
    &= \sum _{ K \in \mathcal{T} _h } \sum _{ e \subset \partial K } h _K \lVert \widehat{ w } ^{\mathrm{tan}} \rVert _e ^2 \\
    &\asymp \sum _{ K \in \mathcal{T} _h } \sum _{ e \subset \partial K } h _K \Bigl( \bigl\lVert \avg{ \widehat{ w } ^{\mathrm{tan}} } \bigr\rVert_e ^2 + \bigl\lVert \jump{ \widehat{ w } ^{\mathrm{tan}} } \bigr\rVert_e ^2 \Bigr) \\
    &= \biglVERT \avg{ \widehat{ w } ^{\mathrm{tan}} } \bigrVERT ^2 _{ \partial \mathcal{T} _h } + \biglVERT \jump{ \widehat{ w } ^{\mathrm{tan}} } \bigrVERT _{ \partial \mathcal{T} _h } ^2 ,
  \end{align*}
  as claimed.
\end{proof}

\begin{theorem}
  \label{t:dirac_trace_error}
  Under the hypotheses of \cref{t:dirac_error}, the approximate traces
  satisfy
  \begin{equation*}
    \lVERT z ^{\mathrm{tan}} - \widehat{ z } _h ^{\mathrm{tan}} \rVERT _{ \partial \mathcal{T} _h } + \lVERT z ^{\mathrm{nor}} - \widehat{ z } _h ^{\mathrm{nor}} \rVERT _{ \partial \mathcal{T} _h } \lesssim h ^{ s + t -1 } \lvert z \rvert _{s, \Omega}.
  \end{equation*}
\end{theorem}

\begin{proof}
  The errors
  $ \widehat{ e } _z ^{\mathrm{tan}} \coloneqq z ^{\mathrm{tan}} -
  \widehat{ z } _h ^{\mathrm{tan}} $ and
  $ \widehat{ e } _z ^{\mathrm{nor}} \coloneqq z ^{\mathrm{nor}} -
  \widehat{ z } _h ^{\mathrm{nor}} $ satisfy
  \begin{align*}
    \widehat{ e } _z ^{\mathrm{tan}}
    &= \avg{ e _z ^{\mathrm{tan}} } + \frac{ 1 }{ \avg{\alpha} } \jump{ e _z ^{\mathrm{nor}} } + \frac{ \jump{\alpha} }{ \avg{ \alpha } } \jump{ e _z ^{\mathrm{tan}} } , \\
    \widehat{ e } _z ^{\mathrm{nor}}
    &= \avg{ e _z ^{\mathrm{nor}} } + \frac{ 1 }{ \avg{\beta} } \jump{ e _z ^{\mathrm{tan}} } + \frac{ \jump{\beta} }{ \avg{ \beta } } \jump{ e _z ^{\mathrm{nor}} } ,
  \end{align*}
  by \cref{l:dg_traces} together with
  $ \jump{ z ^{\mathrm{tan}} } = 0 $ and
  $ \jump{ z ^{\mathrm{nor}} } = 0 $. Taking
  $ \lVERT \cdot \rVERT _{ \partial \mathcal{T} _h } $, applying the
  triangle inequality, and absorbing the penalty coefficients into the
  multiplicative constant, we obtain
  \begin{align*}
    \lVERT \widehat{ e } _z ^{\mathrm{tan}} \rVERT _{ \partial \mathcal{T} _h }
    + \lVERT \widehat{ e } _z ^{\mathrm{nor}} \rVERT _{ \partial \mathcal{T} _h }
    &\lesssim \biglVERT \avg{ e _z ^{\mathrm{tan}} } \bigrVERT _{ \partial \mathcal{T} _h } + \biglVERT \avg{ e _z ^{\mathrm{nor}} } \bigrVERT _{ \partial \mathcal{T} _h } + \biglVERT \jump{ e _z ^{\mathrm{nor}} } \bigrVERT _{ \partial \mathcal{T} _h } + \biglVERT \jump{ e _z ^{\mathrm{tan}} } \bigrVERT _{ \partial \mathcal{T} _h } \\
    &\lesssim \lVERT e _z ^{\mathrm{tan}} \rVERT _{ \partial \mathcal{T} _h } + \lVERT e _z ^{\mathrm{nor}} \rVERT _{ \partial \mathcal{T} _h } ,
  \end{align*}
  where the second line is by \cref{l:pythag_triple_norm}. Now, by
  \eqref{e:ez_decomp} and the triangle inequality, we have
  \begin{equation*}
    \lVERT e _z ^{\mathrm{tan}} \rVERT _{ \partial \mathcal{T} _h } \leq \lVERT \xi  _z ^{\mathrm{tan}} \rVERT _{ \partial \mathcal{T} _h } + \lVERT \Pi e _z ^{\mathrm{tan}} \rVERT _{ \partial \mathcal{T} _h } .
  \end{equation*}
  As before, we can apply the scaled trace inequality and
  Bramble--Hilbert to get
  \begin{equation*}
    \lVERT \xi  _z ^{\mathrm{tan}} \rVERT _{ \partial \mathcal{T} _h } \lesssim h ^s \lvert z \rvert _{s, \Omega} ,
  \end{equation*}
  and since $ \Pi e _z \in W _h $, a scaling argument and
  \cref{t:dirac_error} give
  \begin{equation*}
    \lVERT \Pi e _z ^{\mathrm{tan}} \rVERT _{ \partial \mathcal{T} _h }
    \lesssim \lVert \Pi e _z \rVert _{ \mathcal{T} _h } \leq \lVert e _z \rVert _\Omega \lesssim h ^{ s + t -1 } \lvert z \rvert _{s, \Omega} .
  \end{equation*}
  This last term dominates the previous one since $ t \leq 1 $, so
  adding them gives
  $ \lVERT e _z ^{\mathrm{tan}} \rVERT _{ \partial \mathcal{T} _h }
  \lesssim h ^{ s + t -1 } \lvert z \rvert _{s, \Omega} $. The same
  bound can be obtained for
  $ \lVERT e _z ^{\mathrm{nor}} \rVERT _{ \partial \mathcal{T} _h } $,
  which completes the proof.
\end{proof}

\begin{remark}
  We close this section by noting the specific rates of convergence
  indicated by our results in two different regularity scenarios:
  \begin{itemize}
  \item \emph{maximum regularity}: If $\Omega$ is $1$-regular (e.g.,
    convex) and $ z \in H ^{ r + 1 } \Lambda (\Omega) $, then
    \begin{equation*}
      \lvert z - z _h \rvert _{\mathcal{A}} \lesssim h ^{ r + 1/2 } \lvert z \rvert _{r+1, \Omega} , \qquad \lVert z - z _h \rVert _\Omega \lesssim h ^{ r + 1 } \lvert z \rvert _{r+1, \Omega} ,
    \end{equation*}
    and the approximate traces satisfy
    \begin{equation*}
      \lVERT z ^{\mathrm{tan}} - \widehat{ z } _h ^{\mathrm{tan}} \rVERT _{ \partial \mathcal{T} _h } + \lVERT z ^{\mathrm{nor}} - \widehat{ z } _h ^{\mathrm{nor}} \rVERT _{ \partial \mathcal{T} _h } \lesssim h ^{ r+1 } \lvert z \rvert _{r+1, \Omega} .
    \end{equation*}

  \item \emph{minimum regularity}: If $\Omega$ is $t$-regular with
    $ \frac{1}{2} < t \leq 1 $ and we merely have
    $ z \in H ^t \Lambda (\Omega) $, then
    \begin{equation*}
      \lvert z - z _h \rvert _{\mathcal{A}} \lesssim h ^{ t - 1/2 } \lvert z \rvert _{t, \Omega} , \qquad \lVert z - z _h \rVert _\Omega \lesssim h ^{ 2 t - 1 } \lvert z \rvert _{t, \Omega} ,
    \end{equation*}
    and the approximate traces satisfy
    \begin{equation*}
      \lVERT z ^{\mathrm{tan}} - \widehat{ z } _h ^{\mathrm{tan}} \rVERT _{ \partial \mathcal{T} _h } + \lVERT z ^{\mathrm{nor}} - \widehat{ z } _h ^{\mathrm{nor}} \rVERT _{ \partial \mathcal{T} _h } \lesssim h ^{ 2t - 1 } \lvert z \rvert _{t, \Omega} .
    \end{equation*}
    These worst-case rates are always positive but approach zero as
    $t$ approaches $ \frac{1}{2} $. (This gives another reason why we
    estimate the approximate traces in the
    $ \lVERT \cdot \rVERT _{ \partial \mathcal{T} _h } $ norm: without
    accounting for scaling, we might appear to have negative rates of
    convergence.)
  \end{itemize}
\end{remark}

\subsection{Improved estimates in dimension two}

Observe that we always have
\begin{equation*}
  H \Lambda ^0 (\Omega) = H ^1 \Lambda ^0 (\Omega) , \qquad H ^\ast \Lambda ^n (\Omega) = H ^1 \Lambda ^n (\Omega) ,
\end{equation*}
even if $\Omega$ fails to be $1$-regular at other degrees. When
$ n = 2 $, we can use this to obtain an improved estimate for
$ z _h ^1 $ and its approximate traces.

\begin{theorem}
  \label{t:dirac_error_n=2}
  In dimension $ n = 2 $, under the hypotheses of
  \cref{t:dirac_error}, we have the improved $1$-form error estimates
  \begin{equation*}
    \lVert z ^1 - z _h ^1 \rVert _\Omega \lesssim h ^s \lvert z \rvert _{ s, \Omega } , \qquad \lVERT z ^{1, \mathrm{tan}} - \widehat{ z } _h ^{1, \mathrm{tan}} \rVERT _{ \partial \mathcal{T} _h } + \lVERT z ^{1, \mathrm{nor}} - \widehat{ z } _h ^{1, \mathrm{nor}} \rVERT _{ \partial \mathcal{T} _h } \lesssim h ^s \lvert z \rvert _{ s, \Omega } .
  \end{equation*}
\end{theorem}

\begin{proof}
  We apply the duality argument in \cref{l:dirac_error_l2} to the case
  where $\lambda$ is restricted to be a $1$-form, obtaining
  \begin{align*}
    \lVert e _z ^1 \rVert _\Omega
    &= \sup _{ \lambda \in \mathfrak{H} ^1 (\Omega) ^\perp \setminus \{ 0 \} } \frac{ ( e _z, \lambda ) _\Omega }{ \lVert \lambda \rVert _\Omega } \\
    &\leq \sup _{ \lambda \in \mathfrak{H} ^1 (\Omega) ^\perp \setminus \{ 0 \} } \frac{ K _{\mathcal{A}} ( z, \varphi ) }{ \lVert \lambda \rVert _\Omega } + K _{\mathcal{B}} (z) \sup _{ \lambda \in \mathfrak{H} ^1 (\Omega) ^\perp \setminus \{ 0 \} } \frac{ K _{\mathcal{B}} (\varphi) }{ \lVert \lambda \rVert _\Omega } .
  \end{align*}
  For any such $\lambda$, the solution to the adjoint problem has only
  $0$- and $2$-form components, since
  $ \varphi ^1 \in \mathfrak{H} ^1 (\Omega) ^\perp $ and
  $ \mathrm{D} \varphi ^1 = 0 $ together imply $ \varphi ^1 = 0
  $. Thus,
  $ \varphi = \varphi ^0 \oplus \varphi ^2 \in H ^1 \Lambda (\Omega) $
  and
  $ \lvert \varphi \rvert _1 = \lVert \mathrm{D} \varphi \rVert
  _\Omega = \lVert \lambda \rVert _\Omega $. \Cref{l:KA,l:KB}
  therefore give
  \begin{equation*}
    K _{\mathcal{A}} ( z, \varphi ) = C h ^s \lvert z \rvert _{s, \Omega} \lvert \varphi \rvert _{1, \Omega} , \qquad K _{\mathcal{B}} ( z ) = C h ^{s - 1/2} \lvert z \rvert _{s, \Omega} , \qquad K _{\mathcal{B}} ( \varphi ) = C h ^{1/2} \lvert \varphi \rvert _{1, \Omega} ,
  \end{equation*}
  and substituting these above proves the first claim.

  For the approximate traces, we begin by writing
  \begin{align*}
    \widehat{ e } _z ^{1, \mathrm{tan}}
    &= \avg{ e _z ^{1, \mathrm{tan}} } + \frac{ 1 }{ \avg{\alpha ^1} } \jump{ e _z ^{2, \mathrm{nor}} } + \frac{ \jump{\alpha ^1 } }{ \avg{ \alpha ^1 } } \jump{ e _z ^{1, \mathrm{tan}} } , \\
    \widehat{ e } _z ^{1, \mathrm{nor}}
    &= \avg{ e _z ^{1, \mathrm{nor}} } + \frac{ 1 }{ \avg{\beta ^0} } \jump{ e _z ^{0, \mathrm{tan}} } + \frac{ \jump{\beta ^0} }{ \avg{ \beta ^0} } \jump{ e _z ^{1, \mathrm{nor}} } .
  \end{align*}
  In the proof of \cref{t:dirac_trace_error}, we used
  \cref{l:pythag_triple_norm} to control both the jumps and
  averages. Here we use
  $ \biglVERT \jump{ \cdot } \bigrVERT _{ \partial \mathcal{T} _h }
  \lesssim h ^{ 1/2 } \lvert \cdot \rvert _{\mathcal{A}} $ for the
  jumps and \cref{l:pythag_triple_norm} for the averages, thereby
  avoiding the possibly suboptimal $ L ^2 $-norm estimates for
  $ e _z ^0 $ and $ e _z ^2 $. Doing this gives
  \begin{equation*}
    \lVERT \widehat{ e } _z ^{1, \mathrm{tan}} \rVERT _{ \partial \mathcal{T} _h } +     \lVERT \widehat{ e } _z ^{1, \mathrm{nor}} \rVERT _{ \partial \mathcal{T} _h } \lesssim \lVERT e _z ^{1, \mathrm{tan}} \rVERT _{ \partial \mathcal{T} _h } + \lVERT e _z ^{1, \mathrm{nor}} \rVERT _{ \partial \mathcal{T} _h } + h ^{ 1/2 } \lvert e _z \rvert _{\mathcal{A}} .
  \end{equation*}
  Writing $ e _z ^1 = \xi _z ^1 + \Pi e _z ^1 $, the trace inequality and
  Bramble--Hilbert give
  \begin{equation*}
    \lVERT \xi _z ^{1, \mathrm{tan}} \rVERT _{ \partial \mathcal{T} _h } + \lVERT \xi _z ^{1, \mathrm{nor}} \rVERT _{ \partial \mathcal{T} _h } \lesssim h ^s \lvert z \rvert _{ s, \Omega } ,
  \end{equation*}
  while a scaling argument and the improved $ L ^2 $-norm estimate
  give
  \begin{equation*}
    \lVERT \Pi e _z ^{1, \mathrm{tan}} \rVERT _{ \partial \mathcal{T} _h } + \lVERT \Pi e _z ^{1, \mathrm{nor}} \rVERT _{ \partial \mathcal{T} _h } \lesssim \lVert \Pi e _z ^1 \rVert _{ \mathcal{T} _h } \leq \lVert e _z ^1 \rVert _\Omega \lesssim h ^s \lvert z \rvert _{ s , \Omega } ,
  \end{equation*}
  and \cref{t:dirac_error} gives
  \begin{equation*}
    h ^{ 1/2 } \lvert e _z \rvert _{\mathcal{A}} \lesssim h ^s \lvert z \rvert _{ s, \Omega } .
  \end{equation*}
  Substituting these above completes the proof.
\end{proof}

\section{HDG methods and error analysis for the Hodge--Laplace problem}
\label{s:laplace}

We now turn our attention to the Hodge--Laplace problem, the
Hodge--Dirac problem's better-known cousin. In fact, many of the
methods, results, and analyses we now present for the Hodge--Laplace
problem are straightforward adaptations or applications of those just
developed for Hodge--Dirac; much of the heavy lifting has already been
done in the preceding sections. We therefore proceed quickly through
the parts that are essentially repetitions of earlier arguments, while
drawing focus to the new ingredients that are specific to
Hodge--Laplace.

\subsection{Weak formulation}

Given $ f \in L ^2 \Lambda ^k (\Omega) $, the first-order form
\eqref{e:laplace_first-order} of the $k$-form Hodge--Laplace problem
may be written as
\begin{equation*}
  \mathrm{D} ( \sigma \oplus u \oplus \rho ) + p = \sigma \oplus f \oplus \rho .
\end{equation*}
This would be a Hodge--Dirac problem if not for the $\sigma$ and
$\rho$ terms on the right-hand side. The following HDG weak
formulation is essentially \eqref{e:dirac_ns} with
$ z _h = \sigma _h \oplus u _h \oplus \rho _h $,
$ \widehat{ z } _h ^{\mathrm{tan}} = \widehat{ \sigma } _h
^{\mathrm{tan}} \oplus \widehat{ u } _h ^{\mathrm{tan}} $, and
$ \widehat{ z } _h ^{\mathrm{nor}} = \widehat{ u } _h ^{\mathrm{nor}}
\oplus \widehat{ \rho } _h ^{\mathrm{nor}} $, and with the ``extra''
$\sigma _h $ and $\rho _h $ terms subtracted over to the left-hand
side. We use the same equal-order finite element spaces as for the
Hodge--Dirac problem, but here we only require forms of degrees
$ k -1 $, $k$, and $ k + 1 $, rather than forms of every degree.

\begin{definition}
  The HDG method is defined by the following weak problem: Find
  \begin{alignat*}{4}
    \text{(local variables)} \qquad \sigma _h &\in W _h ^{ k -1 } , \qquad & u _h &\in W _h ^k, \qquad & \rho _h &\in W _h ^{ k + 1 }, &\qquad \overline{p} _h &\in \overline{ \mathfrak{H} } _h ^k , \\
    \text{(global variables)} \qquad \overline{u} _h &\in \overline{ \mathfrak{H} } _h ^k , \qquad & p _h &\in \mathfrak{H} _h ^k, \qquad & \widehat{ \sigma } _h ^{\mathrm{tan}} &\in \widehat{ V } _h ^{ k -1, \mathrm{tan} } , \qquad & \widehat{ u } _h ^{\mathrm{tan}} &\in \widehat{ V } _h ^{k, \mathrm{tan}} ,
  \end{alignat*}
  satisfying
  \begin{subequations}
    \label{e:laplace_ns}
    \begin{alignat}{2}
      - ( \sigma _h , \tau _h ) _{ \mathcal{T} _h } + ( u _h , \mathrm{d} \tau _h ) _{ \mathcal{T} _h } - \langle \widehat{ u } _h ^{\mathrm{nor}} , \tau _h ^{\mathrm{tan}} \rangle _{ \partial \mathcal{T} _h } &= 0, \quad &\forall \tau _h &\in W _h ^{ k -1 } ,\\
      ( \sigma _h , \delta v _h ) _{ \mathcal{T} _h } + ( \rho _h , \mathrm{d} v _h ) _{ \mathcal{T} _h } + ( p _h - \overline{p} _h , v _h ) _{ \mathcal{T} _h } \\
      {}+ \langle \widehat{ \sigma } _h ^{\mathrm{tan}} , v _h ^{\mathrm{nor}} \rangle _{ \partial \mathcal{T} _h } - \langle \widehat{ \rho } _h ^{\mathrm{nor}} , v _h ^{\mathrm{tan}} \rangle _{ \partial \mathcal{T} _h } &= (f, v _h ) _{ \mathcal{T} _h } , \quad &\forall v _h &\in W _h ^k , \notag \\
      - ( \rho _h , \eta _h ) _{ \mathcal{T} _h } + ( u _h , \delta \eta _h ) _{ \mathcal{T} _h } + \langle \widehat{ u } _h ^{\mathrm{tan}} , \eta _h ^{\mathrm{nor}} \rangle _{ \partial \mathcal{T} _h } &= 0, \quad &\forall \eta _h &\in W _h ^{ k + 1 } ,\\
      ( \overline{u} _h - u _h , \overline{q} _h ) _{ \mathcal{T} _h } &= 0 , \quad &\forall \overline{q} _h &\in \overline{ \mathfrak{H} } _h ^k ,\\
      ( \overline{p} _h , \overline{v} _h ) _{ \mathcal{T} _h } &= 0, \quad &\forall \overline{v} _h &\in \overline{ \mathfrak{H} } _h ^k ,\\
      ( u _h , q _h ) _{ \mathcal{T} _h } &= 0, \quad &\forall q _h &\in \mathfrak{H} _h ^k ,\\
      \langle \widehat{ u } _h ^{\mathrm{nor}} , \widehat{ \tau } _h ^{\mathrm{tan}} \rangle _{ \partial \mathcal{T} _h } &= 0 , \quad &\forall \widehat{ \tau } _h ^{\mathrm{tan}} &\in \widehat{ V } _h ^{k-1, \mathrm{tan}} ,\\
      \langle \widehat{ \rho } _h ^{\mathrm{nor}} , \widehat{ v } _h ^{\mathrm{tan}} \rangle _{ \partial \mathcal{T} _h } &= 0 , \quad &\forall \widehat{ v } _h ^{\mathrm{tan}} &\in \widehat{ V } _h ^{k, \mathrm{tan}} ,
    \end{alignat}
  \end{subequations}
  where
  \begin{align*}
    \widehat{ u } _h ^{\mathrm{nor}} &\coloneqq u _h ^{\mathrm{nor}} - \alpha ^{ k -1 } ( \widehat{ \sigma } _h ^{\mathrm{tan}} - \sigma _h ^{\mathrm{tan}} ) ,\\
    \widehat{ \rho } _h ^{\mathrm{nor}} &\coloneqq \rho _h ^{\mathrm{nor}} - \alpha ^k ( \widehat{ u } _h ^{\mathrm{tan}} - u _h ^{\mathrm{tan}} ) .
  \end{align*}
\end{definition}

\begin{remark}
  This agrees with the HDG weak formulation introduced in
  \citet[Section 8]{AwFaGuSt2023}, with only minor differences in
  notation. As noted in the introduction, that paper did not establish
  existence and uniqueness results or conduct any analysis of such
  methods, with its primary focus being on hybridization of
  \emph{conforming} FEEC methods.
\end{remark}

The following is the ``Hodge--Laplace version'' of
\cref{l:dirac_properties}. We omit the proof, which is essentially
identical to the earlier one.

\begin{lemma}
  \label{l:laplace_properties}
  A solution to \eqref{e:laplace_ns} satisfies
  \begin{equation*}
    \overline{p} _h = 0 , \qquad \overline{u} _h = P _{ \overline{ \mathfrak{H} } _h } u _h , \qquad p _h = P _{ \mathfrak{H} _h } f , \qquad \jump{ \widehat{ u } _h ^{\mathrm{nor}} } = 0 , \qquad \jump{ \widehat{ \rho } _h ^{\mathrm{nor}} } = 0 .
  \end{equation*}
\end{lemma}

As we did for Hodge--Dirac, we may substitute the expressions for the
normal traces into \eqref{e:laplace_ns} and integrate by parts to get
\begin{subequations}
  \label{e:laplace}
  \begin{alignat}{2}
    - ( \sigma _h , \tau _h ) _{ \mathcal{T} _h } + ( \delta u _h , \tau _h ) _{ \mathcal{T} _h } +  \bigl\langle \alpha ^{ k -1 } ( \widehat{ \sigma } _h ^{\mathrm{tan}} - \sigma _h ^{\mathrm{tan}} ) , \tau _h ^{\mathrm{tan}} \bigr\rangle _{ \partial \mathcal{T} _h } &= 0, \quad &\forall \tau _h &\in W _h ^{ k -1 } , \label{e:laplace_tau} \\
    ( \sigma _h , \delta v _h ) _{ \mathcal{T} _h } + ( \delta \rho _h ,  v _h ) _{ \mathcal{T} _h } + ( p _h - \overline{p} _h , v _h ) _{ \mathcal{T} _h } \qquad \label{e:laplace_v} \\
    {}+ \langle \widehat{ \sigma } _h ^{\mathrm{tan}} , v _h ^{\mathrm{nor}} \rangle _{ \partial \mathcal{T} _h } + \bigl\langle \alpha ^k ( \widehat{ u } _h ^{\mathrm{tan}} - u _h ^{\mathrm{tan}} ) , v _h ^{\mathrm{tan}} \bigr\rangle _{ \partial \mathcal{T} _h } &= (f, v _h ) _{ \mathcal{T} _h } , \quad &\forall v _h &\in W _h ^k , \notag \\
    - ( \rho _h , \eta _h ) _{ \mathcal{T} _h } + ( u _h , \delta \eta _h ) _{ \mathcal{T} _h } + \langle \widehat{ u } _h ^{\mathrm{tan}} , \eta _h ^{\mathrm{nor}} \rangle _{ \partial \mathcal{T} _h } &= 0, \quad &\forall \eta _h &\in W _h ^{ k + 1 } , \label{e:laplace_eta} \\
    ( \overline{u} _h - u _h , \overline{q} _h ) _{ \mathcal{T} _h } &= 0 , \quad &\forall \overline{q} _h &\in \overline{ \mathfrak{H} } _h ^k , \label{e:laplace_qbar} \\
    ( \overline{p} _h , \overline{v} _h ) _{ \mathcal{T} _h } &= 0, \quad &\forall \overline{v} _h &\in \overline{ \mathfrak{H} } _h ^k , \label{e:laplace_vbar} \\
    ( u _h , q _h ) _{ \mathcal{T} _h } &= 0, \quad &\forall q _h &\in \mathfrak{H} _h ^k , \label{e:laplace_q} \\
    \bigl\langle u _h ^{\mathrm{nor}} - \alpha ^{ k -1 } ( \widehat{ \sigma } _h ^{\mathrm{tan}} - \sigma _h ^{\mathrm{tan}} ) , \widehat{ \tau } _h ^{\mathrm{tan}} \bigr\rangle _{ \partial \mathcal{T} _h } &= 0 , \quad &\forall \widehat{ \tau } _h ^{\mathrm{tan}} &\in \widehat{ V } _h ^{k-1, \mathrm{tan}} , \label{e:laplace_tautan} \\
    \bigl\langle \rho _h ^{\mathrm{nor}} - \alpha ^k ( \widehat{ u } _h ^{\mathrm{tan}} - u _h ^{\mathrm{tan}} ) , \widehat{ v } _h ^{\mathrm{tan}} \bigr\rangle _{ \partial \mathcal{T} _h } &= 0 , \quad &\forall \widehat{ v } _h ^{\mathrm{tan}} &\in \widehat{ V } _h ^{k, \mathrm{tan}} , \label{e:laplace_vtan}
  \end{alignat}
\end{subequations}
which is equivalent to \eqref{e:laplace_ns} but more obviously
symmetric.

\begin{remark}
  \label{r:laplace_alternative_hybridization}
  Since all four approximate traces are single-valued, there are
  equivalent hybridizations taking a different pair of global trace
  variables other than $ \widehat{ \sigma } _h ^{\mathrm{tan}} $,
  $ \widehat{ u } _h ^{\mathrm{tan}} $.

  For example, it is straightforward to show that we could instead
  take the global variables to be
  $ \widehat{ u } _h ^{\mathrm{nor}} $,
  $ \widehat{ \rho } _h ^{\mathrm{nor}} $. Since the local solvers
  would then have normal rather than tangential boundary conditions,
  the corresponding space of local harmonics would need to be
  $ W _h ^k \cap \prod _{ K \in \mathcal{T} _h } \mathfrak{H} ^k (K)
  $.

  A more intriguing alternative hybridization takes the global trace
  variables to be $ \widehat{ u } _h ^{\mathrm{nor}} $,
  $ \widehat{ u } _h ^{\mathrm{tan}} $. This does not require
  \emph{any} local harmonics, and static condensation produces a
  symmetric positive-definite Schur complement, modulo the nullspace
  of global harmonics; by contrast, the previous hybridizations
  generally yield saddle-point problems. This alternative
  hybridization and its properties are discussed further in
  \cref{s:laplace_uhat}.

  In the case of the two-dimensional vector Laplacian (i.e., $ k = 1 $
  and $ n = 2 $), \citet{CoNuSa2026} refer to these as the Type I,
  Type II, and Type III hybridizations, respectively.

  Finally, we note that there is a fourth possible hybridization
  taking $ \widehat{ \sigma } _h ^{\mathrm{tan}} $ and
  $ \widehat{ \rho } _h ^{\mathrm{nor}} $ as the global trace
  variables. However, the space of local harmonics would then need to
  consist of \emph{all} $k$-forms that are closed and coclosed,
  without boundary conditions on $ \partial \mathcal{T} _h $. This can
  be substantially larger than the other local harmonic spaces, so
  this fourth hybridization is considerably less practical for
  implementation than the others.
\end{remark}

\subsection{Existence and uniqueness}

Unlike with Hodge--Dirac, we do not have the freedom to choose which
of $ \alpha ^{ k -1 } $ and $ \alpha ^k $ is positive and which is
negative. The negative-definite terms
$ - ( \sigma _h , \tau _h ) _{ \mathcal{T} _h } $ and
$ - ( \rho _h , \eta _h ) _{ \mathcal{T} _h } $ in \eqref{e:laplace}
lead to the following choice, which we assume for the remainder of
\cref{s:laplace}.

\begin{assumption}
  \label{a:laplace_penalty}
  Assume $ \alpha ^{ k -1 } > 0 $ and $ \alpha ^k < 0 $.
\end{assumption}

Under this assumption, we may now prove local and global existence and
uniqueness. The proofs are closely related to those of
\cref{t:dirac_local_eu,t:dirac_global_eu}, respectively.

\begin{theorem}[local existence and uniqueness]
  \label{t:laplace_local_eu}
  Given any $ \overline{ u } _h \in \overline{ \mathfrak{H} } _h ^k $,
  $ p _h \in \mathfrak{H} _h ^k $,
  $ \widehat{ \sigma } _h ^{\mathrm{tan}} \in \widehat{ V } _h ^{k-1,
    \mathrm{tan}} $,
  $ \widehat{ u } _h ^{\mathrm{tan}} \in \widehat{ V } _h ^{k,
    \mathrm{tan}} $, and $ f \in L ^2 \Lambda ^k (\Omega) $, there
  exist unique $ \sigma _h \in W _h ^{ k -1 } $, $ u _h \in W _h ^k $,
  $ \rho _h \in W _h ^{ k + 1 } $, and
  $ \overline{p} _h \in \overline{ \mathfrak{H} } _h ^k $ satisfying
  \eqref{e:laplace_tau}--\eqref{e:laplace_qbar}.
\end{theorem}

\begin{proof}
  Since this is a square system, it suffices to show that
  \begin{subequations}
    \label{e:laplace_local}
    \begin{alignat}{2}
      - ( \sigma _h , \tau _h ) _{ \mathcal{T} _h } + ( \delta u _h , \tau _h ) _{ \mathcal{T} _h } -  \langle \alpha ^{ k -1 } \sigma _h ^{\mathrm{tan}} , \tau _h ^{\mathrm{tan}} \rangle _{ \partial \mathcal{T} _h } &= 0, \quad &\forall \tau _h &\in W _h ^{ k -1 } , \label{e:laplace_local_tau} \\
      ( \sigma _h , \delta v _h ) _{ \mathcal{T} _h } + ( \delta \rho _h ,  v _h ) _{ \mathcal{T} _h } - ( \overline{p} _h , v _h ) _{ \mathcal{T} _h } - \langle \alpha ^k u _h ^{\mathrm{tan}} , v _h ^{\mathrm{tan}} \rangle _{ \partial \mathcal{T} _h } &= 0 , \quad &\forall v _h &\in W _h ^k , \label{e:laplace_local_v} \\
      - ( \rho _h , \eta _h ) _{ \mathcal{T} _h } + ( u _h , \delta \eta _h ) _{ \mathcal{T} _h } &= 0, \quad &\forall \eta _h &\in W _h ^{ k + 1 } , \label{e:laplace_local_eta} \\
      - ( u _h , \overline{q} _h ) _{ \mathcal{T} _h } &= 0 , \quad &\forall \overline{q} _h &\in \overline{ \mathfrak{H} } _h ^k , \label{e:laplace_local_qbar}
    \end{alignat}
  \end{subequations}
  has only the trivial solution.

  First, taking $ v _h = \overline{p} _h $ in
  \eqref{e:laplace_local_v}, integrating by parts, and rearranging
  gives
  \begin{equation*}
    ( \overline{p} _h , \overline{p} _h ) _{ \mathcal{T} _h } = ( \sigma _h + \rho _h, \mathrm{D} \overline{p} _h ) _{ \mathcal{T} _h } - \langle \rho _h ^{\mathrm{nor}} + \alpha ^k u _h ^{\mathrm{tan}} , \overline{p} _h ^{\mathrm{tan}} \rangle _{ \partial \mathcal{T} _h } = 0 ,
  \end{equation*}
  since $ \mathrm{D} \overline{p} _h = 0 $ and
  $ \overline{p} _h ^{\mathrm{tan}} = 0 $ by definition of
  $ \overline{ \mathfrak{H} } _h ^k $. Hence, $ \overline{p} _h = 0 $.
  Next, summing
  \eqref{e:laplace_local_tau}--\eqref{e:laplace_local_eta} with
  $ \tau _h = - \sigma _h $, $ v _h = u _h $, and
  $ \eta _h = - \rho _h $ gives, after cancellation of terms,
  \begin{equation*}
    ( \sigma _h , \sigma _h ) _{ \mathcal{T} _h } + ( \rho _h , \rho _h ) _{ \mathcal{T} _h } + \langle \alpha ^{ k -1 } \sigma _h ^{\mathrm{tan}} , \sigma _h ^{\mathrm{tan}} \rangle _{ \partial \mathcal{T} _h } - \langle \alpha ^k u _h ^{\mathrm{tan}} , u _h ^{\mathrm{tan}} \rangle _{ \partial \mathcal{T} _h } = 0 .
  \end{equation*}
  By \cref{a:laplace_penalty}, the left-hand side is positive-definite
  in $ \sigma _h $, $ \rho _h $, and $ u _h ^{\mathrm{tan}} $, so we
  conclude that $ \sigma _h = 0 $, $ \rho _h = 0 $, and
  $ u _h ^{\mathrm{tan}} = 0 $. Next, taking $ \tau _h = \delta u _h $
  in \eqref{e:laplace_local_tau} implies $ \delta u _h = 0 $;
  similarly, integrating \eqref{e:laplace_local_eta} by parts and
  taking $ \eta _h = \mathrm{d} u _h $ implies
  $ \mathrm{d} u _h = 0 $. Thus,
  $ u _h \in \overline{ \mathfrak{H} } _h ^k $. Finally, taking
  $ \overline{q} _h = u _h $ in \eqref{e:laplace_local_qbar} implies
  $ u _h = 0 $, which completes the proof.
\end{proof}

\begin{theorem}[global existence and uniqueness]
  \label{t:laplace_global_eu}
  Given any $ f \in L ^2 \Lambda ^k (\Omega) $, there exists a unique
  solution to \eqref{e:laplace}.
\end{theorem}

\begin{proof}
  Since \eqref{e:laplace} is a square system, it suffices to show that
  $ f = 0 $ has only the trivial solution. \Cref{l:laplace_properties}
  gives $ \overline{p} _h = 0 $ and $ p _h = 0 $, while
  $ \overline{ u } _h $ clearly vanishes if $ u _h $ does, so it
  remains to show that $ \sigma _h $, $ u _h $, $ \rho _h $,
  $ \widehat{ \sigma } _h ^{\mathrm{tan}} $, and
  $ \widehat{ u } _h ^{\mathrm{tan}} $ all vanish.

  Take $ \tau _h = - \sigma _h $, $ v _h = u _h $,
  $ \eta _h = - \rho _h $,
  $ \widehat{ \tau } _h ^{\mathrm{tan}} = - \widehat{ \sigma } _h
  ^{\mathrm{tan}} $, and
  $ \widehat{ v } _h ^{\mathrm{tan}} = \widehat{ u } _h
  ^{\mathrm{tan}} $. Adding
  \eqref{e:laplace_tau}--\eqref{e:laplace_eta} and
  \eqref{e:laplace_tautan}--\eqref{e:laplace_vtan}, and cancelling
  terms, gives
  \begin{equation*}
    ( \sigma _h , \sigma _h ) _{ \mathcal{T} _h } + ( \rho _h , \rho _h ) _{ \mathcal{T} _h } + \bigl\langle \alpha ^{ k -1 } ( \widehat{ \sigma } _h ^{\mathrm{tan}} - \sigma _h ^{\mathrm{tan}} ) , \widehat{ \sigma } _h ^{\mathrm{tan}} - \sigma _h ^{\mathrm{tan}} \bigr\rangle _{ \partial \mathcal{T} _h } - \bigl\langle \alpha ^k ( \widehat{ u } _h ^{\mathrm{tan}} - u _h ^{\mathrm{tan}} ) , \widehat{ u } _h ^{\mathrm{tan}} - u _h ^{\mathrm{tan}} \bigr\rangle _{ \partial \mathcal{T} _h } = 0 .
  \end{equation*}
  By \cref{a:laplace_penalty}, it follows that $ \sigma _h = 0 $,
  $ \rho _h = 0 $, $ \widehat{ \sigma } _h ^{\mathrm{tan}} = 0 $, and
  $ \widehat{ u } _h ^{\mathrm{tan}} = u _h ^{\mathrm{tan}} $, and
  consequently also
  $ \widehat{ u } _h ^{\mathrm{nor}} = u _h ^{\mathrm{nor}} $. Hence,
  $ \jump{ u _h ^{\mathrm{tan}} } = 0 $ and
  $ \jump{ u _h ^{\mathrm{nor}} } = 0 $, so
  $ u _h \in H \Lambda ^k (\Omega) \cap \mathring{ H } ^\ast \Lambda
  ^k (\Omega) $.

  Accounting for the terms already shown to vanish,
  \eqref{e:laplace_tau} with $ \tau _h = \delta u _h $ now implies
  $ \delta u _h = 0 $; likewise, integrating \eqref{e:laplace_eta} by
  parts and taking $ \eta _h = \mathrm{d} u _h $ implies
  $ \mathrm{d} u _h = 0 $. Hence, $ u _h \in \mathfrak{H} _h ^k
  $. Finally, taking $ q _h = u _h $ in \eqref{e:laplace_q} implies
  $ u _h = 0 $, and thus
  $ \widehat{ u } _h ^{\mathrm{tan}} = u _h ^{\mathrm{tan}} = 0 $ as
  well.
\end{proof}

\subsection{Jump-average DG formulation}

As with the Hodge--Dirac HDG method, it will be useful to pose our
Hodge--Laplace HDG method in an ``unhybridized'' DG form, expressed in
terms of tangential and normal jumps and averages, for the subsequent
error analysis.  The following results are a straightforward
adaptation of \cref{l:dg_traces,p:dg,c:dg} to the
Hodge--Laplace problem. We omit the proofs, which largely repeat
earlier arguments.

\begin{proposition}
  The approximate traces for the Hodge--Laplace HDG method satisfy
  \begin{align*}
    \widehat{ \sigma } _h ^{\mathrm{tan}} &= \avg{ \sigma _h ^{\mathrm{tan}} } + \frac{ 1 }{ \avg{ \alpha ^{ k -1 } } } \jump{ u _h ^{\mathrm{nor}} } + \frac{ \jump{ \alpha ^{ k -1 } }}{ \avg{ \alpha ^{ k -1 } } } \jump{ \sigma _h ^{\mathrm{tan}} } ,\\
    \widehat{ u } _h ^{\mathrm{nor}} &= \avg{ u _h ^{\mathrm{nor}} } + \frac{ 1 }{ \avg{ \beta ^{ k -1 } } } \jump{ \sigma _h ^{\mathrm{tan}} } + \frac{ \jump{ \beta ^{ k -1 } }}{ \avg{ \beta ^{ k -1 } } } \jump{ u _h ^{\mathrm{nor}} } ,\\
    \widehat{ u } _h ^{\mathrm{tan}} &= \avg{ u _h ^{\mathrm{tan}} } + \frac{ 1 }{ \avg{ \alpha ^k } } \jump{ \rho _h ^{\mathrm{nor}} } + \frac{ \jump{ \alpha ^k }}{ \avg{ \alpha ^k } } \jump{ u _h ^{\mathrm{tan}} } ,\\
    \widehat{ \rho } _h ^{\mathrm{nor}} &= \avg{ \rho _h ^{\mathrm{nor}} } + \frac{ 1 }{ \avg{ \beta ^k } } \jump{ u _h ^{\mathrm{tan}} } + \frac{ \jump{ \beta ^k }}{ \avg{ \beta ^k } } \jump{ \rho _h ^{\mathrm{nor}} } ,
  \end{align*}
  and the method is equivalent to finding
  $ \sigma _h \in W _h ^{ k -1 } $, $ u _h \in W _h ^k $,
  $ \rho _h \in W _h ^{ k + 1 } $, and $ p _h \in \mathfrak{H} _h ^k $
  satisfying
  \begin{subequations}
    \label{e:laplace_dg}
    \begin{alignat}{2}
      - ( \sigma _h , \tau _h ) _{ \mathcal{T} _h } + ( \delta u _h , \tau _h ) _{ \mathcal{T} _h } + \bigl\langle \jump{ u _h ^{\mathrm{nor}} }, \avg{ \tau _h ^{\mathrm{tan}} } \bigr\rangle _{ \partial \mathcal{T} _h } \qquad \label{e:laplace_dg_tau} \\
      {}- \biggl\langle \frac{ 1 }{ \avg{ \beta ^{ k -1 } } } \jump{ \sigma _h ^{\mathrm{tan}} } + \frac{ \jump{ \beta ^{ k -1 } }}{ \avg{ \beta ^{ k -1 } } } \jump{ u _h ^{\mathrm{nor}} } , \jump{ \tau _h ^{\mathrm{tan}} } \biggr\rangle _{ \partial \mathcal{T} _h } &= 0, \quad & \forall \tau _h &\in W _h ^{ k -1 } , \notag \\
      ( \sigma _h , \delta v _h ) _{ \mathcal{T} _h } + ( \delta \rho _h , v _h ) _{ \mathcal{T} _h } + ( p _h , v _h ) _{ \mathcal{T} _h } \qquad \qquad \qquad \label{e:laplace_dg_v} \\
      {}+ \bigl\langle \avg{ \sigma _h ^{\mathrm{tan}} } , \jump{ v _h ^{\mathrm{nor}} } \bigr\rangle _{ \partial \mathcal{T} _h } + \bigl\langle \jump{ \rho _h ^{\mathrm{nor}} } , \avg{ v _h ^{\mathrm{tan}} } \bigr\rangle _{ \partial \mathcal{T} _h } \qquad \notag \\
      {}+ \biggl\langle \frac{ 1 }{ \avg{ \alpha ^{ k -1 } } } \jump{ u _h ^{\mathrm{nor}} } + \frac{ \jump{ \alpha ^{ k -1 } }}{ \avg{ \alpha ^{ k -1 } } } \jump{ \sigma _h ^{\mathrm{tan}} } , \jump{ v _h ^{\mathrm{nor}} } \biggr\rangle _{ \partial \mathcal{T} _h } \notag\\
      {}- \biggl\langle \frac{ 1 }{ \avg{ \beta ^k } } \jump{ u _h ^{\mathrm{tan}} } + \frac{ \jump{ \beta ^k }}{ \avg{ \beta ^k } } \jump{ \rho _h ^{\mathrm{nor}} } , \jump{ v _h ^{\mathrm{tan}} } \biggr\rangle _{ \partial \mathcal{T} _h } &= ( f, v _h ) _{ \mathcal{T} _h } , \quad & \forall v _h &\in W _h ^k , \notag \\
      - ( \rho _h , \eta _h ) _{ \mathcal{T} _h } + ( u _h , \delta \eta _h ) _{ \mathcal{T} _h } + \bigl\langle \avg{ u _h ^{\mathrm{tan}} }, \jump{ \eta _h ^{\mathrm{nor}} } \bigr\rangle _{ \partial \mathcal{T} _h } \qquad \label{e:laplace_dg_eta} \\
      {}+ \biggl\langle \frac{ 1 }{ \avg{ \alpha ^k } } \jump{ \rho _h ^{\mathrm{nor}} } + \frac{ \jump{ \alpha ^k }}{ \avg{ \alpha ^k } } \jump{ u _h ^{\mathrm{tan}} }, \jump{ \eta _h ^{\mathrm{nor}} } \biggr\rangle _{ \partial \mathcal{T} _h } &= 0 , \quad & \forall \eta _h &\in W _h ^{ k + 1 } , \notag \\
      ( u _h , q _h ) _{ \mathcal{T} _h } &= 0, \quad &\forall q _h &\in \mathfrak{H} _h ^k .
      \end{alignat}
  \end{subequations}
\end{proposition}

\begin{corollary}
  If $ \alpha ^{ k -1 } $ and $ \alpha ^k $ are single-valued, then
  the approximate traces satisfy
  \begin{alignat*}{2}
    \widehat{ \sigma } _h ^{\mathrm{tan}} &= \avg{ \sigma _h ^{\mathrm{tan}} } + \beta ^{ k -1 } \jump{ u _h ^{\mathrm{nor}} } , \qquad & \widehat{ u } _h ^{\mathrm{tan}} &= \avg{ u _h ^{\mathrm{tan}} } + \beta ^k \jump{ \rho _h ^{\mathrm{nor}} } ,\\
    \widehat{ u } _h ^{\mathrm{nor}} &= \avg{ u _h ^{\mathrm{nor}} } + \alpha ^{ k -1 } \jump{ \sigma _h ^{\mathrm{tan}} } , \qquad & \widehat{ \rho } _h ^{\mathrm{nor}} &= \avg{ \rho _h ^{\mathrm{nor}} } + \alpha ^k \jump{ u _h ^{\mathrm{tan}} } ,
  \end{alignat*}
  and the method is equivalent to finding
  $ \sigma _h \in W _h ^{ k -1 } $, $ u _h \in W _h ^k $,
  $ \rho _h \in W _h ^{ k + 1 } $, and $ p _h \in \mathfrak{H} _h ^k $
  satisfying
  \begin{alignat*}{2}
    - ( \sigma _h , \tau _h ) _{ \mathcal{T} _h } + ( \delta u _h , \tau _h ) _{ \mathcal{T} _h } + \bigl\langle \jump{ u _h ^{\mathrm{nor}} }, \avg{ \tau _h ^{\mathrm{tan}} } \bigr\rangle _{ \partial \mathcal{T} _h } \\
    {}- \bigl\langle \alpha ^{ k -1 } \jump{ \sigma _h ^{\mathrm{tan}} } , \jump{ \tau _h ^{\mathrm{tan}} } \bigr\rangle _{ \partial \mathcal{T} _h } &= 0, \quad & \forall \tau _h &\in W _h ^{ k -1 } , \\
    ( \sigma _h , \delta v _h ) _{ \mathcal{T} _h } + ( \delta \rho _h , v _h ) _{ \mathcal{T} _h } + ( p _h , v _h ) _{ \mathcal{T} _h } \qquad \qquad \qquad \\
    {}+ \bigl\langle \avg{ \sigma _h ^{\mathrm{tan}} } , \jump{ v _h ^{\mathrm{nor}} } \bigr\rangle _{ \partial \mathcal{T} _h } + \bigl\langle \jump{ \rho _h ^{\mathrm{nor}} } , \avg{ v _h ^{\mathrm{tan}} } \bigr\rangle _{ \partial \mathcal{T} _h } \qquad \quad \\
    {}+ \bigl\langle \beta ^{ k -1 } \jump{ u _h ^{\mathrm{nor}} } , \jump{ v _h ^{\mathrm{nor}} } \bigr\rangle _{ \partial \mathcal{T} _h } - \bigl\langle \alpha ^k \jump{ u _h ^{\mathrm{tan}} } , \jump{ v _h ^{\mathrm{tan}} } \bigr\rangle _{ \partial \mathcal{T} _h } &= ( f, v _h ) _{ \mathcal{T} _h } , \quad & \forall v _h &\in W _h ^k , \\
    - ( \rho _h , \eta _h ) _{ \mathcal{T} _h } + ( u _h , \delta \eta _h ) _{ \mathcal{T} _h } + \bigl\langle \avg{ u _h ^{\mathrm{tan}} }, \jump{ \eta _h ^{\mathrm{nor}} } \bigr\rangle _{ \partial \mathcal{T} _h } \\
    {}+ \bigl\langle \beta ^k \jump{ \rho _h ^{\mathrm{nor}} }, \jump{ \eta _h ^{\mathrm{nor}} } \bigr\rangle _{ \partial \mathcal{T} _h } &= 0 , \quad & \forall \eta _h &\in W _h ^{ k + 1 } , \\
    ( u _h , q _h ) _{ \mathcal{T} _h } &= 0, \quad &\forall q _h &\in \mathfrak{H} _h ^k .
  \end{alignat*}
\end{corollary}

\subsection{Error analysis}

We now analyze the error of these methods for the Hodge--Laplace
problem, following the same general outline we used in
\cref{s:dirac_error} for Hodge--Dirac.

\subsubsection{Bilinear forms}
We begin by expressing the Hodge--Laplace variational problem in terms
of the same bilinear forms used in the Hodge--Dirac error
analysis. This will allow us to use many of the analytical results we
have already established, including the approximation-theoretic
estimates for $ K _{\mathcal{A}} $ and $ K _{\mathcal{B}} $ from
\cref{s:KA_KB}. First, we write
\begin{equation*}
  z _h = \sigma _h \oplus u _h \oplus \rho _h , \qquad w _h = \tau _h \oplus v _h \oplus \eta _h ,
\end{equation*}
and observe that \eqref{e:laplace_dg} can be expressed as
\begin{equation*}
  a ( z _h , w _h ) - ( \sigma _h , \tau _h ) _{ \mathcal{T} _h } - ( \rho _h , \eta _h ) _{ \mathcal{T} _h } = ( f - p _h , w _h ) _{ \mathcal{T} _h } .
\end{equation*}
The bilinear form $ \mathcal{A} ( \cdot , \cdot ) $ replaces
$ \tau _h $ by $ - \tau _h $ and $ \eta _h $ by $ - \eta _h $, i.e.,
it flips the sign of equations \eqref{e:laplace_dg_tau} and
\eqref{e:laplace_dg_eta}. We may therefore write \eqref{e:laplace_dg}
as
\begin{equation*}
  \mathcal{A} ( z _h , w _h ) + \mathcal{C} ( z _h , w _h ) = ( f - p _h , w _h ) _{ \mathcal{T} _h } ,
\end{equation*}
where
\begin{equation*}
  \mathcal{C} ( z _h , w _h ) \coloneqq ( \sigma _h , \tau _h ) _{ \mathcal{T} _h } + ( \rho _h , \eta _h ) _{ \mathcal{T} _h } ,
\end{equation*}
i.e., $ \mathcal{C} ( \cdot , \cdot ) $ is just the $ L ^2 $ inner
product on $ W ^{ k -1 } \oplus W ^{ k + 1 } $.

\begin{remark}
  The astute reader may notice that the expression
  $ \mathcal{A} ( z _h , w _h ) $, according to \cref{d:A}, would also
  contain the extra terms
  \begin{equation*}
    - \biggl\langle \frac{ 1 }{ \avg{ \alpha ^{ k - 2 } } } \jump{ \sigma _h ^{\mathrm{nor}} } , \jump{ \tau _h ^{\mathrm{nor}} } \biggr\rangle _{ \partial \mathcal{T} _h } + \biggl\langle \frac{ 1 }{ \avg{ \beta ^{ k + 1 } } }  \jump{ \rho _h ^{\mathrm{tan}} }, \jump{ \eta _h ^{\mathrm{tan}} } \biggr\rangle _{ \partial \mathcal{T} _h } ,
  \end{equation*}
  which do not appear in \eqref{e:laplace_dg}. Such terms arise in the
  Hodge--Dirac HDG method from traces of $ ( k \pm 2 ) $-forms, but
  those form degrees are not involved in the Hodge--Laplace problem.

  To eliminate these extra terms, we formally set
  $ \beta ^{ k - 2 } = 0 $ and $ \alpha ^{ k + 1 } = 0 $. In doing so,
  we no longer have the full seminorm equivalence
  \eqref{e:A_equiv_L2}, which is used in the last step of the proof of
  \cref{l:KB}. However, we do have the modified equivalence
  \begin{equation*}
    \lvert \tau \oplus v \oplus \eta \rvert _{\mathcal{A}} ^2 \asymp \bigl\lVert \jump{ \tau ^{\mathrm{tan}} } \bigr\rVert _{ \partial \mathcal{T} _h } ^2 + \bigl\lVert \jump{ v ^{\mathrm{nor}} } \bigr\rVert _{ \partial \mathcal{T} _h } ^2 + \bigl\lVert \jump{ v ^{\mathrm{tan}} } \bigr\rVert _{ \partial \mathcal{T} _h } ^2 + \bigl\lVert \jump{ \eta ^{\mathrm{nor}} } \bigr\rVert _{ \partial \mathcal{T} _h } ^2 ,
  \end{equation*}
  which is sufficient for our purposes: the missing terms
  $ \jump{ \tau ^{\mathrm{nor}} } $ and
  $ \jump{ \eta ^{\mathrm{tan}} } $ never appear in the Hodge--Laplace
  problem, so we have no need to control them. Hence the instance of
  \cref{l:KB} we need for the Hodge--Laplace error analysis still goes
  through.
\end{remark}

Denoting $ w = \tau \oplus v \oplus \eta $, the seminorm
$ \lvert \cdot \rvert _{\mathcal{A}} $ on
$ W ^{ k -1 } \oplus W ^k \oplus W ^{ k + 1 } $ may now be written as
\begin{align*}
  \lvert w \rvert _{\mathcal{A}} ^2 &= \biggl\langle \frac{ 1 }{ \avg{ \beta ^{ k -1 } } } \jump{ \tau ^{\mathrm{tan}} } , \jump{ \tau ^{\mathrm{tan}} } \biggr\rangle _{ \partial \mathcal{T} _h } + \biggl\langle \frac{ 1 }{ \avg{ \alpha ^{ k -1 } } } \jump{ v ^{\mathrm{nor}} } , \jump{ v ^{\mathrm{nor}} } \biggr\rangle _{ \partial \mathcal{T} _h } \\
  &\quad - \biggl\langle \frac{ 1 }{ \avg{ \beta ^k } } \jump{ v ^{\mathrm{tan}} } , \jump{ v ^{\mathrm{tan}} } \biggr\rangle _{ \partial \mathcal{T} _h } - \biggl\langle \frac{ 1 }{ \avg{ \alpha ^k } } \jump{ \eta ^{\mathrm{nor}} } , \jump{ \eta ^{\mathrm{nor}} } \biggr\rangle _{ \partial \mathcal{T} _h } .
\end{align*}
Since $ \mathcal{C} ( \cdot , \cdot ) $ is a semi-inner product, it
also has an associated seminorm
$ \lvert \cdot \rvert _{\mathcal{C}} $, given by
\begin{equation*}
  \lvert w \rvert _{\mathcal{C}} ^2 \coloneqq \mathcal{C} ( w, w ) = \lVert \tau \rVert _{ \mathcal{T} _h } ^2 + \lVert \eta \rVert _{ \mathcal{T} _h } ^2 ,
\end{equation*}
and we have the Cauchy--Schwarz inequality
\begin{equation*}
  \bigl\lvert \mathcal{C} ( z, w ) \bigr\rvert \leq \lvert z \rvert _{\mathcal{
C}} \lvert w \rvert _{\mathcal{C}} .
\end{equation*}

\subsubsection{Abstract error estimates}

We next obtain abstract error estimates analogous to those in
\cref{s:dirac_error_abstract}. The functionals $ K _{\mathcal{A}} $
and $ K _{\mathcal{B}} $ have the same meaning as before, where the
inequalities \eqref{e:KA_KB} are only assumed to hold for all
$ z, w \in W ^{ k -1 } \oplus W ^k \oplus W ^{ k + 1 } $ and
$ w _h \in W _h ^{ k -1 } \oplus W _h ^k \oplus W _h ^{ k + 1 } $.

Since the Hodge--Laplace problem only involves degree-$k$ harmonic
forms, we use the following assumption, which weakens
\cref{a:harmonic} by requiring only that it hold for $k$-forms.

\begin{assumption}
  \label{a:harmonic_k}
  Assume that $ \mathfrak{H} ^k (\Omega) \subset W _h ^k $, so that
  $ \mathfrak{H} _h ^k = \mathfrak{H} ^k (\Omega) $.
\end{assumption}

Denoting the exact solution to the Hodge--Laplace problem by
$ z = \sigma \oplus u \oplus \rho $, it follows from this assumption
that we have the Galerkin orthogonality relation
\begin{equation*}
  \mathcal{A} ( z - z _h , w _h ) + \mathcal{C} ( z - z _h , w _h ) = 0 , \quad \forall w _h \in W _h ^{ k -1 } \oplus W _h ^k \oplus W _h ^{ k + 1 } ,
\end{equation*}
and furthermore that $ u - u _h \in \mathfrak{H} ^k (\Omega) ^\perp $,
which we will use in a duality argument. As before, we denote
$ e _z \coloneqq z - z _h $ and $ \xi _z \coloneqq z - \Pi z $, where
$\Pi$ is again the $ L ^2 $-orthogonal projection onto $ W _h $. Note
that $\Pi$ maps $ W ^{ k -1 } \oplus W ^k \oplus W ^{ k + 1 } $ onto
$ W _h ^{ k -1 } \oplus W _h ^k \oplus W _h ^{ k + 1 } $, since
different form degrees are $ L ^2 $-orthogonal to one another.

The following estimate is quite similar to \cref{l:dirac_error_A},
with extra terms arising from $ \mathcal{C} ( \cdot , \cdot ) $.

\begin{lemma}
  \label{l:laplace_error_A_C}
  Under the preceding assumptions, we have the $\mathcal{A}$- and
  $\mathcal{C}$-seminorm error estimates
  \begin{equation*}
    \lvert e _z \rvert _{\mathcal{A}} ^2 \leq K _{\mathcal{A}} ( z, z ) + K _{\mathcal{B}} (z) ^2 , \qquad \lvert e _z \rvert _{\mathcal{C}} ^2 \leq \lvert \xi _z \rvert _{\mathcal{C}} ^2 + K _{\mathcal{B}} (z) ^2 .
  \end{equation*}
\end{lemma}

\begin{proof}
  We begin by observing that $ \xi _z $ and $ \Pi e _z $ are
  $ L ^2 $-orthogonal and thus $\mathcal{C}$-orthogonal, so in
  particular
  $ \lvert \Pi e _z \rvert _{\mathcal{C}} ^2 = \mathcal{C} ( e _z ,
  \Pi e _z ) $. Using this observation and Galerkin orthogonality, we
  have
  \begin{align*}
    \lvert e _z \rvert _{\mathcal{A}} ^2
    &\leq \lvert e _z \rvert _{\mathcal{A}} ^2 + \lvert \Pi e _z \rvert _{\mathcal{C}} ^2 \\
    &= \mathcal{A} ( e _z, e _z ) + \mathcal{C} ( e _z , \Pi e _z ) \\
    &= \mathcal{A} ( e _z , \xi _z ) \\
    &= \mathcal{A} ( \xi _z, \xi _z ) + \mathcal{A} ( \Pi e _z , \xi _z ) \\
    &\leq K _{\mathcal{A}} ( z, z ) + \lvert \Pi e _z \rvert _{\mathcal{A}} K _{\mathcal{B}} (z) ,
  \end{align*}
  where the last line uses \eqref{e:KA} and \eqref{e:KB_wz}. To
  estimate $ \lvert \Pi e _z \rvert _{\mathcal{A}} $, we similarly write
  \begin{equation}
    \label{e:Pi_ez_laplace}
    \begin{aligned}
      \lvert \Pi e _z \rvert _{\mathcal{A}} ^2
      &\leq \lvert \Pi e _z \rvert _{\mathcal{A}} ^2 + \lvert \Pi e _z \rvert _{\mathcal{C}} ^2 \\
      &= \mathcal{A} ( \Pi e _z , \Pi e _z ) + \mathcal{C} ( e _z , \Pi e _z ) \\
      &= - \mathcal{A} ( \xi _z , \Pi e _z ) \\
      &\leq K _{\mathcal{B}} (z) \lvert \Pi e _z \rvert _{\mathcal{A}} ,
    \end{aligned}
  \end{equation}
  again using Galerkin orthogonality, along with
  \eqref{e:KB_zw}. Thus,
  $ \lvert \Pi e _z \rvert _{\mathcal{A}} \leq K _{\mathcal{B}} (z) $,
  and substituting this above implies the $\mathcal{A}$-seminorm
  estimate.

  Next, since $ e _z = \xi _z + \Pi e _z $ is an orthogonal sum, the
  Pythagorean theorem implies
  \begin{equation*}
    \lvert e _z \rvert _{\mathcal{C}} ^2 = \lvert \xi _z \rvert _{\mathcal{C}} ^2 + \lvert \Pi e _z \rvert _{\mathcal{C}} ^2 .
  \end{equation*}
  From \eqref{e:Pi_ez_laplace}, we have
  \begin{align*}
    \lvert \Pi e _z \rvert _{\mathcal{C}} ^2
    &\leq \lvert \Pi e _z \rvert _{\mathcal{A}} ^2 + \lvert \Pi e _z \rvert _{\mathcal{C}} ^2 \\
    &\leq K _{\mathcal{B}} (z) \lvert \Pi e _z \rvert _{\mathcal{A}} \\
    &\leq K _{\mathcal{B}} (z) ^2 ,
  \end{align*}
  which completes the proof.
\end{proof}

To describe the adjoint problem, we first recall the spaces of coexact
and exact forms
\begin{equation*}
  \mathring{ \mathfrak{B} } ^\ast _{ k - 1 } (\Omega) \coloneqq \delta \mathring{ H } ^\ast \Lambda ^k (\Omega) , \qquad \mathfrak{B} ^{ k + 1 } (\Omega) \coloneqq \mathrm{d} H \Lambda ^k (\Omega) .
\end{equation*}
Given
$ \lambda \in \mathring{ \mathfrak{B} } ^\ast _{ k -1 } (\Omega)
\oplus \mathfrak{H} ^k (\Omega) ^\perp \oplus \mathfrak{B} ^{ k + 1 }
(\Omega) $, we seek $ \varphi = \chi \oplus \psi \oplus \omega $ with
$ \psi \in \mathfrak{H} ^k (\Omega) ^\perp $ such that
\begin{alignat*}{2}
  \chi + \delta \psi &= \lambda ^{ k -1 }  \quad &\text{in } &\Omega ,\\
  -\mathrm{d} \chi - \delta \omega &= \lambda ^k \quad &\text{in } &\Omega ,\\
  \omega + \mathrm{d} \psi &= \lambda ^{ k + 1 }  \quad &\text{in } &\Omega ,\\
  \psi ^{\mathrm{nor}} &= 0 \quad &\text{on } &\partial \Omega ,\\
  \omega ^{\mathrm{nor}} &= 0 \quad &\text{on } &\partial \Omega ,
\end{alignat*}
which therefore satisfies
\begin{equation}
  \label{e:laplace_adjoint}
  \mathcal{A} ( w, \varphi ) + \mathcal{C} ( w, \varphi ) = ( w, \lambda ) _{ \mathcal{T} _h } , \quad \forall w \in W ^{ k -1 } \oplus W ^k \oplus W ^{ k + 1 } .
\end{equation}
If $ \lambda ^{ k \pm 1 } = 0 $, then this is just the Hodge--Laplace
problem $ L \psi = \lambda ^k $. However, $ \lambda ^k $ only gives us
a duality estimate for $ e _u $; by allowing $ \lambda ^{ k \pm 1 } $
to be nonzero, we will also be able to obtain duality estimates for
$ P _{ \mathring{ \mathfrak{B} ^\ast } } e _\sigma $ and
$ P _{ \mathfrak{B} } e _\rho $ that improve on the estimate for
$ \lvert e _z \rvert _{\mathcal{C}} $ in \cref{l:laplace_error_A_C}.

Although this is not a Hodge--Laplace problem in general, the solution
to the adjoint problem may still be obtained by solving two successive
Hodge--Dirac problems:
\begin{itemize}
\item First, solve
  $ - \mathrm{D} ( \chi \oplus \omega ) = \lambda ^k $.
\item Next, solve
  $ - \mathrm{D} \psi = ( \chi - \lambda ^{ k -1 } ) \oplus ( \omega -
  \lambda ^{ k + 1 } ) $.
\end{itemize}
Since the solution to the first problem has
$ \chi \in \mathfrak{H} ^{ k -1 } (\Omega) ^\perp $ and
$ \delta \chi = 0 $, it follows that
$ \chi \in \mathring{ \mathfrak{B} } ^\ast _{ k -1 } (\Omega) $, and
likewise $ \omega \in \mathfrak{B} ^{ k + 1 } (\Omega) $. The
conditions
$ \lambda ^{ k -1 } \in \mathring{ \mathfrak{B} } ^\ast _{ k -1 }
(\Omega) $ and
$ \lambda ^{ k + 1 } \in \mathfrak{B} ^{ k + 1 } (\Omega) $ then
ensure that the right-hand side of the second problem is in
$ \mathrm{D} \bigl( H \Lambda ^k (\Omega) \cap \mathring{ H } ^\ast
\Lambda ^k (\Omega) \bigr) = \mathring{ \mathfrak{B} } ^\ast _{ k -1 }
(\Omega) \oplus \mathfrak{B} ^{ k + 1 } (\Omega) $, meaning that
$\psi$ need not have any $ ( k \pm 2 ) $-form components. From
$t$-regularity of $\Omega$ and Poincar\'e's inequality, we therefore
have
\begin{equation*}
  \lVert \chi \oplus \omega \rVert _{ t, \Omega } \lesssim \lVert \lambda ^k \rVert _\Omega , \qquad \lVert \psi \rVert _{t, \Omega} \lesssim \bigl\lVert ( \chi - \lambda ^{ k -1 } ) \oplus ( \omega - \lambda ^{ k + 1 } ) \bigr\rVert _\Omega ,
\end{equation*}
which together imply
$ \lVert \varphi \rVert _{ t, \Omega } \lesssim \lVert \lambda \rVert
_\Omega $.

\begin{lemma}
  \label{l:laplace_error_l2}
  Under the preceding assumptions, we have the $ L ^2 $-norm error
  estimate
  \begin{equation*}
    \lVert P _{ \mathring{ \mathfrak{B} } ^\ast } e _\sigma \oplus e _u \oplus P _{ \mathfrak{B} } e _\rho \rVert _\Omega \leq \sup _{\lambda \neq 0} \frac{ K _{\mathcal{A}} ( z, \varphi ) }{ \lVert \lambda \rVert _\Omega } + K _{\mathcal{B}} (z) \sup _{\lambda \neq 0} \frac{ K _{\mathcal{B}} (\varphi) }{ \lVert \lambda \rVert _\Omega } + \lvert \xi _z \rvert _{\mathcal{C}} \sup _{\lambda \neq 0} \frac{ \lvert \xi _\varphi \rvert _{\mathcal{C}} }{ \lVert \lambda \rVert _\Omega } ,
  \end{equation*}
  where $\varphi$ denotes the solution to the adjoint problem for
  $ \lambda \in \mathring{ \mathfrak{B} } ^\ast _{ k -1 } (\Omega)
  \oplus \mathfrak{H} ^k (\Omega) ^\perp \oplus \mathfrak{B} ^{ k + 1
  } (\Omega) $.
\end{lemma}

\begin{proof}
  For any
  $ \lambda \in \mathring{ \mathfrak{B} } ^\ast _{ k -1 } (\Omega)
  \oplus \mathfrak{H} ^k (\Omega) ^\perp \oplus \mathfrak{B} ^{ k + 1
  } (\Omega) $, taking $ w = e _z $ in \eqref{e:laplace_adjoint} gives
  \begin{align*}
    ( e _z, \lambda ) _{ \mathcal{T} _h }
    &= \mathcal{A} ( e _z, \varphi ) + \mathcal{C} ( e _z, \varphi ) \\
    &= \mathcal{A} ( e _z, \xi _\varphi ) + \mathcal{C} ( e _z, \xi _\varphi ) \\
    &= \mathcal{A} ( \xi _z, \xi _\varphi ) + \mathcal{A} ( \Pi e _z , \xi _\varphi ) + \mathcal{C} ( \xi _z , \xi _\varphi ) \\
    &\leq K _{\mathcal{A}} ( z , \varphi ) + \lvert \Pi e _z \rvert _{\mathcal{A}} K _{\mathcal{B}} (\varphi) + \lvert \xi _z \rvert _{\mathcal{C}} \lvert \xi _\varphi \rvert _{\mathcal{C}} \\
    &\leq K _{\mathcal{A}} ( z , \varphi ) + K _{\mathcal{B}} (z) K _{\mathcal{B}} (\varphi) + \lvert \xi _z \rvert _{\mathcal{C}} \lvert \xi _\varphi \rvert _{\mathcal{C}} ,
  \end{align*}
  where the second line holds by Galerkin orthogonality, the third
  holds by the orthogonality of the decomposition \eqref{e:ez_decomp},
  the fourth holds by \eqref{e:KA_KB} and Cauchy--Schwarz, and the
  last holds by \eqref{e:Pi_ez_laplace}. Finally, since
  $ e _u \in \mathfrak{H} ^k (\Omega) ^\perp $, we have
  \begin{equation*}
    \lVert P _{ \mathring{ \mathfrak{B} } ^\ast } e _\sigma \oplus e _u \oplus P _{ \mathfrak{B} } e _\rho \rVert _\Omega = \sup _{ \lambda \neq 0 } \frac{ ( e _z, \lambda ) _\Omega }{ \lVert \lambda \rVert _\Omega } ,
  \end{equation*}
  where this supremum is taken over
  $ \mathring{ \mathfrak{B} } ^\ast _{ k -1 } (\Omega) \oplus
  \mathfrak{H} ^k (\Omega) ^\perp \oplus \mathfrak{B} ^{ k + 1 }
  (\Omega) $.
\end{proof}

\subsubsection{Final error estimates}

We now combine the abstract estimates for the Hodge--Laplace problem
with the approximation-theoretic results previously obtained in
\cref{s:KA_KB}. To apply these results, we assume the same
mesh-regularity condition (\cref{a:mesh}) and that the penalty
coefficients
$ \alpha ^{ k -1 } , \alpha ^k , \beta ^{ k -1 } , \beta ^k $ are
$ \mathcal{O} (1) $.

\begin{theorem}
  \label{t:laplace_error}
  Let $ ( \sigma, u , \rho, p ) $ solve the $k$-form Hodge--Laplace
  problem on $\Omega$, and let
  $ ( \sigma _h , u _h , \rho _h , p _h ) $ be the degree-$r$ HDG
  solution in the form \eqref{e:laplace_dg}. Suppose $\Omega$ is
  $t$-regular and $ z \in H ^s \Lambda (\Omega) $, where
  $ \frac{1}{2} < t \leq s \leq r + 1 $. Under the preceding
  assumptions, the error satisfies the estimates
  \begin{subequations}
    \begin{align}
      \lvert z - z _h \rvert _{\mathcal{A}} + \lvert z - z _h \rvert _{\mathcal{C}} &\lesssim h ^{ s - 1/2 } \lvert z \rvert _{ s, \Omega } , \label{e:laplace_error_A_C} \\
      \lVert \sigma - P _{ \mathring{ \mathfrak{B} } ^\ast } \sigma _h \rVert _\Omega + \lVert u - u _h \rVert _\Omega + \lVert \rho - P _{ \mathfrak{B} } \rho _h \rVert _\Omega &\lesssim h ^{ s + t -1 } \lvert z \rvert _{ s, \Omega } , \label{e:laplace_error_l2}
    \end{align}
  \end{subequations}
  where $ z = \sigma \oplus u \oplus \rho $ and
  $ z _h = \sigma _h \oplus u _h \oplus \rho _h $.
\end{theorem}

\begin{proof}
  As in the proof of \cref{t:dirac_error}, we have
  \begin{alignat*}{3}
    K _{\mathcal{A}} ( z, z ) &= C h ^{ 2 s -1 } \lvert z \rvert _{ s , \Omega } ^2 , \qquad & K _{\mathcal{B}} (z) &= C h ^{ s - 1/2 } \lvert z \rvert _{ s, \Omega } , \qquad & \lvert \xi _z \rvert _{\mathcal{C}} &\lesssim h ^s \lvert z \rvert _{ s, \Omega } ,\\
    K _{\mathcal{A}} ( z, \varphi ) &= C h ^{ s + t - 1 } \lvert z \rvert _{ s , \Omega } \lvert \varphi \rvert _{ t , \Omega } , \qquad & K _{\mathcal{B}} (\varphi) &= C h ^{ t - 1/2 } \lvert \varphi  \rvert _{ t, \Omega } , \qquad &\lvert \xi _\varphi \rvert _{\mathcal{C}} &\lesssim h ^t \lvert \varphi \rvert _{ t, \Omega } ,
  \end{alignat*}
  where the three columns respectively use \cref{l:KA}, \cref{l:KB},
  and Bramble--Hilbert. Hence, the first estimate follows by
  \cref{l:laplace_error_A_C}, and the second follows by
  \cref{l:laplace_error_l2} together with the estimate
  $ \lvert \varphi \rvert _{ t, \Omega } \lesssim \lVert \lambda
  \rVert _\Omega $ for the adjoint problem.
\end{proof}

In general, the auxiliary variables $ \sigma _h $ and $ \rho _h $ have
$ L ^2 $-norm errors controlled by the term
$ \lvert z - z _h \rvert _{\mathcal{C}} $ in
\eqref{e:laplace_error_A_C}; only their projections
$ P _{ \mathring{ \mathfrak{B} } ^\ast } \sigma _h $ and
$ P _{ \mathfrak{B} } \rho _h $ are proved to converge at the improved
rate \eqref{e:laplace_error_l2}. However, in specific cases, we have
$ \sigma _h = P _{ \mathring{ \mathfrak{B} } ^\ast } \sigma _h $
and/or $ \rho _h = P _{ \mathfrak{B} } \rho _h $, giving improved
estimates.

\begin{corollary}
  \label{c:B_Bstar}
  Under the hypotheses of \cref{t:laplace_error}, if $ k = 1 $, then
  $ \sigma _h \in \mathring{ \mathfrak{B} } _0 ^\ast (\Omega) $
  satisfies
  \begin{equation*}
    \lVert \sigma - \sigma _h \rVert _\Omega \lesssim h ^{ s + t -1 } \lvert z \rvert _{ s, \Omega } .
  \end{equation*}
  Similarly, if $ k = n - 1 $, then
  $ \rho _h \in \mathfrak{B} ^n (\Omega) $ satisfies
  \begin{equation*}
    \lVert \rho - \rho _h \rVert _\Omega \lesssim h ^{ s + t -1 } \lvert z \rvert _{ s, \Omega } .
  \end{equation*}
\end{corollary}

\begin{proof}
  Since $\delta$ is trivial on $0$-forms, we have the Hodge
  decomposition
  $ L ^2 \Lambda ^0 (\Omega) = \mathring{ \mathfrak{B} } ^\ast _0
  (\Omega) \oplus \mathfrak{H} ^0 (\Omega) $. If $ k = 1 $, it
  therefore suffices to show that
  $ \sigma _h \in \mathfrak{H} ^0 (\Omega) ^\perp $. Harmonic
  $0$-forms are constant on each connected component of $\Omega$, and
  thus contained in $ W _h ^0 $, so the claim follows by taking
  $ \tau _h \in \mathfrak{H} ^0 (\Omega) $ in
  \eqref{e:laplace_dg_tau}. The $ k = n -1 $ case is immediate since
  $ L ^2 \Lambda ^n (\Omega) = \mathfrak{B} ^n (\Omega) $.
\end{proof}

\begin{remark}
  When $ k = 1 $ and $ n = 2 $, we have improved estimates for both
  auxiliary variables,
  \begin{equation*}
    \lVert \sigma - \sigma _h \rVert _\Omega + \lVert u - u _h \rVert _\Omega + \lVert \rho - \rho _h \rVert _\Omega \lesssim h ^{ s + t -1 } \lvert z \rvert _{ s, \Omega } ,
  \end{equation*}
  which sharpens the estimates in \citet{CoNuSa2026} for the
  two-dimensional vector Laplacian.
\end{remark}

Finally, we obtain estimates for the approximate traces, akin to those
in \cref{t:dirac_trace_error}.

\begin{theorem}
  Under the hypotheses of \cref{t:laplace_error}, the approximate
  traces satisfy
  \begin{align*}
    \lVERT u ^{\mathrm{tan}} - \widehat{ u } _h ^{\mathrm{tan}} \rVERT _{ \partial \mathcal{T} _h } + \lVERT u ^{\mathrm{nor}} - \widehat{ u } _h ^{\mathrm{nor}} \rVERT _{ \partial \mathcal{T} _h }
    &\lesssim \hphantom{\Biggl\{} h ^{ s + t - 1 } \lvert z \rvert _{ s, \Omega } ,\\
    \lVERT \sigma ^{\mathrm{tan}} - \widehat{ \sigma } _h ^{\mathrm{tan}} \rVERT _{ \partial \mathcal{T} _h }
    &\lesssim
      \begin{cases}
        h ^{ s + t -1 } \lvert z \rvert _{ s, \Omega } ,&  k = 1 ,\\
        h ^{ s - 1/2 } \lvert z \rvert _{ s, \Omega } ,&  k \neq 1 ,
      \end{cases}\\
    \lVERT \rho ^{\mathrm{nor}} - \widehat{ \rho } _h ^{\mathrm{nor}} \rVERT _{ \partial \mathcal{T} _h }
    &\lesssim
      \begin{cases}
        h ^{ s + t -1 } \lvert z \rvert _{ s, \Omega } ,&  k = n - 1 ,\\
        h ^{ s - 1/2 } \lvert z \rvert _{ s, \Omega } ,&  k \neq n - 1  .
      \end{cases}
  \end{align*}
\end{theorem}

\begin{proof}
  We start by writing
  \begin{align*}
    \widehat{ e } _u ^{\mathrm{tan}}
    &= \avg{ e _u ^{\mathrm{tan}} } + \frac{ 1 }{ \avg{\alpha ^k} } \jump{ e _\rho ^{\mathrm{nor}} } + \frac{ \jump{\alpha ^k} }{ \avg{ \alpha ^k } } \jump{ e _u ^{\mathrm{tan}} } , \\
    \widehat{ e } _u ^{\mathrm{nor}}
    &= \avg{ e _u ^{\mathrm{nor}} } + \frac{ 1 }{ \avg{\beta ^{ k -1 }} } \jump{ e _\sigma ^{\mathrm{tan}} } + \frac{ \jump{\beta ^{ k -1 }} }{ \avg{ \beta ^{ k -1 }} } \jump{ e _u ^{\mathrm{nor}} } .
  \end{align*}
  As in \cref{t:dirac_error_n=2}, we use
  $ \biglVERT \jump{ \cdot } \bigrVERT _{ \partial \mathcal{T} _h }
  \lesssim h ^{ 1/2 } \lvert \cdot \rvert _{\mathcal{A}} $ for the
  jumps and \cref{l:pythag_triple_norm} for the averages, thereby
  avoiding the possibly suboptimal $ L ^2 $-norm estimates for
  $ e _\sigma $ and $ e _\rho $. Doing this gives
  \begin{equation*}
    \lVERT \widehat{ e } _u ^{\mathrm{tan}} \rVERT _{ \partial \mathcal{T} _h } +     \lVERT \widehat{ e } _u ^{\mathrm{nor}} \rVERT _{ \partial \mathcal{T} _h } \lesssim \lVERT e _u ^{\mathrm{tan}} \rVERT _{ \partial \mathcal{T} _h } + \lVERT e _u ^{\mathrm{nor}} \rVERT _{ \partial \mathcal{T} _h } + h ^{ 1/2 } \lvert e _z \rvert _{\mathcal{A}} .
  \end{equation*}
  Writing $ e _u = \xi _u + \Pi e _u $, the trace inequality and
  Bramble--Hilbert give
  \begin{equation*}
    \lVERT \xi _u ^{\mathrm{tan}} \rVERT _{ \partial \mathcal{T} _h } + \lVERT \xi _u ^{\mathrm{nor}} \rVERT _{ \partial \mathcal{T} _h } \lesssim h ^s \lvert z \rvert _{ s, \Omega } ,
  \end{equation*}
  while a scaling argument and \eqref{e:laplace_error_l2} give
  \begin{equation*}
    \lVERT \Pi e _u ^{\mathrm{tan}} \rVERT _{ \partial \mathcal{T} _h } + \lVERT \Pi e _u ^{\mathrm{nor}} \rVERT _{ \partial \mathcal{T} _h } \lesssim \lVert \Pi e _u \rVert _{ \mathcal{T} _h } \leq \lVert e _u \rVert _\Omega \lesssim h ^{ s + t -1 } \lvert z \rvert _{ s , \Omega } ,
  \end{equation*}
  and \eqref{e:laplace_error_A_C} gives
  \begin{equation*}
    h ^{ 1/2 } \lvert e _z \rvert _{\mathcal{A}} \lesssim h ^s \lvert z \rvert _{ s, \Omega } .
  \end{equation*}
  Combining these, and recalling that $ t \leq 1 $ causes the
  $ \lVert e _u \rVert _\Omega $ term to dominate, proves the first
  estimate. The other two estimates are proved similarly; the
  respective errors are dominated by
  $ \lVert e _\sigma \rVert _\Omega $ and
  $ \lVert e _\rho \rVert _\Omega $, each estimated either by
  \eqref{e:laplace_error_A_C} or by \cref{c:B_Bstar}, depending on
  $k$.
\end{proof}

\begin{remark}
  We close this section by stating the best-case and worst-case rates
  of convergence indicated by our results, as regularity allows:
  \begin{itemize}
  \item \emph{maximum regularity}: If $\Omega$ is $1$-regular (e.g.,
    convex) and $ z \in H ^{ r + 1 } \Lambda (\Omega) $, then
    \begin{align*}
      \lvert z - z _h \rvert _{\mathcal{A}} + \lvert z - z _h \rvert _{\mathcal{C}} &\lesssim h ^{ r + 1/2 } \lvert z \rvert _{ r+1, \Omega } , \\
      \lVert \sigma - P _{ \mathring{ \mathfrak{B} } ^\ast } \sigma _h \rVert _\Omega + \lVert u - u _h \rVert _\Omega + \lVert \rho - P _{ \mathfrak{B} } \rho _h \rVert _\Omega &\lesssim h ^{ r + 1 } \lvert z \rvert _{ r + 1, \Omega } ,
    \end{align*}
    and the approximate traces satisfy
    \begin{align*}
      \lVERT u ^{\mathrm{tan}} - \widehat{ u } _h ^{\mathrm{tan}} \rVERT _{ \partial \mathcal{T} _h } + \lVERT u ^{\mathrm{nor}} - \widehat{ u } _h ^{\mathrm{nor}} \rVERT _{ \partial \mathcal{T} _h }
      &\lesssim \hphantom{\Biggl\{} h ^{ r+1 } \lvert z \rvert _{ r+1, \Omega } ,\\
      \lVERT \sigma ^{\mathrm{tan}} - \widehat{ \sigma } _h ^{\mathrm{tan}} \rVERT _{ \partial \mathcal{T} _h }
      &\lesssim
        \begin{cases}
          h ^{ r+1 } \lvert z \rvert _{ r+1, \Omega } ,&  k = 1 ,\\
          h ^{ r+1/2 } \lvert z \rvert _{ r+1, \Omega } ,&  k \neq 1 ,
        \end{cases}\\
      \lVERT \rho ^{\mathrm{nor}} - \widehat{ \rho } _h ^{\mathrm{nor}} \rVERT _{ \partial \mathcal{T} _h }
      &\lesssim
        \begin{cases}
          h ^{ r+1 } \lvert z \rvert _{ r+1, \Omega } ,&  k = n - 1 ,\\
          h ^{ r+1/2 } \lvert z \rvert _{ r+1, \Omega } ,&  k \neq n - 1  .
        \end{cases}
    \end{align*}
    That is, $ u _h $ and its approximate traces converge with optimal
    order $ r + 1 $, while the auxiliary variables and traces converge
    with order $ r + \frac{1}{2} $ or $ r + 1 $, depending on $k$. For
    the $1$-form Hodge--Laplace problem in dimension $2$, all fields
    converge with optimal order $ r + 1 $.

  \item \emph{minimum regularity}: If $\Omega$ is $t$-regular with
    $ \frac{1}{2} < t \leq 1 $ and we merely have
    $ z \in H ^t \Lambda (\Omega) $, then
    \begin{align*}
      \lvert z - z _h \rvert _{\mathcal{A}} + \lvert z - z _h \rvert _{\mathcal{C}} &\lesssim h ^{ t - 1/2 } \lvert z \rvert _{ t, \Omega } , \\
      \lVert \sigma - P _{ \mathring{ \mathfrak{B} } ^\ast } \sigma _h \rVert _\Omega + \lVert u - u _h \rVert _\Omega + \lVert \rho - P _{ \mathfrak{B} } \rho _h \rVert _\Omega &\lesssim h ^{ 2 t - 1 } \lvert z \rvert _{ t, \Omega } ,
    \end{align*}
    and the approximate traces satisfy
    \begin{align*}
      \lVERT u ^{\mathrm{tan}} - \widehat{ u } _h ^{\mathrm{tan}} \rVERT _{ \partial \mathcal{T} _h } + \lVERT u ^{\mathrm{nor}} - \widehat{ u } _h ^{\mathrm{nor}} \rVERT _{ \partial \mathcal{T} _h }
      &\lesssim \hphantom{\Biggl\{} h ^{ 2t - 1 } \lvert z \rvert _{ t, \Omega } ,\\
      \lVERT \sigma ^{\mathrm{tan}} - \widehat{ \sigma } _h ^{\mathrm{tan}} \rVERT _{ \partial \mathcal{T} _h }
      &\lesssim
        \begin{cases}
          h ^{ 2t - 1 } \lvert z \rvert _{ t, \Omega } ,&  k = 1 ,\\
          h ^{ t-1/2 } \lvert z \rvert _{ t, \Omega } ,&  k \neq 1 ,
        \end{cases}\\
      \lVERT \rho ^{\mathrm{nor}} - \widehat{ \rho } _h ^{\mathrm{nor}} \rVERT _{ \partial \mathcal{T} _h }
      &\lesssim
        \begin{cases}
          h ^{ 2t - 1} \lvert z \rvert _{ t, \Omega } ,&  k = n - 1 ,\\
          h ^{ t-1/2 } \lvert z \rvert _{ t, \Omega } ,&  k \neq n - 1  .
        \end{cases}
    \end{align*}
    For the $1$-form Hodge--Laplace problem in dimension $2$, all
    fields converge with order $ 2 t -1 $.

  \end{itemize}
\end{remark}

\section{An alternative hybridization for the Hodge--Laplace problem}
\label{s:laplace_uhat}

We now discuss an alternative hybridization of the Hodge--Laplace HDG
method that takes the global trace variables to be
$ \widehat{ u } _h ^{\mathrm{nor}} $ and
$ \widehat{ u } _h ^{\mathrm{tan}} $ rather than
$ \widehat{ \sigma } _h ^{\mathrm{tan}} $ and
$ \widehat{ u } _h ^{\mathrm{tan}} $. As mentioned in
\cref{r:laplace_alternative_hybridization}, this alternative
hybridization is equivalent to the one presented in
\cref{s:laplace}---in the sense that it gives the same numerical
solutions---but it has some possible advantages from an implementation
perspective. Namely, it dispenses with the need for local harmonic
forms, and static condensation yields a symmetric positive-definite
Schur complement (modulo global harmonics) rather than a saddle-point
system.

\subsection{Weak formulation}

To formulate this alternative hybridization, we first define
\begin{equation*}
  \widehat{ W } _h ^{k-1, \mathrm{nor}} \coloneqq \prod _{ e \subset \partial \mathcal{T} _h } \mathcal{P} _r \Lambda ^{ k -1 } (e) , \qquad \ringhat{V} _h ^{k-1, \mathrm{nor}} \coloneqq \bigl\{ \widehat{ v } _h ^{\mathrm{nor}} \in \widehat{ W } _h ^{k-1, \mathrm{nor}} : \jump{ \widehat{ v } _h ^{\mathrm{nor}}}  = 0 \bigr\} ,
\end{equation*}
where the ``ring'' over the latter space indicates that these traces
vanish on $ \partial \Omega $ by definition of
$ \jump{ \cdot ^{\mathrm{nor}} } $. We then consider the weak problem:
Find
\begin{alignat*}{4}
  \text{(local variables)} \qquad \sigma _h &\in W _h ^{ k -1 } , \qquad & u _h &\in W _h ^k, \qquad & \rho _h &\in W _h ^{ k + 1 }, \\
  \text{(global variables)} \qquad p _h &\in \mathfrak{H} _h ^k, \qquad & \widehat{ u } _h ^{\mathrm{nor}} &\in \ringhat{ V } _h ^{ k -1, \mathrm{nor} } , \qquad & \widehat{ u } _h ^{\mathrm{tan}} &\in \widehat{ V } _h ^{k, \mathrm{tan}} ,
\end{alignat*}
satisfying
\begin{subequations}
  \label{e:laplace_uhat_ns}
  \begin{alignat}{2}
    - ( \sigma _h , \tau _h ) _{ \mathcal{T} _h } + ( u _h , \mathrm{d} \tau _h ) _{ \mathcal{T} _h } - \langle \widehat{ u } _h ^{\mathrm{nor}} , \tau _h ^{\mathrm{tan}} \rangle _{ \partial \mathcal{T} _h } &= 0, \quad &\forall \tau _h &\in W _h ^{ k -1 } ,\\
    ( \sigma _h , \delta v _h ) _{ \mathcal{T} _h } + ( \rho _h , \mathrm{d} v _h ) _{ \mathcal{T} _h } + ( p _h , v _h ) _{ \mathcal{T} _h } \\
    {}+ \langle \widehat{ \sigma } _h ^{\mathrm{tan}} , v _h ^{\mathrm{nor}} \rangle _{ \partial \mathcal{T} _h } - \langle \widehat{ \rho } _h ^{\mathrm{nor}} , v _h ^{\mathrm{tan}} \rangle _{ \partial \mathcal{T} _h } &= (f, v _h ) _{ \mathcal{T} _h } , \quad &\forall v _h &\in W _h ^k , \notag \\
    - ( \rho _h , \eta _h ) _{ \mathcal{T} _h } + ( u _h , \delta \eta _h ) _{ \mathcal{T} _h } + \langle \widehat{ u } _h ^{\mathrm{tan}} , \eta _h ^{\mathrm{nor}} \rangle _{ \partial \mathcal{T} _h } &= 0, \quad &\forall \eta _h &\in W _h ^{ k + 1 } ,\\
    ( u _h , q _h ) _{ \mathcal{T} _h } &= 0, \quad &\forall q _h &\in \mathfrak{H} _h ^k ,\\
    -\langle \widehat{ \sigma } _h ^{\mathrm{tan}} , \widehat{ v } _h ^{\mathrm{nor}} \rangle _{ \partial \mathcal{T} _h } &= 0 , \quad &\forall \widehat{ v } _h ^{\mathrm{nor}} &\in \ringhat{ V } _h ^{k-1, \mathrm{nor}} ,\\
    \langle \widehat{ \rho } _h ^{\mathrm{nor}} , \widehat{ v } _h ^{\mathrm{tan}} \rangle _{ \partial \mathcal{T} _h } &= 0 , \quad &\forall \widehat{ v } _h ^{\mathrm{tan}} &\in \widehat{ V } _h ^{k, \mathrm{tan}} ,
  \end{alignat}
\end{subequations}
where
\begin{align*}
  \widehat{ \sigma } _h ^{\mathrm{tan}} &\coloneqq \sigma _h ^{\mathrm{tan}} - \beta ^{ k -1 } ( \widehat{ u } _h ^{\mathrm{nor}} - u _h ^{\mathrm{nor}} ) ,\\
  \widehat{ \rho } _h ^{\mathrm{nor}} &\coloneqq \rho _h ^{\mathrm{nor}} - \alpha ^k ( \widehat{ u } _h ^{\mathrm{tan}} - u _h ^{\mathrm{tan}} ) .
\end{align*}
A solution to \eqref{e:laplace_ns} clearly satisfies
\eqref{e:laplace_uhat_ns}, and vice versa, so this alternative
hybridization yields the same numerical solutions as the previous one.

Substituting the trace expressions into \eqref{e:laplace_uhat_ns} and
integrating by parts gives
\begin{subequations}
  \label{e:laplace_uhat}
  \begin{alignat}{2}
    - ( \sigma _h , \tau _h ) _{ \mathcal{T} _h } + ( u _h , \mathrm{d} \tau _h ) _{ \mathcal{T} _h } - \langle \widehat{ u } _h ^{\mathrm{nor}} , \tau _h ^{\mathrm{tan}} \rangle _{ \partial \mathcal{T} _h } &= 0, \quad &\forall \tau _h &\in W _h ^{ k -1 } , \label{e:laplace_uhat_tau} \\
    ( \mathrm{d} \sigma _h + \delta \rho _h + p _h , v _h ) _{ \mathcal{T} _h } - \bigl\langle \beta ^{ k -1 } ( \widehat{ u } _h ^{\mathrm{nor}} - u _h ^{\mathrm{nor}} ) , v _h ^{\mathrm{nor}} \bigr\rangle _{ \partial \mathcal{T} _h } \label{e:laplace_uhat_v} \\
    {}+ \bigl\langle \alpha ^k ( \widehat{ u } _h ^{\mathrm{tan}} - u _h ^{\mathrm{tan}} ) , v _h ^{\mathrm{tan}} \bigr\rangle _{ \partial \mathcal{T} _h } &= (f, v _h ) _{ \mathcal{T} _h } , \quad &\forall v _h &\in W _h ^k , \notag \\
    - ( \rho _h , \eta _h ) _{ \mathcal{T} _h } + ( u _h , \delta \eta _h ) _{ \mathcal{T} _h } + \langle \widehat{ u } _h ^{\mathrm{tan}} , \eta _h ^{\mathrm{nor}} \rangle _{ \partial \mathcal{T} _h } &= 0, \quad &\forall \eta _h &\in W _h ^{ k + 1 } , \label{e:laplace_uhat_eta} \\
    ( u _h , q _h ) _{ \mathcal{T} _h } &= 0, \quad &\forall q _h &\in \mathfrak{H} _h ^k , \label{e:laplace_uhat_q} \\
    - \bigl\langle \sigma _h ^{\mathrm{tan}} - \beta ^{ k -1 } ( \widehat{ u } _h ^{\mathrm{nor}} - u _h ^{\mathrm{nor}} ), \widehat{ v } _h ^{\mathrm{nor}} \bigr\rangle _{ \partial \mathcal{T} _h } &= 0 , \quad &\forall \widehat{ v } _h ^{\mathrm{nor}} &\in \ringhat{ V } _h ^{k-1, \mathrm{nor}} , \label{e:laplace_uhat_vnor} \\
    \bigl\langle \rho _h ^{\mathrm{nor}} - \alpha ^k ( \widehat{ u } _h ^{\mathrm{tan}} - u _h ^{\mathrm{tan}} ) , \widehat{ v } _h ^{\mathrm{tan}} \bigr\rangle _{ \partial \mathcal{T} _h } &= 0 , \quad &\forall \widehat{ v } _h ^{\mathrm{tan}} &\in \widehat{ V } _h ^{k, \mathrm{tan}} , \label{e:laplace_uhat_vtan}
  \end{alignat}
\end{subequations}
which is equivalent to \eqref{e:laplace_uhat_ns} but more obviously
symmetric.

We now show that this system, which does not include local harmonics,
nevertheless satisfies local existence and uniqueness, allowing us to
efficiently eliminate the local variables. (Global existence and
uniqueness is immediate from \cref{t:laplace_global_eu} and the
equivalence of the formulations.)

\begin{theorem}[local existence and uniqueness]
  \label{t:laplace_uhat_existence_uniqueness}
  Given any $ p _h \in \mathfrak{H} _h ^k $,
  $ \widehat{ u } _h ^{\mathrm{nor}} \in \ringhat{V} _h ^{k-1,
    \mathrm{nor}} $,
  $ \widehat{ u } _h ^{\mathrm{tan}} \in \widehat{ V } _h ^{k,
    \mathrm{tan}} $, and $ f \in L ^2 \Lambda ^k (\Omega) $, there
  exist unique $ \sigma _h \in W _h ^{ k -1 } $, $ u _h \in W _h ^k $,
  and $ \rho _h \in W _h ^{ k + 1 } $ satisfying
  \eqref{e:laplace_uhat_tau}--\eqref{e:laplace_uhat_eta}.
\end{theorem}

\begin{proof}
  Since this is a square system, it suffices to show that
  \begin{subequations}
    \begin{alignat}{2}
      - ( \sigma _h , \tau _h ) _{ \mathcal{T} _h } + ( u _h , \mathrm{d} \tau _h ) _{ \mathcal{T} _h } &= 0, \quad &\forall \tau _h &\in W _h ^{ k -1 } , \label{e:laplace_uhat_local_tau} \\
      ( \mathrm{d} \sigma _h + \delta \rho _h , v _h ) _{ \mathcal{T} _h } + \langle \beta ^{ k -1 } u _h ^{\mathrm{nor}} , v _h ^{\mathrm{nor}} \rangle _{ \partial \mathcal{T} _h } - \langle \alpha ^k u _h ^{\mathrm{tan}} , v _h ^{\mathrm{tan}} \rangle _{ \partial \mathcal{T} _h } &= 0 , \quad &\forall v _h &\in W _h ^k , \label{e:laplace_uhat_local_v} \\
      - ( \rho _h , \eta _h ) _{ \mathcal{T} _h } + ( u _h , \delta \eta _h ) _{ \mathcal{T} _h } &= 0, \quad &\forall \eta _h &\in W _h ^{ k + 1 } , \label{e:laplace_uhat_local_eta}
    \end{alignat}
  \end{subequations}
  has only the trivial solution. Taking $ \tau _h = - \sigma _h $ in
  \eqref{e:laplace_uhat_local_tau}, $ v _h = u _h $ in
  \eqref{e:laplace_uhat_local_v}, and $ \eta _h = - \rho _h $ in
  \eqref{e:laplace_uhat_local_eta}, summing and canceling terms gives
  \begin{equation*}
    ( \sigma _h , \sigma _h ) _{ \mathcal{T} _h } + ( \rho _h , \rho _h ) _{ \mathcal{T} _h } + \langle \beta ^{ k -1 } u _h ^{\mathrm{nor}} , u _h ^{\mathrm{nor}} \rangle _{ \partial \mathcal{T} _h } - \langle \alpha ^k u _h ^{\mathrm{tan}} , u _h ^{\mathrm{tan}} \rangle _{ \partial \mathcal{T} _h } = 0 .
  \end{equation*}
  Thus, \cref{a:laplace_penalty} implies that $ \sigma _h = 0 $,
  $ \rho _h = 0 $, $ u _h ^{\mathrm{nor}} = 0 $, and
  $ u _h ^{\mathrm{tan}} = 0 $. Integrating
  \eqref{e:laplace_uhat_local_tau} by parts with
  $ \tau _h = \delta u _h $ implies $ \delta u _h = 0 $, and
  integrating \eqref{e:laplace_uhat_local_eta} by parts with
  $ \eta _h = \mathrm{d} u _h $ implies $ \mathrm{d} u _h = 0 $. It
  follows that $ L u _h = 0 $ with $ u _h ^{\mathrm{nor}} = 0 $ and
  $ u _h ^{\mathrm{tan}} = 0 $. In other words, we have
  $ \Delta u _h = 0 $ in $ \mathcal{T} _h $ and $ u _h = 0 $ on
  $ \partial \mathcal{T} _h $, and therefore $ u _h = 0 $.
\end{proof}

\subsection{Static condensation}

We now examine the form of the Schur complement system obtained via
static condensation. In addition to eliminating the local variables
using \cref{t:laplace_uhat_existence_uniqueness}, we may also
separately compute $ p _h = P _{ \mathfrak{H} _h } f $ using
\cref{l:laplace_properties} to obtain a condensed system involving
only the global trace variables. Denote the combined trace space
$ \widehat{ V } _h ^k \coloneqq \ringhat{ V } _h ^{k-1, \mathrm{nor}}
\oplus \widehat{ V } _h ^{k, \mathrm{tan}} $. Let us write
\begin{equation*}
  \sigma _h = \sigma _{\widehat{ u }} + \sigma _f , \qquad u _h = u _{\widehat{ u }} + u _f , \qquad \rho _h = \rho _{\widehat{ u }} + \rho _f ,
\end{equation*}
where
$ \sigma _{\widehat{u}} , u _{\widehat{u}} , \rho _{\widehat{u}} $ is
the solution to \eqref{e:laplace_uhat_tau}--\eqref{e:laplace_uhat_eta}
for given $ \widehat{ u } _h \in \widehat{ V } _h ^k $ with
$ p _h = 0 $ and $ f = 0 $, and where $ \sigma _f , u _f , \rho _f $
is the solution to
\eqref{e:laplace_uhat_tau}--\eqref{e:laplace_uhat_eta} for given $f$
with $ p _h = P _{ \mathfrak{H} _h } f $ and $ \widehat{ u } _h = 0
$. Substituting these into
\eqref{e:laplace_uhat_vnor}--\eqref{e:laplace_uhat_vtan}, we obtain
the system
\begin{equation}
  \label{e:ahat_Fhat}
  \widehat{ a } ( \widehat{ u } _h , \widehat{ v } _h ) = \widehat{ F } ( \widehat{ v } _h ) , \quad \forall \widehat{ v } _h \in \widehat{ V } _h ^k ,
\end{equation}
where
\begin{subequations}
  \begin{align}
    \widehat{ a } ( \widehat{ u } _h , \widehat{ v } _h )
    &\coloneqq
      - \bigl\langle \sigma _{\widehat{u}} ^{\mathrm{tan}} - \beta ^{ k -1 } ( \widehat{ u } _h ^{\mathrm{nor}} - u _{\widehat{u}} ^{\mathrm{nor}} ), \widehat{ v } _h ^{\mathrm{nor}} \bigr\rangle _{ \partial \mathcal{T} _h } + \bigl\langle \rho _{\widehat{u}} ^{\mathrm{nor}} - \alpha ^k ( \widehat{ u } _h ^{\mathrm{tan}} - u _{\widehat{u}} ^{\mathrm{tan}} ) , \widehat{ v } _h ^{\mathrm{tan}} \bigr\rangle _{ \partial \mathcal{T} _h } , \label{e:ahat_ns} \\
    \widehat{ F } ( \widehat{ v } _h )
    &\coloneqq \langle \sigma _f ^{\mathrm{tan}} + \beta ^{ k -1 } u _f ^{\mathrm{nor}} , \widehat{ v } _h ^{\mathrm{nor}} \rangle _{ \partial \mathcal{T} _h } - \langle \rho _f ^{\mathrm{nor}} + \alpha ^k u _f ^{\mathrm{tan}} , \widehat{ v } _h ^{\mathrm{tan}} \rangle _{ \partial \mathcal{T} _h } . \label{e:Fhat}
  \end{align}
\end{subequations}
The system \eqref{e:ahat_Fhat} does not yet account for
\eqref{e:laplace_uhat_q}, and consequently it has a nullspace
corresponding to traces of global harmonics, as we will show. Before
characterizing $ \widehat{ a } ( \cdot , \cdot ) $, we first show that
$ \widehat{ F } $ vanishes on these global harmonic traces, which will
later allow us to quotient them out.

\begin{lemma}
  \label{l:Fhat}
  For all $ q _h \in \mathfrak{H} _h ^k $, we have
  $ \widehat{ F } ( q _h ^{\mathrm{nor}} \oplus q _h ^{\mathrm{tan}} )
  = 0 $.
\end{lemma}

\begin{proof}
  Taking $ \widehat{ u } _h = 0 $ and
  $ p _h = P _{ \mathfrak{H} _h } f $ in \eqref{e:laplace_uhat_v}
  implies
  \begin{equation*}
    ( \mathrm{d} \sigma _f + \delta \rho _f , v _h ) _{ \mathcal{T} _h } + \langle \beta ^{ k -1 } u _f ^{\mathrm{nor}} , v _h ^{\mathrm{nor}} \rangle _{ \partial \mathcal{T} _h } - \langle \alpha ^k u _f ^{\mathrm{tan}} , v _h ^{\mathrm{tan}} \rangle _{ \partial \mathcal{T} _h } = ( f - P _{ \mathfrak{H} _h } f , v _h ) _{ \mathcal{T} _h } , \quad \forall v _h \in W _h ^k .
  \end{equation*}
  If $ v _h = q _h \in \mathfrak{H} _h ^k $, the right-hand side
  vanishes, and integrating by parts on the left-hand side gives
  \begin{equation*}
    \langle \sigma _f ^{\mathrm{tan}} + \beta ^{ k -1 } u _f ^{\mathrm{nor}} , q _h ^{\mathrm{nor}} \rangle _{ \partial \mathcal{T} _h } - \langle \rho _f ^{\mathrm{nor}} + \alpha ^k u _f ^{\mathrm{tan}} , q _h ^{\mathrm{tan}} \rangle _{ \partial \mathcal{T} _h } = 0 , \quad \forall q _h \in \mathfrak{H} _h ^k ,
  \end{equation*}
  as claimed.
\end{proof}

Next, we express $ \widehat{ a } ( \cdot , \cdot ) $ in a form that
better reveals its symmetry and other properties. Compare with the
condensed bilinear form for the original LDG-H method in
\citet[p.~1341]{CoGoLa2009}.

\begin{lemma}
  \label{l:ahat}
  The condensed bilinear form may be written as
  \begin{align*}
    \widehat{ a } ( \widehat{ u } _h , \widehat{ v } _h ) =
    ( \sigma _{ \widehat{ u } } , \tau _{ \widehat{ v } } ) _{ \mathcal{T} _h } + ( \rho _{ \widehat{ u } } , \eta _{ \widehat{ v } } ) _{ \mathcal{T} _h }
    &+ \bigl\langle \beta ^{ k -1 } ( \widehat{ u } _h ^{\mathrm{nor}} - u _{ \widehat{ u } } ^{\mathrm{nor}} ) , \widehat{ v } _h ^{\mathrm{nor}} - v _{ \widehat{ v } } ^{\mathrm{nor}} \bigr\rangle _{ \partial \mathcal{T} _h } \\
    &- \bigl\langle \alpha ^k ( \widehat{ u } _h ^{\mathrm{tan}} - u _{ \widehat{ u } } ^{\mathrm{tan}} ) , \widehat{ v } _h ^{\mathrm{tan}} - v _{ \widehat{ v } }  ^{\mathrm{tan}} \bigr\rangle _{ \partial \mathcal{T} _h },
  \end{align*}
  where $ \tau _{ \widehat{ v } } $, $ v _{ \widehat{ v } } $,
  $ \eta _{ \widehat{ v } } $ are obtained by applying the local
  solvers to the test function $ \widehat{ v } _h $.
\end{lemma}

\begin{proof}
  By \eqref{e:laplace_uhat_tau}--\eqref{e:laplace_uhat_eta} applied to
  $ \tau _{ \widehat{ v } } $, $ v _{ \widehat{ v } } $,
  $ \eta _{ \widehat{ v } } $, we have
  \begin{subequations}
    \label{e:laplace_vhat}
    \begin{alignat}{2}
      - ( \tau _{ \widehat{ v } } , \tau _h ) _{ \mathcal{T} _h } + ( v _{ \widehat{ v } } , \mathrm{d} \tau _h ) _{ \mathcal{T} _h } - \langle \widehat{ v } _h ^{\mathrm{nor}} , \tau _h ^{\mathrm{tan}} \rangle _{ \partial \mathcal{T} _h } &= 0, \quad &\forall \tau _h &\in W _h ^{ k -1 } , \label{e:laplace_vhat_tau} \\
      ( \mathrm{d} \tau _{ \widehat{ v } } + \delta \eta _{ \widehat{ v } } , v _h ) _{ \mathcal{T} _h } - \bigl\langle \beta ^{ k -1 } ( \widehat{ v } _h ^{\mathrm{nor}} - v _{ \widehat{ v } } ^{\mathrm{nor}} ) , v _h ^{\mathrm{nor}} \bigr\rangle _{ \partial \mathcal{T} _h }  \label{e:laplace_vhat_v} \\
      + \bigl\langle \alpha ^k ( \widehat{ v } _h ^{\mathrm{tan}} - v _{ \widehat{ v } } ^{\mathrm{tan}} ) , v _h ^{\mathrm{tan}} \bigr\rangle _{ \partial \mathcal{T} _h } &= 0 , \quad &\forall v _h &\in W _h ^k , \notag \\
      - ( \eta _{ \widehat{ v } } , \eta _h ) _{ \mathcal{T} _h } + ( v _{ \widehat{ v } } , \delta \eta _h ) _{ \mathcal{T} _h } + \langle \widehat{ v } _h ^{\mathrm{tan}} , \eta _h ^{\mathrm{nor}} \rangle _{ \partial \mathcal{T} _h } &= 0, \quad &\forall \eta _h &\in W _h ^{ k + 1 } . \label{e:laplace_vhat_eta}
    \end{alignat}
  \end{subequations}
  Taking $ \tau _h = \sigma _{ \widehat{ u } } $ in
  \eqref{e:laplace_vhat_tau} and $ \eta _h = \rho _{ \widehat{ u } } $
  in \eqref{e:laplace_vhat_eta}, adding and rearranging gives
  \begin{align*}
    - \langle \sigma _{ \widehat{ u } } ^{\mathrm{tan}} , \widehat{ v } _h ^{\mathrm{nor}} \rangle _{ \partial \mathcal{T} _h } + \langle \rho _{ \widehat{ u } } ^{\mathrm{nor}} , \widehat{ v } _h ^{\mathrm{tan}} \rangle _{ \partial \mathcal{T} _h }
    &= ( \sigma _{ \widehat{ u } } , \tau _{ \widehat{ v } } ) _{ \mathcal{T} _h } + ( \rho _{ \widehat{ u } } , \eta _{ \widehat{ v } } ) _{ \mathcal{T} _h } - ( \mathrm{d} \sigma _{ \widehat{ u } } + \delta \rho _{ \widehat{ u } } , v _{ \widehat{ v } } ) _{ \mathcal{T} _h } \\
    &=( \sigma _{ \widehat{ u } } , \tau _{ \widehat{ v } } ) _{ \mathcal{T} _h } + ( \rho _{ \widehat{ u } } , \eta _{ \widehat{ v } } ) _{ \mathcal{T} _h }
      \begin{aligned}[t]
        &- \bigl\langle \beta ^{ k -1 } ( \widehat{ u } _h ^{\mathrm{nor}} - u _{ \widehat{ u } } ^{\mathrm{nor}} ) , v _{ \widehat{ v } } ^{\mathrm{nor}} \bigr\rangle _{ \partial \mathcal{T} _h } \\
        &+ \bigl\langle \alpha ^k ( \widehat{ u } _h ^{\mathrm{tan}} - u _{ \widehat{ u } } ^{\mathrm{tan}} ) , v _{ \widehat{ v } } ^{\mathrm{tan}} \bigr\rangle _{ \partial \mathcal{T} _h } ,
      \end{aligned}
  \end{align*}
  where the last equality holds since
  $ \sigma _{ \widehat{ u } } , u _{ \widehat{ u } } , \rho _{
    \widehat{ u } } $ satisfy \eqref{e:laplace_uhat_v} with
  $ p _h = 0 $, $ f = 0 $, and $ v _h = v _{ \widehat{ v } }
  $. Substituting into \eqref{e:ahat_ns} completes the proof.
\end{proof}

Denote the equivalence class of
$ \widehat{ v } _h \in \widehat{ V } _h ^k $, modulo global harmonics,
as
\begin{equation*}
  \widehat{ v } _h + \mathfrak{H} _h ^k \coloneqq \bigl\{ ( \widehat{ v } _h ^{\mathrm{nor}} + q _h ^{\mathrm{nor}} ) \oplus ( \widehat{ v } _h ^{\mathrm{tan}} + q _h ^{\mathrm{tan}} ) : q _h \in \mathfrak{H} _h ^k \bigr\} ,
\end{equation*}
and denote the quotient space
$ \widehat{ V } _h ^k / \mathfrak{H} _h ^k $. We are now ready to
state the main result of this section.

\begin{theorem}
  The condensed bilinear form $ \widehat{ a } ( \cdot , \cdot ) $ is
  symmetric positive-semidefinite on $ \widehat{ V } _h ^k $ with
  nullspace $ \mathfrak{H} _h ^k $, and thus is symmetric
  positive-definite on $ \widehat{ V } _h ^k / \mathfrak{H} _h ^k
  $. Consequently, there exists a unique equivalence class
  $ \widehat{u} _h + \mathfrak{H} _h ^k \in \widehat{ V } _h ^k /
  \mathfrak{H} _h ^k $ satisfying \eqref{e:ahat_Fhat}, and furthermore
  there is a unique $ \widehat{ u } _h \in \widehat{ V } _h ^k $ in
  this equivalence class such that
  $ u _h \in \mathfrak{H} _h ^{k \perp} $.
\end{theorem}

\begin{proof}
  From \cref{a:laplace_penalty,l:ahat}, we immediately have
  $ \widehat{ a } ( \widehat{ v } _h , \widehat{ v } _h ) \geq 0 $ for
  all $ \widehat{ v } _h \in \widehat{ V } _h ^k $, i.e.,
  $ \widehat{ a } ( \cdot , \cdot ) $ is positive-semidefinite. This
  equals zero if and only if
  \begin{equation*}
    \tau _{ \widehat{ v } } = 0 , \qquad \eta _{
      \widehat{ v } } = 0 , \qquad v _{ \widehat{ v } }
    ^{\mathrm{nor}} = \widehat{ v } _h ^{\mathrm{nor}} , \qquad v _{
      \widehat{ v } } ^{\mathrm{tan}} = \widehat{ v } _h
    ^{\mathrm{tan}} ;
  \end{equation*}
  the last two equalities imply
  $ v _{ \widehat{ v } } \in H \Lambda ^k (\Omega) \cap \mathring{ H }
  ^\ast \Lambda ^k (\Omega) $. Taking
  $ \tau _h = \delta v _{ \widehat{ v } } $ and
  $ \eta _h = \mathrm{d} v _{ \widehat{ v } } $ in
  \eqref{e:laplace_vhat}, and integrating by parts, implies
  $ \delta v _{ \widehat{ v } } = 0 $ and
  $ \mathrm{d} v _{ \widehat{ v } } = 0 $. Thus,
  $ v _{ \widehat{ v } } \in \mathfrak{H} _h ^k $, proving the first
  claim.

  Next, \cref{l:Fhat} implies that $ \widehat{ F } $ is a well-defined
  linear functional on the quotient
  $ \widehat{ V } _h ^k / \mathfrak{H} _h ^k $. Since
  $ \widehat{ a } ( \cdot , \cdot ) $ is symmetric positive-definite
  on $ \widehat{ V } _h ^k / \mathfrak{H} _h ^k $, it follows that
  there exists a unique equivalence class
  $ \widehat{ u } _h + \mathfrak{H} _h ^k $ satisfying
  \eqref{e:ahat_Fhat}.

  Finally, observe from \eqref{e:laplace_vhat} that adding the same
  harmonic $ q _h \in \mathfrak{H} _h ^k $ to both
  $ \widehat{ v } _h $ and $ v _{ \widehat{ v } } $ yields another
  solution, i.e.,
  $ \widehat{ v } _h + q _h \coloneqq ( \widehat{ v } _h
  ^{\mathrm{nor}} + q _h ^{\mathrm{nor}} ) \oplus ( \widehat{ v } _h
  ^{\mathrm{tan}} + q _h ^{\mathrm{tan}} ) $ has the local
  solution
  $ \tau _{ \widehat{ v } } , v _{ \widehat{ v } } + q _h , \eta _{
    \widehat{ v } } $. Indeed, the $ q _h $ terms cancel in
  \eqref{e:laplace_vhat_v} by simple subtraction, while they cancel in
  \eqref{e:laplace_vhat_tau} and \eqref{e:laplace_vhat_eta} due to
  integration by parts. Hence, if
  $ \widehat{ u } _h ^\prime \in \widehat{ V } _h ^k $ is a particular
  solution to \eqref{e:ahat_Fhat} with
  $ u _h ^\prime = u _{ \widehat{ u } ^\prime } + u _f $, then
  $ \widehat{ u } _h = \widehat{ u } _h ^\prime - P _{ \mathfrak{H} _h
  } u _h ^\prime $ uniquely gives
  $ u _h = u _h ^\prime - P _{ \mathfrak{H} _h } u _h ^\prime \in
  \mathfrak{H} _h ^{ k \perp } $.
\end{proof}

This proof gives a practical recipe for solving \eqref{e:laplace_uhat}
using static condensation:
\begin{itemize}
\item Compute $ p _h = P _{ \mathfrak{H} _h } f $.
\item Use the local solvers
  \eqref{e:laplace_uhat_tau}--\eqref{e:laplace_uhat_eta} to assemble
  the condensed system \eqref{e:ahat_Fhat}.
\item Compute a particular solution $ \widehat{ u } _h ^\prime $ to
  \eqref{e:ahat_Fhat}. (When there are nontrivial harmonics, one can
  do this in Firedrake/PETSc, for example, by explicitly supplying a
  basis for the nullspace $ \mathfrak{H} _h ^k $.)
\item Use the local solvers to compute $ \sigma _h $,
  $ u _h ^\prime $, $ \rho _h $.
\item Compute $ P _{ \mathfrak{H} _h } u _h ^\prime $, and subtract to
  get
  $ \widehat{ u } _h = \widehat{ u } _h ^\prime - P _{ \mathfrak{H} _h
  } u _h ^\prime $ and
  $ u _h = u _h ^\prime - P _{ \mathfrak{H} _h } u _h ^\prime $.
\end{itemize}
For contractible $\Omega$ with $ k > 0 $, we recall that
$ \mathfrak{H} _h ^k $ is trivial, so the first and last steps may be
omitted, and the condensed system \eqref{e:ahat_Fhat} is symmetric
positive-definite (rather than -semidefinite) on
$ \widehat{ V } _h ^k $.

\begin{remark}
  We assemble a possibly singular condensed system for
  \eqref{e:laplace_uhat_vnor}--\eqref{e:laplace_uhat_vtan}, using its
  nullspace to obtain a solution that also satisfies
  \eqref{e:laplace_uhat_q}. Alternatively, one could assemble a Schur
  complement for
  \eqref{e:laplace_uhat_q}--\eqref{e:laplace_uhat_vtan}, which would
  give a nonsingular \emph{mixed} system for both $ p _h $ and
  $ \widehat{ u } _h $, rather than a positive-semidefinite one for
  $ \widehat{ u } _h $ alone. From this perspective, our preferred
  approach might be seen as solving the mixed system using an
  ``ellipticity on the kernel''-style argument; cf.~\citep[Section
  3.4.5]{BoBrFo2013}.
\end{remark}

\section{Numerical experiments}
\label{s:numerical_experiments}

This section presents numerical experiments for the Hodge--Dirac and
Hodge--Laplace problems in dimensions $ n = 2 $ and $ n = 3 $,
illustrating the convergence behavior of the equal-order HDG
methods. All computations have been carried out using the Firedrake
finite element library
\citep{RaHaMiLaLuMcBeMaKe2017,FiredrakeUserManual} (version 2026.4.0),
and a Firedrake component called Slate \citep{GiMiHaCo2020} was used
to implement the local solvers for static condensation. Differential
forms and operators are implemented using the corresponding objects
from vector calculus, as \cref{r:proxies} illustrates for $ n = 3
$; for $ n = 2 $, we likewise identify
\begin{equation*}
  \begin{tikzcd}
    \Lambda ^0 (\Omega) \ar[r, shift left, "\mathrm{d}"] \ar[d,
    phantom, sloped, "\cong"] & \Lambda ^1 (\Omega) \ar[l, shift left,
    "\delta"] \ar[r, shift left, "\mathrm{d}"] \ar[d, phantom, sloped,
    "\cong"] & \Lambda ^2 (\Omega) \ar[l, shift left, "\delta"] \ar[d, phantom, sloped, "\cong"] \\
    C ^\infty (\Omega) \ar[r, shift left, "\operatorname{grad}"] & C
    ^\infty ( \Omega , \mathbb{R}^2 ) \ar[l, shift left,
    "-\operatorname{div}"] \ar[r, shift left, "\operatorname{curl}"] &
    C ^\infty ( \Omega ) \ar[l, shift left,
    "\operatorname{curl}"] \mathrlap{\quad .}
  \end{tikzcd}
\end{equation*}

\subsection{Hodge--Dirac}

For all of the Hodge--Dirac numerical experiments, we take the HDG
penalty coefficients to be $ \alpha ^k = ( - 1 ) ^k $.

\begin{table}
  \centering
  \small
  \begin{filecontents*}{dirac_square.csv}
r,N,z0_error,z1_error,z2_error,A_error,z0_rate,z1_rate,z2_rate,A_rate
0,1,0.5000011387595962,1.1762509491503388,0.7669970143285231,1.3612278223956877,,,,
0,2,0.5520687353070441,0.6226675739517034,0.5904449510472906,1.8501846891659877,-0.1429165202547935,0.9176618306360426,0.3774184026846562,-0.442760746420954
0,4,0.30175832715759304,0.34570487087084234,0.3355239484201043,1.5245853809320413,0.8714543175519208,0.8489212205029555,0.8153868149372944,0.27925234319744563
0,8,0.15441441874915965,0.18479798713362533,0.16767132133926527,1.1259921936634976,0.9665861099775492,0.9035918882137051,1.0007797940875038,0.4372201219976142
0,16,0.07770443449518341,0.09537720621808203,0.08211857395797387,0.8055603249803863,0.9907386344725397,0.9542326134908992,1.0298554706378233,0.48313229008809755
0,32,0.038920382855429245,0.04838094647119885,0.04041469521105787,0.5713543314654025,0.9974710319717106,0.9792055311568264,1.0228286067712595,0.4956069044555957
0,64,0.019468838252093404,0.02435367435336253,0.020020635479922607,0.40432522249374686,0.9993591039407628,0.9902995379663198,1.0133921997230937,0.4988695216236696
1,1,0.3697556957576622,0.5764418859946756,0.36023840051057354,1.449300184483036,,,,
1,2,0.12352744393414,0.15786723185257595,0.12895826438912328,0.6901002010176385,1.581740771798414,1.8684634245689158,1.4820477448105145,1.0704786839152618
1,4,0.03548537456914084,0.04255333334752237,0.0341804462008279,0.27814006025880805,1.7995351590971125,1.891367692630881,1.9156610967237486,1.3109943035193248
1,8,0.00915897506352585,0.01119262318944016,0.00840266628706217,0.1024688443275054,1.9539664674582848,1.9267239504811036,2.024252136899096,1.440626225009628
1,16,0.002301875041145521,0.002873394548519257,0.0020567362266685333,0.03667898898401789,1.992376644437635,1.9617201876690444,2.030490406203609,1.482159546397296
1,32,0.0005753442832417437,0.0007276894300904211,0.0005070637370149605,0.013017316838680924,2.000312096594163,1.9813613439023623,2.0201177742998766,1.494521765877706
1,64,0.00014371677798347035,0.0001830695382722864,0.00012577645886783332,0.004608127450834493,2.0011970194729223,1.99093110040255,2.011305179454478,1.4981795830336795
2,1,0.14862639068604508,0.19295912121705566,0.2280536506617159,0.6837961116462349,,,,
2,2,0.024883172179431482,0.03551304048941339,0.030506045947121578,0.1784389022022546,2.5784479879429894,2.441874453241533,2.902206163234797,1.9381360415656887
2,4,0.0032449441429395864,0.0047068937711277695,0.0034394515310323713,0.03302505199530901,2.938904866543906,2.9155016849925914,3.1488447662153094,2.4337974394602977
2,8,0.0004052641636611325,0.0005786062377588233,0.0003985904655681712,0.005779578316537759,3.0012591317276005,3.024121511847836,3.1091994193320467,2.514524691588938
2,16,5.0373224542637495e-05,7.092212413171319e-05,4.85162465311562e-05,0.0010119300405595868,3.008133617328623,3.028274225560793,3.03836735528294,2.51385468293344
2,32,6.274703860663497e-06,8.763209546045113e-06,6.0175364751278756e-06,0.00017800554490208338,3.0050378085177734,3.016704483720334,3.0112230536991813,2.5071154660543424
2,64,7.829187495968292e-07,1.088847085495267e-06,7.504970178691608e-07,3.1393612114869306e-05,3.002612872426777,3.0086579966517704,3.0032547372198657,2.5033792435746456
\end{filecontents*}
\pgfplotstabletypeset[
  col sep = comma,
  every head row/.style = {
    before row=\toprule,
    after row=\midrule
  },
  every last row/.style = { after row=\bottomrule },
  every nth row = {7}{ before row=\midrule }, 
  every odd column/.style = {
    fixed,
    fixed zerofill,
    column type = r
  },
  every even column/.style = {
    sci,
    sci e,
    sci zerofill
  },
  columns={
    r,
    N,
    A_error,
    A_rate,
    z0_error,
    z0_rate,
    z1_error,
    z1_rate,
    z2_error,
    z2_rate
  },
  columns/r/.style = {
    column name = $r$,
    fixed,
    fixed zerofill,
    precision = 0,
    column type = r,
    postproc cell content/.code = {
      \pgfmathtruncatemacro{\tmp}{mod(\pgfplotstablerow, 7)}
      \ifnum\tmp=0\relax
      \else
        \pgfkeyssetvalue{/pgfplots/table/@cell content}{}
      \fi
    },
  },
  columns/N/.style = {
    column name = $N$,
    precision = 0,
    column type = r,
  },
  columns/{A_error}/.style={
    column name=$ \lvert z - z _h \rvert _{ \mathcal{A} } $
  },
  columns/{A_rate}/.style={
    column name=rate
  },
  columns/{z0_error}/.style={
    column name=$ \lVert z ^0 - z _h ^0  \rVert _\Omega  $
  },
  columns/{z0_rate}/.style={
    column name=rate
  },
  columns/{z1_error}/.style={
    column name=$ \lVert z ^1 - z _h ^1  \rVert _\Omega  $
  },
  columns/{z1_rate}/.style={
    column name=rate
  },
  columns/{z2_error}/.style={
    column name=$ \lVert z ^2 - z _h ^2  \rVert _\Omega  $
  },
  columns/{z2_rate}/.style={
    column name=rate
  }
  ]{dirac_square.csv}
  \medskip
  \caption{Hodge--Dirac on unit square, smooth solution. \label{t:dirac_square}}
\end{table}

We first consider the unit square $ \Omega = ( 0 , 1 ) ^2 $, taking
$ \mathcal{T} _h $ to be the structured triangle mesh obtained by
dividing $\Omega$ into an $ N \times N $ square grid and further
subdividing each smaller square into two
triangles. \Cref{t:dirac_square} shows the result of applying the
equal-order HDG methods to the Hodge--Dirac problem whose exact
solution is
\begin{equation*}
  z ^0 = \cos ( \pi x _1 ) \cos ( \pi x _2 ) , \qquad z ^1 =
  \begin{bmatrix}
    \sin ( \pi x _1 ) \\
    \sin ( \pi x _2 )
  \end{bmatrix} +
  \begin{bmatrix}
    \hphantom{-}\sin ( \pi x _1 ) \cos ( \pi x _2 ) \\
    - \cos ( \pi x _1 ) \sin ( \pi x _2 )
  \end{bmatrix}, \qquad z ^2 = \sin ( \pi x _1 ) \sin ( \pi x _2 ) ,
\end{equation*}
noting that the two terms of $ z ^1 $ are in
$ \mathfrak{B} ^1 (\Omega) $ and
$ \mathring{ \mathfrak{B} } ^\ast _1 (\Omega) $ respectively. Since
$\Omega$ is convex and thus $1$-regular, and $z$ is smooth, we observe
the expected convergence rates of $ r + \frac{1}{2} $ in the
$\mathcal{A}$ seminorm and $ r + 1 $ in the $ L ^2 $ norm. These match
the estimates in \cref{t:dirac_error} with $ s = r + 1 $ and
$ t = 1 $; since $ t = 1 $, there is no improvement from
\cref{t:dirac_error_n=2}.

\begin{table}
  \centering
  \small
  \begin{filecontents*}{dirac_lsd.csv}
r,N,z0_error,z1_error,z2_error,A_error,z0_rate,z1_rate,z2_rate,A_rate
0,1,0.3318650587527897,0.19281082277401496,0.14396749127757755,0.6556533960690101,,,,
0,2,0.2688072366211462,0.13143306747875722,0.1422781999791369,0.6199692119376846,0.3040247618740717,0.5528577430427104,0.017028451764078856,0.08073677941780055
0,4,0.22519067204497736,0.08602097565755908,0.1234724270828021,0.5741675569637441,0.255424909106611,0.6115678918708585,0.2045257213262599,0.11072475800973192
0,8,0.18673176055867335,0.05540568232029785,0.1021510301275187,0.5239350195220304,0.27017973721685884,0.634654550137046,0.2734851533675551,0.13208391996763813
0,16,0.15273609138029323,0.035375383107826903,0.08275305085028363,0.47375177638617405,0.28992632167549637,0.6472881723062169,0.30381934705850755,0.14525654131183674
0,32,0.1236845893171233,0.022469822319252823,0.06633626885898554,0.4259764970830715,0.3043752531907926,0.6547570556959847,0.3190146318528466,0.15335751975598555
0,64,0.09948623947356428,0.01422785398375547,0.05288858709355766,0.38169214808084084,0.31409685963463163,0.6592706427744535,0.32684143414544353,0.1583643243992987
1,1,0.1986358164530196,0.10936458991253936,0.1373070032571236,0.5355746819370489,,,,
1,2,0.16049850888773784,0.06886919669697,0.10032626417165798,0.4810789973923336,0.3075658877364442,0.6672149432852494,0.45270587604926016,0.1548139455476959
1,4,0.1272390108488488,0.04338562459745362,0.07467017945802082,0.42901669883864335,0.3350188324655985,0.6666417503172206,0.42609523161593454,0.16524001304983993
1,8,0.10081039758357055,0.027327208877518484,0.05668571784303361,0.3822684548581104,0.33589661551168604,0.6668789842836942,0.39754690924361835,0.1664476510096824
1,16,0.07991881322026673,0.017213542676214146,0.043656630478752685,0.34057041528663123,0.335037381421484,0.666794069267179,0.37678450505616734,0.16663303632606907
1,32,0.06338494473716175,0.010843417446074583,0.03396650463628911,0.30341484516288664,0.3343949490151612,0.666724532664781,0.3620880204350172,0.16666144226690371
1,64,0.050285608539374686,0.0068308106276707475,0.026611755056148694,0.2703120497704407,0.33399464268613865,0.666690811078495,0.35204910377311815,0.16666584928202832
2,1,0.1563458952805405,0.07412573451867528,0.10004315634800086,0.4724484383435557,,,,
2,2,0.12271151917493034,0.04670379351626471,0.07382780760438805,0.42064053371229104,0.34947065909517105,0.6664347579077909,0.4383862582866307,0.16756900878636638
2,4,0.09705243382394575,0.02941250744701396,0.05558441534127927,0.374737296763074,0.3384343864528949,0.6671099568900623,0.40948387738733505,0.16670830455726682
2,8,0.0768848017123908,0.018526117235768972,0.042527568494417925,0.3338524831802047,0.3360659524075116,0.6668692324268397,0.3862820678865937,0.16666880204267098
2,16,0.06095277436944834,0.011670068051880185,0.032934173774863566,0.2974287391561752,0.3350065509796081,0.6667475744181051,0.36881301406325484,0.1666667159339592
2,32,0.048343142999218915,0.007351518421365898,0.025721537156818425,0.2649788846331674,0.33438061699573485,0.6666988062059414,0.356608494630132,0.16666665479297404
2,64,0.03835246515892202,0.00463112544894491,0.020203259999314382,0.23606934917363148,0.3339919600967747,0.6666794249535046,0.34838875769152367,0.16666666349070125
\end{filecontents*}
\pgfplotstabletypeset[
  col sep = comma,
  every head row/.style = {
    before row=\toprule,
    after row=\midrule
  },
  every last row/.style = { after row=\bottomrule },
  every nth row = {7}{ before row=\midrule }, 
  every odd column/.style = {
    fixed,
    fixed zerofill,
    column type = r
  },
  every even column/.style = {
    sci,
    sci e,
    sci zerofill
  },
  columns={
    r,
    N,
    A_error,
    A_rate,
    z0_error,
    z0_rate,
    z1_error,
    z1_rate,
    z2_error,
    z2_rate
  },
  columns/r/.style = {
    column name = $r$,
    fixed,
    fixed zerofill,
    precision = 0,
    column type = r,
    postproc cell content/.code = {
      \pgfmathtruncatemacro{\tmp}{mod(\pgfplotstablerow, 7)}
      \ifnum\tmp=0\relax
      \else
        \pgfkeyssetvalue{/pgfplots/table/@cell content}{}
      \fi
    },
  },
  columns/N/.style = {
    column name = $N$,
    precision = 0,
    column type = r,
  },
  columns/{A_error}/.style={
    column name=$ \lvert z - z _h \rvert _{ \mathcal{A} } $
  },
  columns/{A_rate}/.style={
    column name=rate
  },
  columns/{z0_error}/.style={
    column name=$ \lVert z ^0 - z _h ^0  \rVert _\Omega  $
  },
  columns/{z0_rate}/.style={
    column name=rate
  },
  columns/{z1_error}/.style={
    column name=$ \lVert z ^1 - z _h ^1  \rVert _\Omega  $
  },
  columns/{z1_rate}/.style={
    column name=rate
  },
  columns/{z2_error}/.style={
    column name=$ \lVert z ^2 - z _h ^2  \rVert _\Omega  $
  },
  columns/{z2_rate}/.style={
    column name=rate
  }
  ]{dirac_lsd.csv}
  \medskip
  \caption{Hodge--Dirac on L-shaped domain, minimum-regularity
    solution. \label{t:dirac_lsd}}
\end{table}

Next, we consider the L-shaped domain
$ \Omega = ( - \frac{1}{2} , \frac{1}{2} ) ^2 \setminus ( -
\frac{1}{2} , 0] ^2 $, which removes the third quadrant from the unit
square centered at the origin. Since this domain has a reentrant
corner at the origin with angle $ \frac{ 3 \pi }{ 2 } $, it is merely
$t$-regular for all $ t < \frac{ 2 }{ 3 } $. We take
$ \mathcal{T} _h $ to be the structured triangle mesh corresponding to
a $ 2 N \times 2 N $ grid on the original unit square.
\Cref{t:dirac_lsd} shows the result of applying the equal-order HDG
methods to the Hodge--Dirac problem whose exact solution in polar
coordinates is
\begin{equation*}
  z ^0 = 0 , \qquad z ^1 = \operatorname{grad} \Bigl( \lvert x \rvert ^{ 2 / 3 } \cos \bigl[  \tfrac{2}{3} ( \theta + \tfrac{ \pi }{ 2 } ) \bigr]  \Bigr) , \qquad z ^2 = 0 ,
\end{equation*}
which has minimum regularity $ H ^s \Lambda (\Omega) $ for all
$ s < \frac{ 2 }{ 3 } $. For all $r$, we observe the expected
minimum-regularity convergence rates of
$ \frac{ 2 }{ 3 } - \frac{1}{2} = \frac{ 1 }{ 6 } $ in the
$\mathcal{A}$ seminorm and
$ \frac{ 2 }{ 3 } + \frac{ 2 }{ 3 } - 1 = \frac{ 1 }{ 3 } $ in the
$ L ^2 $ norm, matching the estimates in
\cref{t:dirac_error}. Furthermore, we observe the improved
$ L ^2 $-norm convergence rate of $ \frac{ 2 }{ 3 } $ for $ z ^1 $,
matching the estimate in \cref{t:dirac_error_n=2}.

\begin{remark}
  \label{r:lsd_bc}
  Note that this minimum-regularity solution satisfies the normal
  boundary conditions on the two radial segments adjacent to the
  reentrant corner ($ \theta = - \frac{ \pi }{ 2 } $ and
  $ \theta = \pi $) but not all of $ \partial \Omega $. However, since
  the solution is smooth away from the reentrant corner, a bump
  function/smooth extension argument can be used to relate this
  solution and its properties to one satisfying the boundary
  conditions on all of $ \partial \Omega $; cf.~\citet[Remark
  4.1]{BaCaSt2024}.
\end{remark}

\begin{table}
  \centering
  \small
  \begin{filecontents*}{dirac_cube.csv}
r,N,z0_error,z1_error,z2_error,z3_error,A_error,z0_rate,z1_rate,z2_rate,z3_rate,A_rate
0,1,0.3542883954880806,1.2595604373477636,0.8317766711706803,0.46400111466982386,1.8260144745626863,,,,,
0,2,0.5021624942810154,0.6568536396528648,0.6645246120085676,0.4367863357294757,2.2786379317608865,-0.5032300684456611,0.939276498729301,0.3238735886928932,0.0872005469450643,-0.3194735018014718
0,4,0.29367185570885584,0.36554695755843125,0.3689696852753076,0.2762030640014414,1.9056605679143193,0.7739492696465394,0.8455152002617742,0.84882034496547,0.6611984002749347,0.25788053076143613
0,8,0.15241645492319153,0.1965739717300813,0.19088936546819227,0.14365305743406176,1.4192899784292092,0.9461863453853956,0.8949844365722321,0.9507655568490614,0.9431406251841459,0.42512179256437344
0,16,0.07659570322356878,0.10167674157230046,0.09656905061136135,0.07101749706740702,1.0180435537394648,0.9926832960915393,0.9510826054866843,0.9831039324847639,1.0163422785514924,0.4793700960560494
1,1,0.3342291300659368,0.5072962709895207,0.5638580097838061,0.28564089919606445,1.7015352485692075,,,,,
1,2,0.11054692315199736,0.15970890164361606,0.18622736832973877,0.1210401332050758,0.898613589889088,1.5961786071320982,1.6673838295915688,1.5982668004967726,1.2387170845922804,0.9210642528479703
1,4,0.03186294612773487,0.04959042070216808,0.04974856846649838,0.0317157140336826,0.3760093900554047,1.7947072979036736,1.6873113563463473,1.9043381924149392,1.9322157525061638,1.2569321894814267
1,8,0.008227983546856162,0.013733872696184493,0.012781182265472759,0.00788339647826323,0.1417710507337418,1.9532688541410872,1.8523229685889042,1.9606337197301933,2.008308584059951,1.4072057224580194
1,16,0.0020647308321666735,0.0035814849127981997,0.0032535182141410655,0.0019378756072303599,0.05130230429895239,1.9945851914466664,1.9391087262275848,1.9739487569457461,2.0243413668349692,1.4664674355924823
2,1,0.1260979093900864,0.2238572970357071,0.2621993973172071,0.1835764599271902,0.9254800552271308,,,,,
2,2,0.02460455273122786,0.04426924148870959,0.04598476267517994,0.035453103308477854,0.2652040616175809,2.357547160963811,2.3382027867218116,2.5114365695552534,2.3723972521252588,1.8030990287075463
2,4,0.003268276265286603,0.005896261413888547,0.006273405013485729,0.004192948072882096,0.051720087470509754,2.912323447140226,2.908432260825858,2.873835283039723,3.0798750475861505,2.358306250398276
2,8,0.00041131238925666834,0.0006959974131536556,0.0007919130906344665,0.00047112691186013757,0.009112111475868481,2.990223505289899,3.0826466394395196,2.9858346926477566,3.1537773138384795,2.50486741314979
2,16,5.125742954776196e-05,8.361989613952526e-05,9.836626751665863e-05,5.654668203255617e-05,0.001596556318755735,3.0044014908821737,3.057163787361583,3.009106542950134,3.0586014650799314,2.5128219517329713
\end{filecontents*}
\pgfplotstabletypeset[
  col sep = comma,
  every head row/.style = {
    before row=\toprule,
    after row=\midrule
  },
  every last row/.style = { after row=\bottomrule },
  every nth row = {5}{ before row=\midrule }, 
  every odd column/.style = {
    fixed,
    fixed zerofill,
    column type = r
  },
  every even column/.style = {
    sci,
    sci e,
    sci zerofill
  },
  columns={
    r,
    N,
    A_error,
    A_rate,
    z0_error,
    z0_rate,
    z1_error,
    z1_rate,
    z2_error,
    z2_rate,
    z3_error,
    z3_rate
  },
  columns/r/.style = {
    column name = $r$,
    fixed,
    fixed zerofill,
    precision = 0,
    column type = r,
    postproc cell content/.code = {
      \pgfmathtruncatemacro{\tmp}{mod(\pgfplotstablerow, 5)}
      \ifnum\tmp=0\relax
      \else
        \pgfkeyssetvalue{/pgfplots/table/@cell content}{}
      \fi
    },
  },
  columns/N/.style = {
    column name = $N$,
    precision = 0,
    column type = r,
  },
  columns/{A_error}/.style={
    column name=$ \lvert z - z _h \rvert _{ \mathcal{A} } $
  },
  columns/{A_rate}/.style={
    column name=rate
  },
  columns/{z0_error}/.style={
    column name=$ \lVert z ^0 - z _h ^0  \rVert _\Omega  $
  },
  columns/{z0_rate}/.style={
    column name=rate
  },
  columns/{z1_error}/.style={
    column name=$ \lVert z ^1 - z _h ^1  \rVert _\Omega  $
  },
  columns/{z1_rate}/.style={
    column name=rate
  },
  columns/{z2_error}/.style={
    column name=$ \lVert z ^2 - z _h ^2  \rVert _\Omega  $
  },
  columns/{z2_rate}/.style={
    column name=rate
  },
  columns/{z3_error}/.style={
    column name=$ \lVert z ^3 - z _h ^3  \rVert _\Omega  $
  },
  columns/{z3_rate}/.style={
    column name=rate
  }
  ]{dirac_cube.csv}
  \medskip
  \caption{Hodge--Dirac on unit cube, smooth
    solution. \label{t:dirac_cube}}
\end{table}

Finally, we consider the unit cube $ \Omega = ( 0, 1 ) ^3 $, taking
$ \mathcal{T} _h $ to be a structured tetrahedral mesh, obtained by
dividing $\Omega$ into an $ N \times N \times N $ cubical grid and
further subdividing each smaller cube into six
tetrahedra. \Cref{t:dirac_cube} shows the result of applying the equal-order HDG methods to the Hodge--Dirac problem whose exact solution is
\begin{alignat*}{2}
  z ^0
  &= \mathrlap{\cos ( \pi x _1 ) \cos ( \pi x _2 ) \cos ( \pi x _3 ) ,}\\
  z ^1
  &=
    \begin{bmatrix}
      \sin ( \pi x _1 )\\
      \sin ( \pi x _2 )\\
      \sin ( \pi x _3 )
    \end{bmatrix}
    &&+
    \begin{bmatrix}
      \hphantom{-} \sin ( \pi x _1 ) \cos ( \pi x _2 ) \\
      - \cos ( \pi x _1 ) \sin ( \pi x _2 ) \\
      0
    \end{bmatrix} ,\\
  z ^2
  &=
    \begin{bmatrix}
      \sin ( \pi x _2 ) \sin ( \pi x _3 ) \\
      \sin ( \pi x _1 ) \sin ( \pi x _3 ) \\
      \sin ( \pi x _1 ) \sin ( \pi x _2 )
    \end{bmatrix}
    &&+
    \begin{bmatrix}
      \cos ( \pi x _1 ) \sin ( \pi x _2 ) \sin ( \pi x _3 ) \\
      \sin ( \pi x _1 ) \cos ( \pi x _2 ) \sin ( \pi x _3 ) \\
      \sin ( \pi x _1 ) \sin ( \pi x _2 ) \cos ( \pi x _3 )
    \end{bmatrix}, \\
  z ^3 &= \mathrlap{\sin ( \pi x _1 ) \sin ( \pi x _2 ) \sin ( \pi x _3 ) ,}
\end{alignat*}
where $ z ^1 $ and $ z ^2 $ are each written as a sum of exact and
coexact terms. Again, since $\Omega$ is convex and thus $1$-regular,
and $z$ is smooth, we observe the expected convergence rates of
$ r + \frac{1}{2} $ in the $\mathcal{A}$ seminorm and $ r + 1 $ in the
$ L ^2 $ norm, matching the estimates in \cref{t:dirac_error}.

\subsection{Hodge--Laplace}

We now present numerical experiments for the Hodge--Laplace problem,
using the same domains $\Omega$ just considered for the Hodge--Dirac
problem. Since $ k = 0 $ and $ k = n $ correspond to the scalar
Poisson equation, for which HDG methods have already been well
studied, we focus on the cases corresponding to the vector Poisson
equation. For each domain, we take the $k$-form Hodge--Laplace problem
whose exact solution is $u = z ^k $ from the corresponding
Hodge--Dirac experiment, and we take the HDG penalty coefficients to
be $ \alpha ^{k-1} = 1 $ and $ \alpha ^k = - 1 $.

\begin{table}
  \centering
  \small
  \begin{filecontents*}{laplace_square.csv}
r,N,sigma_error,u_error,rho_error,A_error,sigma_rate,u_rate,rho_rate,A_rate
0,1,3.1415926446411446,3.729007015102228,2.5492617702453697,7.151908846200964,,,,
0,2,1.5699546278345204,2.354873679439213,1.9244663611049828,6.468673712750739,1.0007732599682766,0.6631418385437293,0.4056210700687384,0.14485840553047424
0,4,0.7756003909534835,1.4075666231675643,0.9746236630834861,5.212846788828113,1.0173374296241262,0.7424464628117043,0.9815412980918142,0.31139848521217006
0,8,0.3815124753645369,0.7650072926904302,0.4671173824176596,3.838862628171966,1.0235832970520873,0.8796578038317542,1.0610601185920623,0.44139252322892597
0,16,0.18869912638698622,0.39640802743242737,0.2262191360943824,2.7462663119113944,1.0156424898411134,0.948487322835621,1.0460641583944439,0.4832074011654681
0,32,0.09377748597518455,0.20130280006711337,0.11118251652757856,1.9483799766438699,1.0087742354276727,0.9776189381191835,1.0247910316313047,0.4951964719066986
0,64,0.04673906583903871,0.10136456655012235,0.055116689447726665,1.3791027108403204,1.0046127035677426,0.9898138143394521,1.0123688007113092,0.4985451540653252
1,1,0.648347407264223,1.0963331738011868,1.7468057792705534,4.102111633078228,,,,
1,2,0.28274821371641506,0.47065798969134637,0.5090147282950294,2.227096603953318,1.1972491569563983,1.2199353066810406,1.778939903097404,0.8812026144889615
1,4,0.0703610838662772,0.14319171084526552,0.13267857607130418,0.8966497159747339,2.0066682980439223,1.716731106172169,1.9397719670782725,1.3125477398105907
1,8,0.017067564907253228,0.03899324500095282,0.03288066146807047,0.3308499487010963,2.043520468351481,1.8766518542532569,2.012624206614186,1.4383674379517444
1,16,0.004172174786881968,0.010076024299177437,0.00813626291161478,0.1185163810501796,2.032385734660263,1.9522977138186404,2.0148011186769166,1.481090577105328
1,32,0.0010297510165217252,0.0025527734903716573,0.0020217022852309107,0.04207631690870501,2.018504048258574,1.9807890693143548,2.0087957361584783,1.4940061477710502
1,64,0.00025570732962629835,0.0006418851880897245,0.0005038094071999631,0.014897646840813525,2.009730130494784,1.9916783565633338,2.0046205967614537,1.497923958218907
2,1,0.351744810710066,0.49801954668915643,0.6243139017721783,1.894823108809575,,,,
2,2,0.04212795523517938,0.11666240865910009,0.11046342559744585,0.569298914248742,3.0616793409299152,2.093862601475438,2.4987028182036184,1.7348049185270489
2,4,0.005431060906385458,0.014846698047500826,0.014019889516783002,0.10550097754911161,2.955471945998104,2.974125754151432,2.9780218874407884,2.431929981570457
2,8,0.0006917060403955409,0.001763479971021278,0.0017696569754171315,0.018538062716329677,2.9730030823150306,3.073645011418561,2.9859333351077133,2.508693975738519
2,16,8.732457239818414e-05,0.0002134177872792631,0.00022244277128964556,0.0032573370768365613,2.9856994766508986,3.046672862712761,2.991963618860785,2.508725559823014
2,32,1.0956399766078193e-05,2.6287915906780365e-05,2.7886282546396158e-05,0.0005742577681196811,2.994613861479836,3.0212087469750397,2.9958066860223305,2.503922648383611
2,64,1.371217781870662e-06,3.2649445890489354e-06,3.4909952523183504e-06,0.00010140080401412181,2.9982441831391924,3.0092693579086554,2.997845324990978,2.5016293757925325
\end{filecontents*}
\pgfplotstabletypeset[
  col sep = comma,
  every head row/.style = {
    before row=\toprule,
    after row=\midrule
  },
  every last row/.style = { after row=\bottomrule },
  every nth row = {7}{ before row=\midrule }, 
  every odd column/.style = {
    fixed,
    fixed zerofill,
    column type = r
  },
  every even column/.style = {
    sci,
    sci e,
    sci zerofill
  },
  columns={
    r,
    N,
    A_error,
    A_rate,
    sigma_error,
    sigma_rate,
    u_error,
    u_rate,
    rho_error,
    rho_rate
  },
  columns/r/.style = {
    column name = $r$,
    fixed,
    fixed zerofill,
    precision = 0,
    column type = r,
    postproc cell content/.code = {
      \pgfmathtruncatemacro{\tmp}{mod(\pgfplotstablerow, 7)}
      \ifnum\tmp=0\relax
      \else
        \pgfkeyssetvalue{/pgfplots/table/@cell content}{}
      \fi
    },
  },
  columns/N/.style = {
    column name = $N$,
    precision = 0,
    column type = r,
  },
  columns/{A_error}/.style={
    column name=$ \lvert z - z _h \rvert _{ \mathcal{A} } $
  },
  columns/{A_rate}/.style={
    column name=rate
  },
  columns/{sigma_error}/.style={
    column name=$ \lVert \sigma - \sigma _h  \rVert _\Omega  $
  },
  columns/{sigma_rate}/.style={
    column name=rate
  },
  columns/{u_error}/.style={
    column name=$ \lVert u - u _h \rVert _\Omega  $
  },
  columns/{u_rate}/.style={
    column name=rate
  },
  columns/{rho_error}/.style={
    column name=$ \lVert \rho - \rho _h \rVert _\Omega  $
  },
  columns/{rho_rate}/.style={
    column name=rate
  }
  ]{laplace_square.csv}
  \medskip
  \caption{$1$-form Hodge--Laplace on unit square, smooth
    solution. \label{t:laplace_square}}
\end{table}

We again start by letting $\Omega$ be the unit square, where
$ \mathcal{T} _h $ is an $ N \times N $ structured triangle
mesh. \Cref{t:laplace_square} shows the result of applying the
equal-order HDG methods to the $1$-form Hodge--Laplace problem whose
exact solution is
\begin{equation*}
  u =
  \begin{bmatrix}
    \sin ( \pi x _1 ) \\
    \sin ( \pi x _2 )
  \end{bmatrix} +
  \begin{bmatrix}
    \hphantom{-}\sin ( \pi x _1 ) \cos ( \pi x _2 ) \\
    - \cos ( \pi x _1 ) \sin ( \pi x _2 )
  \end{bmatrix} ,
\end{equation*}
with $ \sigma = - \operatorname{div} u $ and
$ \rho = \operatorname{curl} u $. Since $\Omega$ is $1$-regular and
the exact solution is smooth, we observe the expected convergence
rates of $ r + \frac{1}{2} $ in the $\mathcal{A}$ seminorm and
$ r + 1 $ in the $ L ^2 $ norm for all three components, matching the
estimates in \cref{t:laplace_error,c:B_Bstar}.

\begin{table}
  \centering
  \small
  \begin{filecontents*}{laplace_lsd.csv}
r,N,sigma_error,u_error,rho_error,A_error,sigma_rate,u_rate,rho_rate,A_rate
0,1,0.23313473679103075,0.20776452236023663,0.10747220074288656,0.5207375860263749,,,,
0,2,0.2060892149448913,0.1543874916420488,0.11238384975578176,0.5078242215432475,0.17789497250618158,0.4283934509873255,-0.06447119249334587,0.03622733353357965
0,4,0.18257212917408142,0.11350572577853782,0.10166849035530666,0.48201324497752607,0.17480246244942776,0.44379079540136424,0.14456210571130246,0.07525641840410968
0,8,0.1579659104756944,0.084120838135708,0.08710194182135736,0.45119182097170574,0.20885329131590424,0.4322299463697856,0.22309583341932626,0.09533187466135518
0,16,0.1336207443600088,0.0633253344450626,0.07268687850030317,0.41818568312425863,0.24146925321863888,0.40968043301800394,0.26100993107767634,0.10959724425249795
0,32,0.11113998489881201,0.048506278634367725,0.05972888080514387,0.38448390988468784,0.2657660514172622,0.38461129008916695,0.28326626213054507,0.12122044643630336
0,64,0.09133086041497089,0.037708021768375684,0.04860268166166964,0.3511818095074052,0.283203619486573,0.3633000352367297,0.2973927719472904,0.1307051064908915
1,1,0.16445746421525273,0.11548035290587842,0.11080662762805005,0.4468112045734749,,,,
1,2,0.13880268098767898,0.08072317080574092,0.0858727384343217,0.4170189614557038,0.24467905513718513,0.5165926718195029,0.3677720708811032,0.09955238124644139
1,4,0.1137349558065383,0.058142065381538904,0.06649133081833998,0.3837476715971772,0.28735970729894245,0.47340052334723914,0.3690339458722045,0.11995498622797807
1,8,0.09234356549115572,0.04285570898665023,0.051855216336358735,0.35023455991033287,0.3005923839971314,0.44009491673698253,0.3586771302279619,0.1318365460543982
1,16,0.0745826313686771,0.03225725507145679,0.04072315365853206,0.3178503958909289,0.3081717394776877,0.4098637268295686,0.34863983307612756,0.13997356550431173
1,32,0.060012850920920686,0.02470289060500243,0.03215464777181421,0.28727863226204864,0.3135682311318222,0.3849438092679547,0.340822004076597,0.14589719590806838
1,64,0.04815228869881718,0.019165758016913503,0.02548153313846775,0.2588458563944624,0.31766709535724935,0.36614880973188146,0.3355752036725699,0.15035746572293962
2,1,0.1352377414039325,0.08273812474147511,0.08551830713731458,0.4097644894603222,,,,
2,2,0.10990893097681549,0.05915383242317183,0.06584194304391616,0.37706158241688437,0.29918920567274876,0.4840806192213982,0.37722638334977043,0.11999479814186031
2,4,0.08909899512297835,0.04307999405174385,0.05095096730447834,0.34411346239808976,0.3028275555868474,0.457453589375314,0.3698973751472976,0.13191583218578695
2,8,0.07190031693748689,0.032050621839602146,0.039749521970536654,0.3122327824781822,0.30941103072765763,0.42666570269315685,0.3581720256381331,0.14026231508465248
2,16,0.057814679729028616,0.024307969477039844,0.03123309949272179,0.2821226859910558,0.31456227606672416,0.39892296255167903,0.34786176476853825,0.14629933888971935
2,32,0.04636280116475678,0.018720564572367614,0.02466645590232264,0.25412611683522096,0.3184681197726791,0.3768054412496138,0.34052530296071293,0.15077802983303523
2,64,0.03710138851872865,0.014579943922737823,0.019543857361229662,0.2283691811274645,0.321494553567077,0.36063877339449135,0.3358352034598928,0.15417668116166472
\end{filecontents*}
\pgfplotstabletypeset[
  col sep = comma,
  every head row/.style = {
    before row=\toprule,
    after row=\midrule
  },
  every last row/.style = { after row=\bottomrule },
  every nth row = {7}{ before row=\midrule }, 
  every odd column/.style = {
    fixed,
    fixed zerofill,
    column type = r
  },
  every even column/.style = {
    sci,
    sci e,
    sci zerofill
  },
  columns={
    r,
    N,
    A_error,
    A_rate,
    sigma_error,
    sigma_rate,
    u_error,
    u_rate,
    rho_error,
    rho_rate
  },
  columns/r/.style = {
    column name = $r$,
    fixed,
    fixed zerofill,
    precision = 0,
    column type = r,
    postproc cell content/.code = {
      \pgfmathtruncatemacro{\tmp}{mod(\pgfplotstablerow, 7)}
      \ifnum\tmp=0\relax
      \else
        \pgfkeyssetvalue{/pgfplots/table/@cell content}{}
      \fi
    },
  },
  columns/N/.style = {
    column name = $N$,
    precision = 0,
    column type = r,
  },
  columns/{A_error}/.style={
    column name=$ \lvert z - z _h \rvert _{ \mathcal{A} } $
  },
  columns/{A_rate}/.style={
    column name=rate
  },
  columns/{sigma_error}/.style={
    column name=$ \lVert \sigma - \sigma _h  \rVert _\Omega  $
  },
  columns/{sigma_rate}/.style={
    column name=rate
  },
  columns/{u_error}/.style={
    column name=$ \lVert u - u _h \rVert _\Omega  $
  },
  columns/{u_rate}/.style={
    column name=rate
  },
  columns/{rho_error}/.style={
    column name=$ \lVert \rho - \rho _h \rVert _\Omega  $
  },
  columns/{rho_rate}/.style={
    column name=rate
  }
  ]{laplace_lsd.csv}
  \medskip
  \caption{$1$-form Hodge--Laplace on L-shaped domain,
    minimum-regularity solution. \label{t:laplace_lsd}}
\end{table}

Next, we take $\Omega$ to be the L-shaped domain, with the same family
of structured triangle meshes $ \mathcal{T} _h $ as
before. \Cref{t:laplace_lsd} shows the result of applying the
equal-order HDG methods to the $1$-form Hodge--Laplace problem whose
exact solution is
\begin{equation*}
  u = \operatorname{grad} \Bigl( \lvert x \rvert ^{ 2 / 3 } \cos \bigl[  \tfrac{2}{3} ( \theta + \tfrac{ \pi }{ 2 } ) \bigr]  \Bigr) ,
\end{equation*}
with $ \sigma = - \operatorname{div} u = 0 $ and
$ \rho = \operatorname{curl} u = 0 $. (See \cref{r:lsd_bc} regarding
the boundary conditions.) Recall that $\Omega$ is only $t$-regular for
$ t < \frac{ 2 }{ 3 } $, and $u$ has minimum Sobolev regularity. For
all $r$, we therefore observe the expected convergence rates of
$ \frac{ 2 }{ 3 } - \frac{1}{2} = \frac{ 1 }{ 6 } $ in the
$\mathcal{A}$ seminorm and
$ \frac{ 2 }{ 3 } + \frac{ 2 }{ 3 } - 1 = \frac{ 1 }{ 3 } $ in the
$ L ^2 $ norm for all three components, matching the estimates in
\cref{t:laplace_error,c:B_Bstar}. Unlike in the Hodge--Dirac case when
$ n = 2 $, we do not observe improved $ L ^2 $-norm convergence in the
$1$-form component $u$.

\begin{table}
  \centering
  \small
  \begin{filecontents*}{laplace_cube.csv}
r,N,sigma_error,u_error,rho_error,A_error,sigma_rate,u_rate,rho_rate,A_rate
0,1,2.407737353552305,1.6018889659033986,2.689214464371947,6.101388146478902,,,,
0,2,1.8328643706173235,1.4481083772677983,2.085991446499186,5.968993122384218,0.3935779928137836,0.14560457296705498,0.36645157229710446,0.031649920736452675
0,4,0.9912013752308734,1.0623534810012727,1.0932607989276417,5.179353377321956,0.8868499381828275,0.44690569904452043,0.9320956431409821,0.20471559734571843
0,8,0.5131808293023775,0.6203292839247002,0.522700056631923,3.901893542615579,0.9497109118964856,0.776157742519566,1.0645823783126844,0.4085975774482884
0,16,0.26084986595963083,0.32863356985479814,0.25135638405292005,2.8078273146694848,0.9762475852116929,0.9165543730841829,1.0562489839453795,0.4747202060912512
1,1,1.6479699663436107,1.00490011020503,2.0111681632777345,4.60801181915245,,,,
1,2,0.5288580255576177,0.4937824328324021,0.6669649486328645,2.860664345154916,1.639737569055277,1.025104684608747,1.5923508674519375,0.6877941875962679
1,4,0.14849951535168548,0.1623910197137732,0.1773356828647894,1.2124465127051518,1.8324222533949222,1.6044036578771061,1.911128085186817,1.2384291256232105
1,8,0.03806049380595096,0.04794815546327555,0.044200483844747264,0.46097752545348347,1.9640920378861741,1.7599246313226242,2.0043487920754375,1.395152783776895
1,16,0.00950934775469083,0.012920849229583604,0.010905130450529096,0.16730138049352652,2.0008759843353623,1.8917744216643806,2.019055134255596,1.462247065187325
2,1,0.6686482534533593,0.41949406253612453,0.8188910125954281,2.389699572136005,,,,
2,2,0.12939367602978907,0.1616607945269715,0.18529166328147653,0.9316564194975229,2.3694803628175345,1.3756805477767622,2.1438734817477507,1.3589593426869684
2,4,0.017300799348848178,0.023607155556414797,0.024008974006188316,0.18663507214979547,2.902856507602998,2.7756737175553767,2.9481523142884,2.3195778888058247
2,8,0.0021938907526646784,0.002802835974769544,0.0030395584396114856,0.03315836551742398,2.9792751042731753,3.074264996891846,2.9816400911749836,2.49277532188263
2,16,0.0002763828395829657,0.00033611129694137925,0.00038227718856282934,0.005831388054856357,2.988751739162021,3.0598763800572844,2.9911707381819443,2.5074616570226915
\end{filecontents*}
\pgfplotstabletypeset[
  col sep = comma,
  every head row/.style = {
    before row=\toprule,
    after row=\midrule
  },
  every last row/.style = { after row=\bottomrule },
  every nth row = {5}{ before row=\midrule }, 
  every odd column/.style = {
    fixed,
    fixed zerofill,
    column type = r
  },
  every even column/.style = {
    sci,
    sci e,
    sci zerofill
  },
  columns={
    r,
    N,
    A_error,
    A_rate,
    sigma_error,
    sigma_rate,
    u_error,
    u_rate,
    rho_error,
    rho_rate
  },
  columns/r/.style = {
    column name = $r$,
    fixed,
    fixed zerofill,
    precision = 0,
    column type = r,
    postproc cell content/.code = {
      \pgfmathtruncatemacro{\tmp}{mod(\pgfplotstablerow, 5)}
      \ifnum\tmp=0\relax
      \else
        \pgfkeyssetvalue{/pgfplots/table/@cell content}{}
      \fi
    },
  },
  columns/N/.style = {
    column name = $N$,
    precision = 0,
    column type = r,
  },
  columns/{A_error}/.style={
    column name=$ \lvert z - z _h \rvert _{ \mathcal{A} } $
  },
  columns/{A_rate}/.style={
    column name=rate
  },
  columns/{sigma_error}/.style={
    column name=$ \lVert \sigma - \sigma _h  \rVert _\Omega  $
  },
  columns/{sigma_rate}/.style={
    column name=rate
  },
  columns/{u_error}/.style={
    column name=$ \lVert u - u _h \rVert _\Omega  $
  },
  columns/{u_rate}/.style={
    column name=rate
  },
  columns/{rho_error}/.style={
    column name=$ \lVert \rho - \rho _h \rVert _\Omega  $
  },
  columns/{rho_rate}/.style={
    column name=rate
  }
  ]{laplace_cube.csv}
  \medskip
  \caption{$2$-form Hodge--Laplace on unit cube, smooth
    solution. \label{t:laplace_cube}}
\end{table}

Finally, we take $\Omega$ to be the unit cube with the same family of
$ N \times N \times N $ structured tetrahedral meshes
$ \mathcal{T} _h $ previously considered. \Cref{t:laplace_cube} shows
the result of applying the equal-order HDG methods to the $2$-form
Hodge--Laplace problem whose exact solution is
\begin{equation*}
  u = \begin{bmatrix}
    \sin ( \pi x _2 ) \sin ( \pi x _3 ) \\
    \sin ( \pi x _1 ) \sin ( \pi x _3 ) \\
    \sin ( \pi x _1 ) \sin ( \pi x _2 )
  \end{bmatrix}
  +
  \begin{bmatrix}
    \cos ( \pi x _1 ) \sin ( \pi x _2 ) \sin ( \pi x _3 ) \\
    \sin ( \pi x _1 ) \cos ( \pi x _2 ) \sin ( \pi x _3 ) \\
    \sin ( \pi x _1 ) \sin ( \pi x _2 ) \cos ( \pi x _3 )
  \end{bmatrix} ,
\end{equation*}
with $ \sigma = \operatorname{curl} u $ and
$ \rho = \operatorname{div} u $. We observe the expected convergence
rates of $ r + \frac{1}{2} $ in the $\mathcal{A}$ seminorm and
$ r + 1 $ in the $ L ^2 $ norm for $u$ and (since $ k = n -1 $) for
$\rho$, matching \cref{t:laplace_error,c:B_Bstar}. Interestingly, we
also observe order-$ ( r + 1 ) $ convergence for $\sigma$, which is
better than the order-$ ( r + \frac{1}{2} ) $ estimate in
\cref{t:laplace_error}; since $ k \neq 1 $, \cref{c:B_Bstar} does not
apply.

\section{Concluding remarks}

We conclude by mentioning several possible directions for future
investigation.

First, it would be interesting to investigate the application of the
HDG methods in this paper to domains with arbitrary topology, in which
\cref{a:harmonic,a:harmonic_k} might not necessarily hold. One
possible approach would be to replace $ \mathfrak{H} _h $ by a space
of ``discrete harmonic forms'' used in conforming FEEC, as mentioned
in \cref{r:contractible}. For instance, on a simplicial triangulation
$ \mathcal{T} _h $ of $\Omega$, one might use discrete harmonic forms
taken from the conforming $ \mathcal{P} _{ r + 1 } ^- \Lambda $
complex, which are in $ \mathcal{P} _r \Lambda $ by \citep[Theorem
3.4]{ArFaWi2006}. However, such discrete harmonics are generally not
contained in $ \mathfrak{H} (\Omega) $, so additional analytical
ingredients (e.g., an inf-sup stability analysis) might be needed to
control the associated ``variational crime.''

Next, for the $k$-form Hodge--Laplace problem, it would be interesting
to know whether the convergence rate of $ r + \frac{1}{2} $ for the
auxiliary variables can be sharpened to $ r + 1 $ for general $k$, not
just the cases $ k = 1 $ and $ k = n - 1 $ covered by
\cref{c:B_Bstar}. This is known to be possible at least for $ k = 0 $
and $ k = n $, where optimal-order estimates for the auxiliary
variable in the scalar Poisson equation have been obtained by HDG
projection and $M$-decomposition techniques
\citep{CoGoSa2010,CoQiSh2012,CoFuSa2017,CoFu2017a,CoFu2017b}. It would
be worthwhile to see whether these techniques can be extended to the
Hodge--Laplace problem, and to investigate related questions of
superconvergence and postprocessing.

Finally, it would be interesting to investigate efficient solvers and
preconditioners for the HDG methods introduced here. The alternative
Hodge--Laplace hybridization of \cref{s:laplace_uhat} seems
particularly attractive for such an investigation, due to the
symmetric positive-(semi)definite form of the condensed system, which
might be solved by the conjugate gradient method with an
appropriately devised preconditioner.

\end{document}